\documentclass[final]{svjour2}
\usepackage{textcomp,amsmath,amssymb,graphicx}
\def\ZZ{{\mathbb Z}}
\def\RR{{\mathbb R}}
\def\Sphere{{\mathbb S}}

\def\eps{\varepsilon}
\def\dist{{\rm dist}\,}

\def\divop{{\rm div}}

\def\tueps{{\tilde u_\eps}}
\def\hueps{{\hat u_\eps}}

\def\ueps{u_\eps}
\def\uepsn{u_{\eps,n}}

\def\Ueps{U_\eps}
\def\uepsdef{W_{\eps,d}}
\def\uepsnodef{W_{\eps}}
\def\uhom{W_{*}}
\def\vhom{V_{*}}
\def\err{\mbox{err}}
\def\vieps{v^i_\eps}
\def\viepsdef{v^i_{\eps,d}}
\def\vjeps{v^j_\eps}
\def\vjepsdef{v^j_{\eps,d}}
\def\Veps{V_\eps}

\def\Weps{W_\eps}
\def\homeg{{\omega}}
\def\HOmeg{{\Omega}}
\def\HomO{{\Omega_{0}}}
\def\homegtau{\homeg_\tau^\eps}
\def\Omegtau{\Omega_\tau^\eps}
\def\Omegtaul{\Omega_{\tau}^{\eps}}

\def\homeps{\homeg_\eps}

\def\homeps{\homeg_\eps}
\def\aeps{a_\eps}
\def\aepsdef{a_{\eps,d}}

\def\ceps{\alpha_\eps}
\def\cepsdef{\gamma_1}
\def\csigeps{{c_{\eps}^{\sigma}}}
\def\reps{r_\eps}

\def\Deps{D_\eps}
\def\Ceps{Q_\eps}
\def\Geps{G_\eps}
\def\tbeta{{\tilde \beta}}
\def\fib{\Phi_b}
\def \MMMMI#1{-\mskip -#1mu\int}
\def \moy {\displaystyle\MMMMI{19.2}}
\def \twepsn {\tilde{w}_{\eps,1,n}}
\def \ttwepsn {\tilde{w}_{\eps,2,n}}
\def \tvn {\tilde{v}_{1,n}}
\def \ttvn {\tilde{v}_{2,n}}

\def\bh{{\mathbf{h}}}

\def\DN{{{\mycal R}_{\eps}}}

\DeclareFontFamily{OT1}{rsfs}{}
\DeclareFontShape{OT1}{rsfs}{m}{n}{ <-7> rsfs5 <7-10> rsfs7 <10-> rsfs10}{}
\DeclareMathAlphabet{\mycal}{OT1}{rsfs}{m}{n}

\begin{document}
\title{Interior Regularity Estimates in High Conductivity Homogenization and Application}
\author{Marc Briane \and Yves Capdeboscq \and Luc Nguyen}
\institute{
M.~Briane \at {Institut de Recherche Math\'ematique de Rennes, Universit\'e Europ\'eenne de Bretagne,
\\
Campus de Beaulieu, 35042  Rennes Cedex FRANCE}
\\\email{mbriane@insa-rennes.fr}
\and
Y.~Capdeboscq\at {Mathematical Institute, 24-29 St Giles', Oxford OX1 3LB, UK} 
\\\email{capdeboscq@maths.ox.ac.uk} 
\and
L.~Nguyen \at {Department of Mathematics, Fine Hall, Washington Road, Princeton, NJ 08544-1000 USA}
\\\email{llnguyen@math.princeton.edu}
}
\titlerunning{Regularity for High Conductivity Homogenization}
\authorrunning{M. Briane \and Y. Capdeboscq \and L. Nguyen}
\maketitle
\begin{abstract}
In this paper, uniform pointwise regularity estimates for the solutions of conductivity equations are obtained 
in a unit conductivity medium reinforced by a $\eps$-periodic lattice of highly conducting thin rods. 
The estimates are derived only at a distance $\eps^{1+\tau}$ (for some $\tau>0$) away from the fibres. 
This distance constraint is rather sharp since the gradients of the solutions are shown to be unbounded 
locally in $L^p$ as soon as $p>2$. One key ingredient is the derivation in dimension two of 
regularity estimates to the solutions of the equations deduced from a Fourier series expansion with 
respect to the fibres direction, and weighted by the high-contrast conductivity. 
The dependence on powers of $\eps$ of these two-dimensional estimates is shown to be sharp. 
The initial motivation for this work comes from imaging, and enhanced resolution phenomena 
observed experimentally in the presence of micro-structures \cite{LEROSEY-ET-AL-07}. We use these regularity estimates
to characterize the signature of low volume fraction heterogeneities in the fibred
reinforced medium assuming that the heterogeneities stay at a distance $\eps^{1+\tau}$ away from the fibres. 
\end{abstract}
\section{\label{sec:intro}Introduction}
Consider a material contained in $\Omega$ ${}=$ $\homeg \times (-L,L)\, \subset \, \RR^3$ where $\homeg$ 
is a bounded domain in $\RR^2$ with smooth boundary $\partial\homeg$. Given some fixed $\omega_0 \Subset \homeg$, 
we assume that inside $\HomO=\omega_0\times(-L,L)$, the material contains small cylindrical rods of high conductivity. 
\begin{figure}
\begin{center}
\includegraphics[width=0.8\columnwidth]{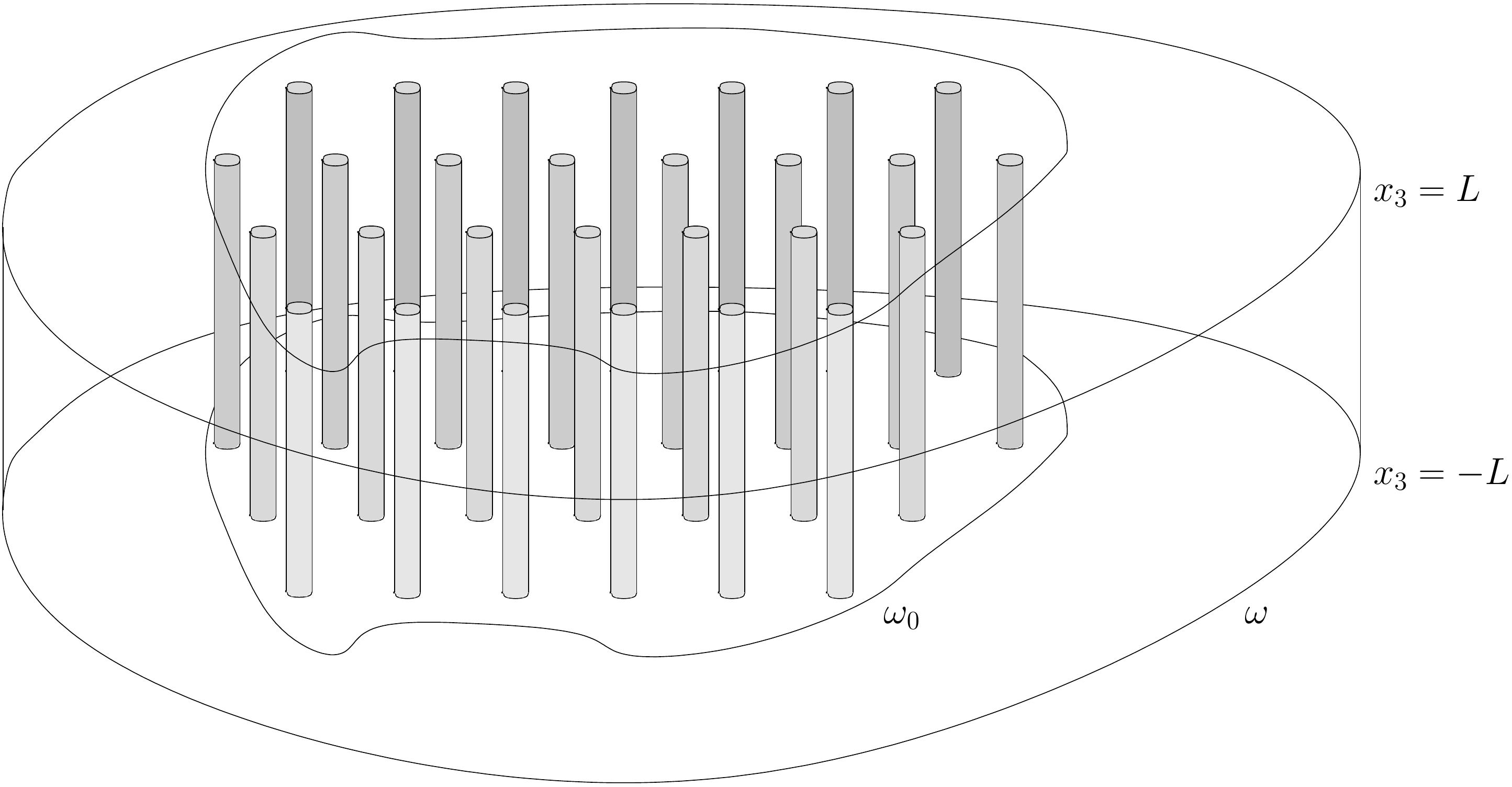}
\caption{\label{fig:fig1} Sketch of a domain $\Omega=\omega\times(-L,L)$ containing  cylindrical rods.} 
\end{center}
\end{figure}
A sketch of the domain is represented in Figure~\ref{fig:fig1}. 
For $\eps > 0$, $\reps \ll \eps$, and $(m,n) \in \ZZ^2$, let 
\begin{align*}
D_{m,n,\eps}
	&{}= D_{\eps\reps}(m\eps,n\eps)
	\;,\\
Q_{m,n,\eps}
	&{}= D_{m,n,\eps} \times (-L,L)
	\;,
\end{align*}
where $D_{\eps\reps}(m\eps,n\eps)$ is the disc of radius $\eps\reps$ centered at
$(m\eps,n\eps)$. Introduce the index set
\[
I_\eps  {}= \{(m,n) \in \ZZ^2: D_{m,n,\eps} \subset \omega_0\}\;,
\]
and set
$$
\Deps {}= \bigcup_{(m,n) \in I_\eps} D_{m,n,\eps}	\;, \qquad \Ceps {}= \bigcup_{(m,n) \in I_\eps} Q_{m,n,\eps}	\;.
$$
We will assume that the conductivity parameter of $\HOmeg$ is of the form
\begin{equation}
\aeps {}= \left\{\begin{array}{ll}
1 & \text{ in } \HOmeg \setminus \Ceps	\;,\\
\ceps & \text{ in } \Ceps	\;,
\end{array}\right.
\label{eq:defaeps}
\end{equation}
where 
\begin{equation}
0 < \kappa_{-} \leq \ceps\,\pi\reps^2 \leq \kappa_{+}	\;.
\label{eq:boundkappa}
\end{equation}
Additionally, we assume
\begin{equation}
0 < \frac{2\pi}{\gamma_{+}} \leq -\eps^2\ln\reps\, \leq \frac{2\pi}{\gamma_{-}}
	\;.
\label{eq:Cond1}
\end{equation}
We consider, for $\fib \in C^1(\bar\HOmeg)$ and $F$ $\in$ $L^2(\HOmeg)$, the solution $\Psi_\eps \in H^1(\HOmeg)$ of
\begin{equation}
\left\{\begin{array}{ll}
-\divop(\aeps\,\nabla \Psi_\eps) = F &\text{ in } \HOmeg	\;,\\
\Psi_\eps =  \fib&\text{ on } \partial\HOmeg	\;.
\end{array}\right.
\label{eq:Full-3D}
\end{equation}

This model, initially introduced by Fenchenko \& Khruslov \cite{FENCHENKO-KHRUSLOV-81}, has been studied in the context of 
homogenization by several authors \cite{KHRUSLOV-91,BELLIEUD-BOUCHITTE-98,BRIANE-TCHOU-01,BRIANE-03,MARCHENKO-KHRUSLOV-06}. 
It is known to have a non-standard behaviour when $\eps$ tends to zero. Namely, the homogenized limit is not of divergence 
form, and admits a non-local term (see Theorem~\ref{thm:hom} for the precise form of the limit). 
While it is clear that, thanks to the ellipticity, a global $W^{1,2}\left(\Omega\right)$ bound holds, one can show that 
except for special boundary data $\fib$, the solutions $\Psi_\eps$ of \eqref{eq:Full-3D} are unbounded in
 $W^{1,p}_{\rm loc}(\HOmeg_0)$ for any $p>2$, see Corollary~\ref{cor:hom}. This makes the situation very 
different from the case of bounded coefficients: Meyers' Theorem \cite{MEYERS-63} shows that solutions are bounded in $W^{1,p}_{\rm loc}$ for some $p > 2$ in that case. 
In the case of periodic composites with bounded coefficients, Li \& Vogelius \cite{LI-VOGELIUS-00} and Li \& Nirenberg \cite{LI-NIRENBERG-03} showed one can in fact 
obtain $W^{1,\infty}$ estimates. However, such improved regularity estimates strongly depend on the contrast of the coefficients as shown 
for example in \cite{LEONETTI-NESI-97,AMMARI-KANG-LIM-05,BAO-LI-YIN-09}.

The first goal of this work is to establish interior
$C^{1,\alpha}$ estimates, uniformly in $\eps$, for $\Psi_\eps$ away from $\Ceps$. 
We show that in a set $\Omega_\eps=\omega_\eps \times (-l,l)$ improved 
regularity estimates can be obtained. The set $\homeps$ is ``almost'' the complement of the high conductivity fibres in the sense 
that $\omega_\eps \cap D_\eps = \emptyset$ and $\omega_0\setminus \omega_\eps$ tends to zero with $\eps$. 
Introducing $V$ the solution  in $H_0^1(\Omega)$ of  
$$
-\Delta V= F, 
$$
we show that 
$$
\|\Psi_\eps-V\|_{C^{1,\nu}\left(\Omega_\eps\right)}\leq C\,\|F\|_{L^2(\HOmeg)},
$$
see Theorems~\ref{thm:Regularity-3d} and \ref{thm:Regularity-3d**} for precise statements.

We can think of several applications for this work. One could for example use this result to 
establish a posteriori error estimates for the numerical solutions of
\eqref{eq:Full-3D}. Our initial motivation came from a question related to imaging. 
In a recent work, Ammari et al. \cite{AMMARI-BONNETIER-CAPDEBOSCQ-09} showed that the signature
of small inclusions inside a periodic medium was determined by the effective properties of this 
medium, in the limit when both the period and the size of the inclusions tend to zero. 
The motivation for this study was to provide a mathematical perspective on the so-called 
``resolution beyond the diffraction limit''  verified experimentally \cite{LEROSEY-ET-AL-07}.

Recent developments in material sciences, see Bouchitt\'e et al. \cite{BOUCHITTE-BOUREL-FELBACQ-09}, 
have shown that very high contrast composite materials, with scalings similar to the ones used 
in this paper, could be used to construct  meta-materials with particularly interesting 
properties (such as materials with negative optical indices). Such composites are out of 
the scope of \cite{AMMARI-BONNETIER-CAPDEBOSCQ-09}. This work relies on elliptic regularity 
estimates shown by Li \& Nirenberg which do not apply for large contrast.
These imaging problems involve either the Helmholtz equation, or the full Maxwell equation. We limited 
our study to the case of a real valued conductivity coefficient in this work. Assume a perturbation of small volume $\Geps$ is located 
inside $\Omega_\eps$ with conductivity $\cepsdef \in C^0(\Omega)$. The conductivity of 
the defective medium is
\begin{equation}
\aepsdef {}= \left\{\begin{array}{ll}
1 & \text{ in } \Omega \setminus (\Ceps \cup \Geps)	\;,\\
\cepsdef & \text{ in } \Geps	\;,\\
\ceps & \text{ in } \Ceps	\;.
\end{array}\right.
\label{eq:aepsgamma1}
\end{equation}
For $\fib \in C^1(\bar\Omega)$, consider $\uepsdef$ the solution of
\begin{equation}
\left\{\begin{array}{ll}
\divop(\aepsdef\,\nabla \uepsdef) = 0 	& \text{ in } \Omega	\;,\\
\uepsdef = \fib	&\text{ on } \partial\Omega	\;,
\end{array}\right.
\label{eq:u-with-defect}
\end{equation}
and compare it to the unperturbed solution $\uepsnodef$ of 
\begin{equation}
\left\{\begin{array}{ll}
\divop(\aeps\,\nabla \uepsnodef) = 0 	& \text{ in } \Omega	\;,\\
\uepsnodef = \fib	&\text{ on } \partial\Omega	\;.
\end{array}\right.
\label{eq:u-no-defect}
\end{equation}
The signature on the boundary of the defect is characterised by the quadratic form $\DN:H^{1/2}(\partial \Omega) \to  \mathbb{R}$ defined by
$$
\DN(\fib) {}=  \int_{\partial\Omega}
		\aeps \frac{\partial}{\partial n}\left(\uepsdef-\uepsnodef\right)(s)\fib(s)ds.
$$
We show in Theorem~\ref{thm:defect} that at first order, for $\fib \in C^1(\bar\Omega)$, it is of the form
$$
\DN(\fib) = \left|\Geps \right|\int_{\Omega}M\nabla \uhom \cdot\nabla\uhom\,d\mu +o\left(\left|\Geps \right|\right),
$$
where $\uhom$ is the solution of the homogenized problem \eqref{eq:hom-pro} associated with \eqref{eq:u-no-defect}. 
This can be seen as an extension of \cite{CAPDEBOSCQ-VOGELIUS-03A}  to \eqref{eq:u-no-defect}. 
At first order, the signature of the defect is similar to what would be observed if defect was 
introduced in the effective medium, instead of the homogeneous substrate: the only difference is the formula 
of polarisation tensor $M$. This is what was observed in \cite{BENHASSEN-BONNETIER-06} for finite 
conductivities. 

Theorem~\ref{thm:defect} has another interpretation, probably of equal if not greater importance for applications. 
Theorem~\ref{thm:defect}, compared with the main result 
of \cite{AMMARI-BONNETIER-CAPDEBOSCQ-09} or  \cite{BENHASSEN-BONNETIER-06}  shows that impurities in the 
substrate do not affect the  overall properties of the composite material  more than impurities  would affect a 
regular composite, provided these impurities are located in $\Omega_\eps$, that is, not too
 close to the highly 
conducting fibres but not necessarily at a distance proportional to the size of the microstructure. As far as the 
authors are aware, this is the first result of this nature. If such highly contrasted structures are used to manufacture 
meta-materials, as it is suggested in \cite{BOUCHITTE-BOUREL-FELBACQ-09}, this result is of practical importance.

Another very related question is the regularity of the solution of a problem similar to \eqref{eq:Full-3D}, where cavities 
replace inclusions. This problem shows a similar non-standard effective behaviour, the celebrated ``strange term'' of 
Cioranescu and Murat \cite{CIORANESCU-MURAT-82}. This will be the subject of a forthcoming paper.

\medskip{}

Our paper is structured as follows. In Section~\ref{sec:main}, we state our main results concerning the regularity 
of \eqref{eq:Full-3D} away from the highly conductive fibres. 
First, we consider for $F$ $\in$ $L^2(\HOmeg)$, the solution $\Ueps \in H_0^1(\HOmeg)$ of
\begin{equation}
\left\{\begin{array}{ll}
-\divop(\aeps\,\nabla \Ueps) = F  &\text{ in } \HOmeg	\;,\\
\Ueps =  0&\text{ on } \partial\HOmeg	\;.
\end{array}\right.
\label{eq:INT-3D}
\end{equation}
Our result concerning problem \eqref{eq:INT-3D} is Theorem~\ref{thm:Regularity-3d}.
 We then turn to the boundary value 
problem \eqref{eq:u-no-defect}, and obtain an estimate for $\Weps$ in Theorem~\ref{thm:Regularity-3d**}. 
To highlight the fact that excluding a buffer zone around the highly conducting fibres is necessary, for any $p>2$ we 
provide an explosive lower bound for the $W^{1,p}$ norm of  $\Weps$ in Corollary~\ref{cor:hom} (see also Remark~\ref{rem:bup}). 
This result is a corollary of the homogenization result given by Theorem~\ref{thm:hom}. 

Our approach relies on the fact that one can perform a partial Fourier series decomposition of $\Ueps$ and $\Weps$ in $x_3$. In Section~\ref{sec:2D-3D} 
we prove Theorem~\ref{thm:Regularity-3d} and Theorem~\ref{thm:Regularity-3d**} starting from the regularity of the Fourier coefficients. 
Note that translating the interior regularity results from $\Ueps$ to $\Weps$ is not 
obvious. A natural idea is to study $\tilde\Weps=\eta\Weps$, where $\eta$ is a cut-off function. Provided the cut-off function 
is chosen carefully, this leads to an interior problem in $H^1_0(\HOmeg)$, 
$$
-\divop(\aeps\,\nabla \tilde\Weps) =  - 2 \aeps \nabla \Weps \cdot \nabla \eta - \aeps \Delta \eta \Weps  \text{ in } \HOmeg.
$$
Let us focus on the first right-hand side term. Remembering that the only bound uniform in $\eps$ is the energy bound, 
$\int_{\Omega} a_\eps |\nabla \Weps|^2 \leq C \|\fib\|_{C^1(\bar\HOmeg)}$, since $\nabla \Weps$ is not bounded uniformly in 
$\eps$ in any $L^{p}(\HOmeg)$ for $p>2$, we are then left with a problem of the form 
$$
-\divop(\aeps\,\nabla \hat\Weps) =  - \aeps f,\mbox{ in } H^1_0(\Omega) \mbox{ with } \int_\Omega \aeps f^2 dx <\infty. 
$$
Unfortunately, without additional information on $f$, this is not enough to guarantee that $\hat\Weps$ is bounded in 
$L^\infty(\Omega)$. In Proposition~\ref{pro:contrex-marc}, we show that the best one can hope for in 
this case is  $\eps \|\hat\Weps\|_{L^\infty}< C\| \sqrt{\aeps} f \|_{L^2}$.

In Section~\ref{sec:defect}, we show how these results can be used to obtain a representation 
formula for inclusions of small measure located away from the fibres. This part was the initial motivation of this work. 
Because outside of the highly conducting fibres the substrate is homogeneous, we provide a self-contained proof. 

We prove the homogenization Theorem~\ref{thm:hom} in Section~\ref{sec:hom}, and finally turn to more technical results.
Section~\ref{sec:dGNM} is devoted to the proof of Lemma~\ref{lem:dGNM}, a supremum estimate of de Giorgi-Moser-Nash type adapted to 
our problem. The proof of this Lemma uses a Poincar\'e-Sobolev inequality proved in 
Lemma~\ref{lem:Sob-Marc}. Section~\ref{sec:proof-2d} contains the proof of the regularity results for the Fourier coefficients 
associated with $\Ueps$. Section~\ref{sec:ABCD} is the counterpart of Section~\ref{sec:proof-2d}, for $\Weps$. Several intermediate 
technical results are proved in Appendix~\ref{sec:AnnexA}.

\section{\label{sec:main} Main interior regularity estimates}

The main part of this paper is devoted to the derivation of interior regularity estimates for problem~\eqref{eq:Full-3D}, 
outside of the highly conducting rods.

For a fixed $\omega_0 \Subset \omega_1 \Subset \homeg$, define $\homegtau$ by
\begin{equation}\label{eq:defhomegtau}
\homegtau{}= \{x \in \omega_1: \dist(x,\Deps) \geq \eps \tau\}.
\end{equation}
Our result concerning problem \eqref{eq:INT-3D} is the following.
\begin{theorem}\label{thm:Regularity-3d}
For $\kappa > 0$,  $\tau > \kappa\,\eps^{\frac{1-\eta}{2(1+\eta)}}$ with 
$\eta\in\left(\frac{3}{4},1\right)$, and $0<\nu <2\left(\eta - \frac{3}{4}\right)$, 
the solution $\Ueps$ of \eqref{eq:INT-3D} satisfies
\[
\left\| \Ueps - V\right\|_{C^{1,\nu}\left(\homegtau \times(-L,L)\right)}
	\leq C(\kappa,\eta,\nu) \|F\|_{L^2(\Omega)},
\]
where $V\in H^1_0(\Omega)$ is the solution of
$$
-\Delta V = F.
$$ 
\end{theorem}
This result is proved in Section~\ref{sec:2D-3D}.
\begin{remark}
It should be noted that, generically, $V$ may not belong to any $W^{1,p}(\Omega)$ for $p > 6$.
The above result asserts that the difference $\Ueps - V$ enjoys a better regularity in most of the domain since
$|\omega_1 \setminus \homegtau| \rightarrow 0$ as $\eps \rightarrow 0$.
\end{remark}

We use the notation $C$ for various constants in the paper which are always independent of $\eps$. 
When appropriate, we highlight the dependence on the parameters appearing in the statements of the results by writing $C(a, b, \ldots )$.

As a direct consequence of the above theorem, we have
\begin{corollary}\label{cor:Regularity-3d}
Under the assumptions of Theorem \ref{thm:Regularity-3d} we have
\begin{equation}
\left\| \Ueps\right\|_{L^\infty\left(\homegtau \times(-L,L)\right)} 
+ \left\|\nabla \Ueps\right\|_{L^6\left(\homegtau \times(-L,L)\right)}
\leq C(\kappa,\eta) \|F\|_{L^2(\Omega)}
	.\label{eq:sauve-aubin-nitsche}
\end{equation}
\end{corollary}

Next, we derive an estimate for the solution $\Weps$  of \eqref{eq:u-no-defect},  given by the following proposition.

\begin{theorem}\label{thm:Regularity-3d**}
Let $\Omegtau {}= \homegtau \times(-l,l)$, with $l<L$. 
For $\kappa > 0$,  $\tau > \kappa\,\eps^{\frac{1-\eta}{2(1+\eta)}}$ with $\eta\in\left(\frac{3}{4},1\right)$, 
and $0<\nu <2\left(\eta - \frac{3}{4}\right)$, the solution $\Weps$ of \eqref{eq:u-no-defect} satisfies
$$
\|\Weps\|_{C^{1,\nu}\left(\Omegtau\right)}\leq C(\kappa,\eta,\nu, l)\,\|\Phi_b\|_{C^1(\bar\HOmeg)}.
$$
\end{theorem}
This result is proved in Section~\ref{sec:2D-3D}.

\begin{remark}
Once again, as $\eps \rightarrow 0$, we have $|\omega_1\setminus\homegtau|=o(1)$, therefore the solution of \eqref{eq:INT-3D} enjoys a uniform $C^{1,\nu}$-bound in almost 
the whole
domain. 
\end{remark}

It is natural to ask whether a uniform global $C^{1,\nu}$-bound exists for $\Omega$. To answer this question, first we consider the limit homogenized problem 
corresponding to \eqref{eq:u-no-defect}. To state our homogenization result, we need to introduce a capacity function. The function $\csigeps$ is defined, 
for some $\sigma \geq 1$, in $\HOmeg$ by
\begin{equation}
\csigeps(x){}=\left\{\begin{array}{cl}
0 & \mbox{if }r=\sqrt{(x_1-m\eps)^2+(x_2-n\eps)^2}\leq\eps r_\eps
\\*[.4em]
\displaystyle \frac{\ln r-\ln(\eps r_\eps)}{\ln(\eps^\sigma/2)-\ln(\eps r_\eps)} & \mbox{if }r\in(\eps r_\eps,\eps^\sigma/2),\quad\mbox{for }(m,n)\in I_\eps,
\\*[.4em]
1 & \mbox{elsewhere}.
\end{array}\right.
\label{eq:defceps}
\end{equation}
In fact, we choose $\sigma$ so that $\eps^{\sigma} > \eps\,\tau$. Note that $\csigeps$ does not depend on the variable $x_3$, 
is periodic of period $[-\eps/2,\eps/2]^3$ in $\HomO$,
and is equal to $1$ outside of $\HomO$.
This capacity function has been ubiquitous in the derivation of homogenization results
related to conductivities of the form \eqref{eq:defaeps} since its introduction in \cite{CIORANESCU-MURAT-82}. 

We have the following homogenization result, proved in Section \ref{sec:hom}. Note that $\aeps$  is only reinforced in the cylinder $\HomO=\omega_0\times(-L,L)$.  
\begin{theorem}\label{thm:hom}
Assume in addition to (\ref{eq:boundkappa}), (\ref{eq:Cond1}), that 
\begin{equation}\label{alpeps}
\lim_{\eps\to 0}\;\ceps\,\pi \reps^2=\kappa\in[\kappa_-,\kappa_+],\quad
\lim_{\eps\to 0}\;\frac{2\pi}{\eps^2\left|\ln \reps\right|}=\gamma\in[\gamma_-,\gamma_+].
\end{equation}
Let $\fib\in C^1(\bar\HOmeg)$. Then, the solution $\Weps$ converges weakly in $H^1(\HOmeg)$ to the unique solution $\uhom$ of the coupled system
\begin{equation}
\left\{\begin{array}{rl}
\displaystyle -\,\Delta \uhom+\gamma\left(\uhom-\vhom\right)1_{\Omega_0}=0 & \mbox{in }\Omega
\\*[.6em]
-\,\kappa\,\partial^2_{33}\vhom+\gamma\left(\vhom-\uhom\right)=0 & \mbox{in }\Omega_0=\omega_0\times(-L,L)
\\*[.4em]
\uhom=\fib & \mbox{on }\partial\HOmeg
\\*[.4em]
\vhom(\cdot,\pm L)=\fib(\cdot,\pm L) & \mbox{in }\omega_0.
\end{array}\right.
\label{eq:hom-pro}
\end{equation}
The pair $(\uhom, \vhom)$ satisfies $\uhom\in C^{0,\alpha}(\bar\Omega) \cap W^{2,p}_{\rm loc}(\HOmeg)$ 
and $\vhom\in C^{0,\alpha}(\bar\Omega_0) \cap C^\infty_{\rm loc}(\Omega_0)$ for some $\alpha \in (0,1)$ and for any $p>2$. 
\par
Moreover, the following corrector result holds:
\begin{equation}\label{cor.res}
\lim_{\eps\to 0}\;
\int_\HOmeg \aeps\big|\nabla \Weps-\nabla \csigeps\left(\uhom-\vhom\right)-\csigeps\nabla \uhom-(1- \csigeps)\,\partial_3 \vhom\,e_3\big|^2\,dx=0.
\end{equation}
\end{theorem}
Because of the non-local nature of the problem it solves, it is not clear that $\uhom$  is analytic inside $\HomO$. 
The following proposition shows that it enjoys partial
analyticity, that is, with respect to the $x_3$ variable.
\begin{proposition}\label{pro:analx3}
Under the hypothesis of Theorem~\ref{thm:hom}, the solutions $\uhom$ and $\vhom$ of \eqref{eq:hom-pro} are analytic with respect to the $x_3$ variable in $\HomO$.
\end{proposition}
This result is proved in Appendix~\ref{sec:pfax3}. We are now in position to answer the global regularity question. 
Theorem~\ref{thm:hom} and Proposition~\ref{pro:analx3} imply that one cannot expect an $\eps$ independent bound for the sequence  $\Weps$ in a space 
better than $H^1(\HOmeg)$  in general, as the following Corollary shows.
\begin{corollary}\label{cor:hom}
Let $\fib\in C^2(\bar{\HOmeg})$ be such that the sets 
$$
S_{\pm}(\fib)=\left\{ x^\prime \in \omega_0 \mbox{ such that } \Delta_2 \fib(\cdot, \pm L)=0 \right\},
\quad\mbox{where }\Delta_2=\partial^2_{11}+\partial^2_{22},
$$
have an intersection of zero measure. Then, the sequence $\Weps$ is unbounded in $W^{1,p}_{\rm loc}(\HomO)$ for any $p>2$. 
More precisely, for any non empty open set $\Omega^\prime \subset \Omega_0$, there exists a positive constant $C(\Omega^\prime,\Phi_b)$ independent of
$\eps$ such that for any $\rho \in (0,1)$ and $p>2$, 
\[
\liminf_{\eps \to 0}\left(\reps^{2\rho(1 - 2/p)} \|\nabla \Weps\|_{L^p(\Omega' \setminus \Ceps)}^2\right)
	\geq (1 - \rho)C\left(\Omega^\prime,\fib\right).
\]
\end{corollary}

\begin{proof}
Working with a subsequence if necessary, we can assume that \eqref{alpeps} holds.

Fix $\rho \in (0,1)$ and let $\mathcal{O}_\eps$ be the open subset of $\HomO$ defined by
\begin{equation}\label{omeps}
\mathcal{O}_\eps=\big\{x \in \omega_0: 0<\dist(x,\Deps)<\eps\reps^\rho\big\}\times(-L,L),
\end{equation}
By the definition \eqref{eq:defceps} of $\csigeps$ combined with (\ref{alpeps}), 
and by the periodicity of $\csigeps$ and $\mathbf{1}_{\mathcal{O}_\eps}$ in $\HomO$, we have
\[
1_{\mathcal{O}_\eps}\,|\nabla\csigeps|^2\;\rightharpoonup\;
\lim_{\eps\to 0}\left(\frac{2\pi}{\eps^2}\int_{\eps \reps}^{\eps\reps^\rho}\frac{dr}{r\, \ln^2(2\eps^{1-\sigma}\,\reps)}\right)
= \gamma(1-\rho)\quad\mbox{weakly-}*\mbox{ in }\mathcal{M}(\HomO),
\]
which by virtue of the continuity of $\vhom-\uhom$ implies that
\begin{equation}\label{con.hVomeps}
1_{\mathcal{O}_\eps}\,|\nabla\csigeps|^2\,(\vhom-\uhom)^2\,dx\;\rightharpoonup\;\gamma(1-\rho)\,(\vhom-\uhom)^2
\quad\mbox{weakly-}*\mbox{ in }\mathcal{M}(\HomO).
\end{equation}
Let us show that this last term does not cancel on any non-empty open subset $\Omega^\prime\subset\Omega_0$.
By contradiction, suppose that the continuous function $\vhom-\uhom\equiv 0$ on a ball $B_R(p^0)$, centred in $p^0=(p^0_1,p^0_2,p^0_3)$ and of radius $R>0$, 
such that $B_R(p_0)\subset \Omega_0$. 
This implies that for example that for any $x^\prime\in C_R$, with $C_R=(p^0_1 - R/2,p^0_1 + R/2)\times(p^0_2 - R/2,p^0_2 + R/2)$, 
the function $\uhom(x^\prime,\cdot)-\vhom(x^\prime,\cdot)$ is identically zero on $(p^0_3 -\sqrt{3}R/2,p^0_3 +\sqrt{3}R/2)$. 
As we have shown in Proposition~\ref{pro:analx3},
this function is analytic on $(-L,L)$ therefore $\uhom=\vhom$ on $C_R\times(-L,L)$. The system \eqref{eq:hom-pro} then shows that
 $\partial^2_{33}\vhom=\Delta \vhom=\Delta \uhom=0$ in~$C_R\times(-L,L)$. Hence, taking into account the boundary conditions of 
\eqref{eq:hom-pro} we would have
\[
\uhom(x)=\vhom(x)=\frac{L+x_3}{2L}\,\fib(x',L)+\frac{L-x_3}{2L}\,\fib(x',-L)\quad\mbox{for any }x\in C_R\times(-L,L).
\]
Since $\uhom$ is harmonic in $C_R\times(-L,L)$, it follows that both functions $\fib(\cdot,\pm L)$ are 
harmonic in $C_R$, therefore $|S_{+}(\Phi_b)\cap S_{-}(\Phi_b)|\geq |C_R|>0$
which contradicts the hypothesis.

Next, since $\nabla \uhom,\partial_{3}\vhom\in L^2(\HomO)$ and $0\leq \csigeps\leq 1$ in~$\HomO$, we have
\[
\lim_{\eps \to 0}\int_{\mathcal{O}_\eps}\big| \csigeps\nabla \uhom+(1- \csigeps)\,\partial_3 \vhom\,e_3\big|^2\,dx = 0.
\]
Hence, by the convergence (\ref{cor.res}) and the fact that $a_\eps\equiv 1$ in $\mathcal{O}_\eps$, we deduce that
\[
1_{\mathcal{O}_\eps}\big|\nabla \Weps-\nabla \csigeps\left(\uhom-\vhom\right)\big|^2\;\longrightarrow\;0\quad\mbox{strongly in }L^1_{\rm loc}(\HomO).
\]
This combined with convergence (\ref{con.hVomeps}) implies that
\begin{equation}\label{conv.GWeps}
1_{\mathcal{O}_\eps}|\nabla \Weps|^2\;\rightharpoonup\; \gamma(1-\rho)(\vhom-\uhom)^2\quad\mbox{weakly-}*\mbox{ in }{\mathcal M}(\HomO).
\end{equation}

To proceed, fix $\Omega'' \Subset \Omega' \subset \Omega_0$ and select a smooth cut-off function $\zeta \in C_c^\infty(\Omega')$ 
such that $0 \leq \zeta \leq 1$ and $\zeta \equiv 1$ in $\Omega''$. Note that, for some $C_1 > 1$
\[
C_1^{-1}\,\reps^{2\rho} \leq |\mathcal{O}_\eps| \leq C_1\,\reps^{2\rho}.
\]
Thus, \eqref{conv.GWeps} shows
\begin{align*}
\int_{\Omega''} \gamma(1-\rho)(\vhom-\uhom)^2\,dx
	&\leq \int_{\Omega_0} \gamma(1-\rho)(\vhom-\uhom)^2\,\zeta\,dx\\
	&= \lim_{\eps \to 0}\int_{\mathcal{O}_\eps} |\nabla \Weps|^2\,\zeta\,dx\\
	&\leq \liminf_{\eps \to 0} \int_{\mathcal{O}_\eps \cap \Omega'} |\nabla \Weps|^2\,dx\\
	&\leq \liminf_{\eps \to 0} \left(|\mathcal{O}_\eps|^{1 - 2/p} \|\nabla \Weps\|_{L^p(\Omega' \setminus \Ceps)}^2\right)\\
	&\leq C_1\liminf_{\eps \to 0} \left(\reps^{2\rho(1 - 2/p)} \|\nabla \Weps\|_{L^p(\Omega' \setminus \Ceps)}^2\right).
\end{align*}
Since this is valid for all $\Omega'' \Subset \Omega'$, we conclude the proof.
\qed
\end{proof}

\begin{remark}\label{rem:bup}
Note that Corollary~\ref{cor:hom} shows that gradient blow-up occurs outside of the highly conducting rods. 
An variant of the proof shows that one can obtain the following estimate
\[
\liminf_{\eps \to 0}\left(\reps^{2\rho(1 - 2/p)} \|\nabla \Weps\|_{L^p(\Omega' \setminus \hat Q_\eps^{\rho'})}^2\right)
	\geq (\rho' - \rho)C\left(\Omega^\prime,\fib\right),
\]
where $0 < \rho < \rho' < 1$ and
\[
\hat Q_\eps^{\rho'} = \big\{x \in \omega_0: \dist(x,\Deps)\leq \eps\reps^{\rho'}\big\}\times(-L,L),
\]
which shows that the blow-up is not localized on the surface of the rods.
\end{remark}
\section{Regularity estimates in two dimensions for weighted equations and three-dimensional consequences\label{sec:2D-3D}}

In this section, we state the regularity results we have obtained for two-dimensional companion problems of problem~\eqref{eq:Full-3D}. 
To highlight the sharpness of the supremum estimate provided by Lemma~\ref{lem:dGNM}, which provides an $L^\infty$ upper bound of the form $C(s)\eps^{-s}$
 for any $2>s>1$, we construct in Proposition~\ref{pro:contrex-marc} an example where the $L^\infty$ norm is bounded from below by $C\epsilon^{-1}$.
Then, we prove Theorem~\ref{thm:Regularity-3d} and Theorem~\ref{thm:Regularity-3d**}.

\subsection{Towards the proof of Theorem~\ref{thm:Regularity-3d}}
Let us now turn to the proof of Theorem~\ref{thm:Regularity-3d}. Because none of the coefficients depend on $x_3$, we can, as in \cite{BRIANE-CASADO-08},
reduce the study of this three-dimensional problem to that of a two-dimensional problem by separating variables. We write
\begin{align}
\Ueps(x',x_3) 
	&\sim \sum_{n=1}^\infty \uepsn(x')\,\sin\Big(\frac{n\pi}{2}\big(\frac{x_3}{L} + 1\big)\Big)
	\;,\label{UepsFExp}\\
F(x',x_3)
	&\sim \sum_{n=1}^\infty f_n(x')\,\sin\Big(\frac{n\pi}{2}\big(\frac{x_3}{L} + 1\big)\Big)
	\;,\label{FFExp}
\end{align}
and \eqref{eq:INT-3D} becomes
\begin{equation}
\left\{\begin{array}{ll}
-\divop_2(\aeps\,\nabla_2 \uepsn) + \frac{n^2\pi^2}{4L^2}\,\aeps\,\uepsn 
                                                      = f_n &\text{ in } \homeg	\;,\\
\uepsn = 0 &\text{ on } \partial\homeg	\;.
\end{array}\right.
\label{eq:ini2.5d}
\end{equation}
In the above, $\divop_2$ and $\nabla_2$ denotes the horizontal divergence and gradient operators.
We are thus led to consider, for $\lambda \geq \frac{\pi^2}{4L^2}$,
\begin{equation}
\left\{\begin{array}{ll}
-\divop_2(\aeps\,\nabla_2 \ueps) + \lambda\,\aeps\,\ueps = f +\aeps\, g+\divop_2(\bh)&\text{ in } \homeg	\;,\\
\ueps = 0 &\text{ on } \partial\homeg	\;,
\end{array}\right.
\label{eq:Eq2.5d}
\end{equation}
where $f \in L^2(\omega)$, $\sqrt{\aeps}g \in L^2(\omega)$ and $h\in L^\infty(\omega)^2$. 
Problem~\eqref{eq:Eq2.5d} is equivalent to Problem~\eqref{eq:ini2.5d} if $g = 0, \bh=0$; the 
additional $\aeps\, g + \divop_2(\bh)$ term will prove useful for the study of boundary value problem \eqref{eq:u-no-defect}.

Section~\ref{sec:proof-2d} is devoted to the proof of the following result.
\begin{proposition}\label{prop:Linfybd2d}
Assume that $\lambda \geq \lambda_0 > 0$. For  $\kappa > 0$, $\eta \in (0,1)$, and $\tau > \kappa\,\eps^{\frac{1-\eta}{2(1+\eta)}}$
the solution $\ueps$ to \eqref{eq:Eq2.5d} enjoys the following bound
\begin{align*}
& \|\ueps - v \|_{L^\infty(\homegtau)} + \|\nabla_2 \ueps - \nabla_2 v \|_{L^\infty(\homegtau)} \\
	&\qquad\qquad\leq
 \frac{C(\lambda_0,\kappa,\eta)}{\lambda^{\frac{\eta}{2}}}\left[\frac{1}{\lambda^{\frac{\eta}{2}}}\left(\|f\|_{L^2(\homeg)}+\|\sqrt{\aeps}\,g\|_{L^2(\homeg)}
 + \|\bh\|_{L^\infty(\homeg)^2}\right) 
         + \|\ueps\|_{L^\infty(\homeg)}\right],
\end{align*}
where $v\in H^1_0(\homeg)$ is the solution of
$$
-\Delta v +\lambda v= f + g + \divop_2(\bh).
$$ 
Furthermore, for $0 < \nu < \eta$,
\begin{align*}
&[\nabla_2 \ueps - \nabla_2 v]_{C^\nu(\homegtau)} \\
	&\qquad\qquad \leq \frac{C(\lambda_0,\kappa,\eta,\nu)}{\lambda^{\frac{\eta}{2} - \frac{\nu}{2}}}
\left[\frac{1}{\lambda^{\frac{\eta}{2}}}\left(\|f\|_{L^2(\homeg)}+\|\sqrt{\aeps}\,g\|_{L^2(\homeg)} 
+ \|\bh\|_{L^\infty(\homeg)^2}\right) +  \|\ueps\|_{L^\infty(\homeg)}\right].
\end{align*}
\end{proposition}
Throughout the paper $[\cdot]_{C^\nu}$ will be used to denote the $\nu$-H\"older semi-norm.
The second ingredient for the proof is a supremum estimate for $\ueps$. 
\newcommand{\vpe}{\varphi_\eps}
\begin{lemma}\label{lem:dGNM} Assume that conditions \eqref{eq:defaeps},\eqref{eq:boundkappa} and \eqref{eq:Cond1} hold.
Assume that $\vpe\in H_{0}^{1}\left(\omega\right)$ satisfies,
\begin{equation}\label{eq:dGNM}
-{\rm div}\left(\aeps \nabla \vpe\right)+\lambda \aeps \vpe= \aeps f_{\eps} +\divop\left(\aeps\bh\right).
\end{equation}
Then, for any $1 < \alpha < 2$, $0 < \beta<1-\alpha/2$,  we have
$$
\left\Vert \vpe \right\Vert _{L^{\infty}(\omega)} \leq \frac{C(\alpha,\beta)}{\displaystyle \eps^{\alpha}\lambda^{\beta}} 
\left( \left\Vert \sqrt{\aeps}f\right\Vert_{L^{2}\left(\omega\right)} + \eps^{\frac{\alpha+1}{2}}\left\Vert \sqrt{\aeps}h\right\Vert_{L^{\infty}\left(\omega\right)}\right).
$$
\end{lemma}
This variation on the standard de Giorgi-Moser-Nash estimates is proved in Section~\ref{sec:dGNM}. As a direct consequence we have the following proposition.
\begin{proposition} \label{pro:supboundueps}
Assume that $\lambda \geq \lambda_0 > 0$, $g=0$ and $\bh=0$. We have
\begin{equation}\label{eq:supboundueps}
\|\ueps\|_{L^\infty(\homeg)} \leq \frac{C(\lambda_0,q)}{\lambda^{1/q}}\|f\|_{L^2(\homeg)} 
\end{equation}
for any $q > 2$.
\end{proposition}
Note that Lemma~\ref{lem:dGNM} does not show that $\|\vpe\|_{L^\infty(\omega)}$ is bounded independently of $\eps$. 
The following counter-example documents the sharpness of our estimate, as it shows that it must be at least $O(\eps^{-1})$.
\begin{proposition}\label{pro:contrex-marc} 
Assume that conditions \eqref{eq:defaeps}, \eqref{eq:boundkappa} and \eqref{eq:Cond1} hold. Let $\vpe \in H^1_0(\omega)$ be the solution of \eqref{eq:dGNM}.
Let $\mathbf{1}_{D_{0,0,\eps}}$ be the indicator function of the disk $D_{0,0,\eps} \subset \RR^2$ 
centred at the origin and  of radius $\eps r_\eps$. 
Set $f_\eps=\eps^{-1}\mathbf{1}_{D_{0,0,\eps}}$ and $\bh \equiv 0$. Then, we have
\begin{equation}\label{est.cex}
\|\sqrt{\aeps}f_\eps\|_{L^2(\omega)}\leq\sqrt{\kappa_{+}},\quad\mbox{and}\quad
\left\Vert \vpe \right\Vert_{L^{\infty}(\omega)}
\geq
\frac{1}{\eps}
\frac{\kappa_{-}}{\kappa_{+}}
\frac{\kappa_{-}}{\gamma_{+}+\lambda\,\kappa_{+}+o(1)},
\end{equation}
where $\lim_{\eps \to 0} o(1) = 0$.
\end{proposition}

\noindent {\it Proof of Theorem~\ref{thm:Regularity-3d}.}
Fix some $l \in (0,L)$ and let $\Omegtau = \homegtau \times (-l,l)$. It suffices to provide the required estimates in $\Omegtau$ with
a constant that is independent of $l$.

Together with \eqref{UepsFExp} and \eqref{FFExp} we consider the Fourier decomposition
$$
V(x',x_3) \sim \sum_{n=1}^\infty v_{n}(x')\,\sin\Big(\frac{n\pi}{2}\big(\frac{x_3}{L} + 1\big)\Big).
$$
Note that $\Ueps - V$ is harmonic, thus regular, in $\HOmeg \setminus \Ceps$. 
Therefore, by Dirichlet's theorem on the convergence of Fourier series,  
the Fourier expansions of $\Ueps - V$ and $\nabla (\Ueps - V)$ converge pointwise to $\Ueps - V$ 
and $\nabla (\Ueps - V)$ in $\Omegtau$. 
Proposition~\ref{prop:Linfybd2d} and Proposition~\ref{pro:supboundueps} combined with the choice $q=\frac{2}{\eta}$ show that 
\begin{align}
 \|\uepsn - v_n \|_{L^\infty(\homegtau)} + \|\nabla_2 \uepsn - \nabla_2 v_n \|_{L^\infty(\homegtau)} \leq  \frac{C(\kappa,\eta)}{n^{2\eta}}  \|f_n\|_{L^2(\homeg)}
	,\label{eq:bdpfueps}\\
[\nabla_2 \uepsn - \nabla_2 v_n ]_{C^\nu(\homegtau)} \leq  \frac{C(\kappa,\eta)}{n^{2\eta-\nu}}  \|f_n\|_{L^2(\homeg)}
	,\label{eq:bdpfuepsXX}
\end{align}
with $0<\nu<\eta<1$. 
For $\eta > \frac{1}{4}$, \eqref{eq:bdpfueps} shows that the sums
\begin{align*}
&I_1(x',x_3)
	=\sum_{n=1}^\infty \left(\uepsn(x') - v_{n}(x^\prime)\right)\,\sin\Big(\frac{n\pi}{2}\big(\frac{x_3}{L} + 1\big)\Big),\\
&I_2(x',x_3)
	=\sum_{n=1}^\infty \left(\nabla_2 \uepsn(x') - \nabla_2 v_{n}(x^\prime)\right)\,\sin\Big(\frac{n\pi}{2}\big(\frac{x_3}{L} + 1\big)\Big)
\end{align*}
are absolutely convergent, 
\begin{align*}
\|I_1\|_{L^\infty(\Omegtau)} + \|I_2\|_{L^\infty(\Omegtau)} 
	&\leq C \sum_{n=1}^\infty \frac{1}{n^{2\eta}}   \|f_n\|_{L^2(\homeg)} \\ 
	&\leq C\|F\|_{L^2(\Omega)} \left(\sum_{n=1}^\infty \frac{1}{n^{4\eta}}\right)^\frac{1}{2} \leq \frac{C}{\eta - \frac{1}{4}} \|F\|_{L^2(\Omega)}.
\end{align*}
This shows that
\[
\left\|\Ueps - V\right\|_{L^\infty\left(\Omegtau\right)} + \left\|\nabla_2 \Ueps -\nabla_2 V\right\|_{L^\infty\left(\Omegtau\right)}\leq  C \|F\|_{L^2(\Omega)}.
\]

We turn to the H\"older estimates of $\nabla_2 \Ueps - \nabla_2 V$. By \eqref{eq:bdpfuepsXX} with $\eta > \frac{\nu}{2} + \frac{1}{4}$ we have,
\begin{align*}
|I_2(x',x_3) - I_2(y',x_3)|
	&\leq \sum_{n=1}^\infty \left|\left(\nabla_2 \uepsn(x') - \nabla_2 v_{n}(x^\prime)\right) - \left(\nabla_2 \uepsn(y') - \nabla_2 v_{n}(y^\prime)\right)\right|\\
	&\leq C\sum_{n=1}^\infty \frac{1}{n^{2\eta - \nu}}\,\|f_n\|_{L^2(\homeg)}\,|x' - y'|^{\nu}\\
	&\leq \frac{C}{\eta - \frac{\nu}{2} - \frac{1}{4}}\|F\|_{L^2(\Omega)}\,|x' - y'|^{\nu}
	.
\end{align*}
By \eqref{eq:bdpfueps} but with $\eta > \frac{\nu}{2} + \frac{1}{4}$ we obtain,
\begin{align*}
& |I_2(y',x_3) - I_2(y',y_3)| \\
	&\qquad\qquad\leq \sum_{n=1}^\infty \left|\nabla_2 \uepsn(y') - \nabla_2 v_{n}(y^\prime)\right|\,\left|\sin\Big(\frac{n\pi}{2}\big(\frac{x_3}{L} + 1\big)\Big)
          - \sin\Big(\frac{n\pi}{2}\big(\frac{y_3}{L} + 1\big)\Big)\right|\\
	&\qquad\qquad\leq C\sum_{n=1}^\infty \frac{1}{n^{2\eta}}\,\|f_n\|_{L^2}\,|nx_3 - ny_3|^\nu\\
	&\qquad\qquad\leq \frac{C}{\eta - \frac{\nu}{2} - \frac{1}{4}}\,\|F\|_{L^2(\Omega)}\,|x_3 - y_3|^\nu
	.
\end{align*}
From these last two estimates, we deduce 
\[
\left[\nabla_2 \Ueps -\nabla_2 V\right]_{C^\nu\left(\Omegtau\right)}\leq  C \|F\|_{L^2(\Omega)}.
\]

Next, we estimate $\partial_{x_3} \Ueps - \partial_{x_3} V$. Choosing $\eta>\frac{3}{4}$ in \eqref{eq:bdpfueps} shows that the sum
$$
I_3 =\sum_{n=1}^\infty n\,\left( \uepsn(x') -  v_{n}(x^\prime)\right) \cos\,\Big(\frac{n\pi}{2}\big(\frac{x_3}{L} + 1\big)\Big)
$$
converges absolutely with bound
\begin{align*}
\|I_3\|_{L^\infty(\Omegtau)} 
	\leq C \sum_{n=1}^\infty \frac{1}{n^{2\eta-1}}   \|f_n\|_{L^2(\homeg)} 
	\leq \frac{C}{\eta -\frac{3}{4}} \|F\|_{L^2(\Omega)}.
\end{align*}
This proves that
\[
\left\|\partial_{x_3} \Ueps - \partial_{x_3} V\right\|_{L^\infty\left(\Omegtau\right)}\leq  C \|F\|_{L^2(\Omega)}.
\]
For the H\"older estimate, we proceed as before. First, with $\eta > \frac{3}{4}$ in \eqref{eq:bdpfueps} we obtain
\begin{align*}
|I_3(x',x_3) - I_3(y',x_3)|
	&\leq \sum_{n=1}^\infty n\left|\left(\uepsn(x') - v_{n}(x^\prime)\right) - \left(\uepsn(y') - v_{n}(y^\prime)\right)\right|\\
	&\leq C\sum_{n=1}^\infty \frac{1}{n^{2\eta - 1}}\,\|f_n\|_{L^2(\homeg)}\,|x' - y'|\\
	&\leq \frac{C}{\eta-\frac{3}{4}}\|F\|_{L^2(\Omega)}\,|x' - y'|
	.
\end{align*}
Also by \eqref{eq:bdpfueps} but with $\eta > \frac{\nu}{2} + \frac{3}{4}$,
\begin{align*}
&|I_3(y',x_3) - I_3(y',y_3)|\\
	&\qquad\qquad\leq \sum_{n=1}^\infty n\left|\uepsn(y') - v_{n}(y^\prime)\right|\,\left|\cos\Big(\frac{n\pi}{2}\big(\frac{x_3}{L} + 1\big)\Big)
         - \cos\Big(\frac{n\pi}{2}\big(\frac{y_3}{L} + 1\big)\Big)\right|\\
	&\qquad\qquad\leq C\sum_{n=1}^\infty \frac{1}{n^{2\eta - 1}}\,\|f_n\|_{L^2}\,|nx_3 - ny_3|^\nu\\
	&\qquad\qquad\leq \frac{C}{\eta - \frac{\nu}{2} - \frac{3}{4}}\,\|F\|_{L^2(\Omega)}\,|x_3 - y_3|^\nu.
\end{align*}
We have shown
\[
\left[\partial_{x_3} \Ueps - \partial_{x_3} V\right]_{C^\nu\left(\Omegtau\right)}\leq  C \|F\|_{L^2(\Omega)},
\]
which concludes the proof. 
\qed 
\par\bigskip\noindent
{\it Proof of Proposition~\ref{pro:contrex-marc}.} 
The bound on $f_\eps$ in \eqref{est.cex} is straightforward, from the definition of $\aeps$, \eqref{eq:defaeps},
 and \eqref{eq:boundkappa}. Integrating \eqref{eq:dGNM} against $\vpe$ and using \eqref{eq:boundkappa}, we find
$$
\int_{\omega} \aeps\left|\nabla \vpe\right|^2 dx + \int_{\omega} \lambda\,\aeps\,\vpe^2\,dx = \int_{\omega}a_\eps\,f_\eps\,\vpe\,dx
=\int_{D_{0,0,\eps}}a_\eps\,\vpe\,\eps^{-1}dx\leq \kappa_{+}\,\eps\left\Vert \vpe \right\Vert _{L^{\infty}(\omega)}.
$$
Let $Y=(-1/2,1/2)^2$ be the unit period cell, and set
\begin{equation}\label{eq:def-geps-test}
g_\eps=\eps^{-1}(1-{c_\eps^{1}})\,\mathbf{1}_{\eps Y} \in H^1_0(\omega),
\end{equation}
where $c_\eps^1$ is defined by \eqref{eq:defceps} for $\sigma=1$.
An easy computation using \eqref{eq:boundkappa} and \eqref{eq:Cond1}, shows that
$$
\begin{array}{ll}
\displaystyle \int_{\omega} \aeps\left|\nabla g_\eps\right|^2 dx + \int_{\omega} \lambda\,\aeps\,g_\eps^2\,dx
& \displaystyle =\int_{\eps Y}\left|\nabla c^1_\eps\right|^2 \eps^{-2}\,dx
+\int_{\eps Y} \lambda\,\aeps\,(1-c^1_\eps)^2\,\eps^{-2} dx
\\
& \displaystyle \leq \gamma_{+}+\lambda\,\kappa_{+}+o(1).
\end{array}
$$ 
Then, the Cauchy-Schwarz inequality combined with the two previous estimates yields
$$
\int_{\omega} \aeps\,\nabla \vpe \cdot \nabla g_\eps\,dx + \int_{\omega} \aeps\,\vpe\,g_\eps\,dx
\leq 
\left(\kappa_{+}\,\eps\left\Vert \vpe \right\Vert _{L^{\infty}(\omega)}\right)^{\frac{1}{2}}
\left(\gamma_{+}+\lambda\,\kappa_{+}+o(1)\right)^{\frac{1}{2}}.
$$
On the other hand, integrating  \eqref{eq:dGNM} against $g_\eps$, we find
$$
\int_{\omega} \aeps\,\nabla \vpe \cdot \nabla g_\eps\,dx + \int_{\omega} \aeps\,\vpe\,g_\eps\,dx
=\int_{\omega}\aeps\,f_\eps\,g_\eps\,dx=\int_{D_{0,0,\eps}}a_\eps\,\eps^{-2}dx\geq \kappa_{-}.
$$
Combining these two inequalities provides the announced bound.
\qed
\subsection{Towards the proof of Theorem~\ref{thm:Regularity-3d**}}
Let us now turn to the proof of Theorem~\ref{thm:Regularity-3d**}. Our approach is to decompose 
the boundary data $\fib$ into three parts. We first define 
\begin{equation}
\phi_L(\cdot,x_3) = \frac{x_3}{2L}\,\big(\fib(\cdot,L)-\fib(\cdot,-L)\big) + \frac{1}{2}\,\big(\fib(\cdot,L)+\fib(\cdot,-L)\big),
\label{eq:phiLdefn}
\end{equation}
which agrees with $\fib(\cdot,\pm L)$ on $\bar\omega\times\{\pm L\}$.
We now freeze the variations of $\phi_L$ in the $x^\prime$-direction on the highly conducting fibres by decomposing $\phi_L$ into
\[
\phi_L=\phi_{1}+\phi_{2},
\]
where 
\begin{equation}
\phi_{1} = \phi_L \csigeps
+\sum_{(m,n) \in I_\eps}\left(1-\csigeps\right)\mathbf{1}_{\{\left|(x_1,x_2)-(m,n)\eps\right| \leq \eps/2\}}\left(x\right)\phi_L \left(m\eps,n\eps,x_3\right),
\label{eq:phi1defn}
\end{equation}
the capacity function $\csigeps$ being defined in \eqref{eq:defceps}. 
We finally define $\phi_0$ to be the trace of $\fib$ on the lateral boundary of the cylinder $\HOmeg$,
$$
\phi_{0}=\fib-\phi_L=\fib-\phi_1-\phi_2\quad\mbox{on }\partial\Omega.
$$ 
Note that $\phi_0$ vanishes on $\bar\omega\times\{\pm L\}$.

\par
The properties of $\phi_{1}$ and $\phi_{2}$ we will use are described by the following proposition.
\begin{proposition}\label{pro:phi1}
The function $\phi_1$ given by \eqref{eq:phi1defn} 
      \begin{itemize}
      \item belongs to $C^\infty([-L,L];C^{0,1}(\bar\omega))$,
      \item is independent of $x_1$ and $x_2$ inside each connected component of $\Ceps$,
      \item satisfies $\partial_{33}\phi_1=0$ in $\HOmeg$, 
      \item is globally Lipschitz,
$$
\|\phi_1\|_{C^{0,1}\left(\bar{\Omega}\right)}\leq C \|\phi_L\|_{C^{1}\left(\bar{\Omega}\right)}.
$$
      \end{itemize}
The function $\phi_2=\phi_L - \phi_1$ where $\phi_L$ is defined in \eqref{eq:phiLdefn}
   \begin{itemize}
      \item is globally Lipschitz,
$$
\|\phi_2\|_{C^{0,1}\left(\bar{\Omega}\right)}\leq C \|\phi_L\|_{C^{1}\left(\bar{\Omega}\right)},
$$
     \item is supported by 
    $$
    \omega_{R_\eps}=\left\{x\in\HomO\mbox{ s.t. } \min_{(m,n)\in I_\eps}\left|(x_1,x_2)-(m\eps,n\eps)\right| \leq \eps^{\sigma}\right\},
    $$
     \item and satisfies 
$$
\| \phi_2 \|_{L^\infty(\Omega)}\leq C \eps^{\sigma+2} \|\phi_L\|_{C^{1}\left(\bar{\Omega}\right)}.
$$
    \end{itemize}
\end{proposition}
The proof is straightforward and given in Appendix~\ref{sec:phi1phi2}.

\par\bigskip\noindent
{\it Proof of Theorem~\ref{thm:Regularity-3d**}.} 
We will study three boundary problems separately. For $i=0,1,2$, we introduce $W_{\eps,i}$, the solution of
\begin{equation}
\left\{\begin{array}{ll}
\divop\left(\aeps\,\nabla W_{\eps,i}\right) = 0 &\text{ in } \HOmeg	\;,\\
W_{\eps,i} =  \phi_{i}&\text{ on } \partial\HOmeg	\;.
\end{array}\right.
\label{eq:BDY-3DC}
\end{equation}
By linearity, $W_\eps = W_{\eps,0} + W_{\eps,1} + W_{\eps,2}$.

Consider $\tilde{W}_{\eps,0}= \zeta_0\,W_{\eps,0}$ where $\zeta_0$ is a cut-off function $\zeta_0\in C_c^\infty(\omega_3)$,  such that
\begin{equation}\label{eq:defzeta0}
0\leq \zeta_0 \leq 1 \mbox{ and } \zeta_0  \equiv 1  \mbox{ in }\omega_{2},
\end{equation}
with $\omega_1 \Subset\omega_2\Subset\omega_3\Subset\omega$. Since $\zeta_0\equiv 1$ on $\omega_0$, $\tilde{W}_{\eps,0}\in H_0^1(\Omega)$  satisfies
\begin{equation}\label{eq:tw0}
-\divop\left(\aeps\,\nabla \tilde{W}_{\eps,0}\right) = -2\nabla \zeta_0 \cdot \nabla W_{\eps,0} - \left(\Delta \zeta_0 \right) W_{\eps,0} \text{ in } \HOmeg. 
\end{equation}
The maximum principle shows that $\|W_{\eps,0}\|_{L^\infty(\Omega)}\leq \| \phi_0 \|_{C^0(\bar\Omega)}$. 
Furthermore, on $\Omega_0^c=(\omega\setminus\omega_0)\times(-L,L)$, $\ueps = W_{\eps,0}$ is the solution of
\begin{eqnarray*}
\Delta \ueps & = & 0\mbox{ in }\Omega_0^c,\\
\ueps & = & 0 \mbox{ on } (\omega \setminus \omega_0) \times \{-L,L\},\\
\ueps & = & W_{\eps,0} \mbox{ on } \partial(\omega \setminus \omega_0) \times(-L,L).
\end{eqnarray*}
Classical regularity estimates then show that
\begin{equation}\label{eq:externalreg}
\| W_{\eps,0} \|_{C^{2}(\overline{(\omega_3\setminus\omega_2)\times(-L,L)})} \leq C(\omega_2,\omega_3,l) \left\Vert W_{\eps,0} \right\Vert_{L^{\infty}(\Omega)}.
\end{equation}
Thus, the right-hand side of \eqref{eq:tw0} is in $C^1(\bar\Omega)$, 
and we can apply Theorem~\ref{thm:Regularity-3d} to conclude that 
for  $\tau > \kappa\,\eps^{\frac{1-\eta}{2(1+\eta)}}$ with
$\eta\in \left(\frac{3}{4},1\right)$, and $0<\nu <2\left(\eta - \frac{3}{4}\right)$, we have
\[
\left\| \tilde{W}_{\eps,0} - \tilde{V}_{0}\right\|_{C^{1,\nu}\left(\Omegtaul\right)}
	\leq C \left\Vert W_{\eps,0} \right\Vert_{L^{\infty}(\Omega)},
\]
where $\tilde{V}_0\in H^1_0(\Omega)$ is the solution of
$$
-\Delta \tilde{V}_0 =   -2\nabla \zeta_0 \cdot \nabla W_{\eps,0} - \left(\Delta \zeta_0 \right) W_{\eps,0} \text{ in } \HOmeg. 
$$ 
Since the right-hand side is $C^1(\bar\Omega)$ , $\tilde{V}_0$ enjoys interior regularity, and we have obtained
\begin{equation}\label{eq:estw0}
\left\| W_{\eps,0}\right\|_{C^{1,\nu}\left(\Omegtaul\right)}
	\leq C \left\Vert \fib \right\Vert_{C^{0}(\bar\HOmeg)}.
\end{equation}
Next, we consider $W_{\eps,1}$. We note that for any $v\in H^1_0(\HOmeg)$,
\begin{eqnarray*}
\int_{\Omega} \aeps \nabla \left(W_{\eps,1} - \phi_1\right) \cdot \nabla v & = & -\int_{\Omega}\aeps \nabla\phi_{1}\cdot\nabla v ,\\
 & = & \int_{\Ceps} \left(1-\aeps\right)\nabla\phi_{1}\cdot \nabla v  - \int_{\Omega} \nabla\phi_{1}\cdot \nabla v, \\
 & = & \int_{\Ceps} \left(1-\aeps\right) \partial_{3}\phi_{1}\cdot \partial_{3}v +\int_{\Omega} \Delta \phi_{1} v, \\
 & = &  \int_{\Omega} \Delta \phi_{1} v.
\end{eqnarray*}
We used first that $\phi_1$ is independent of $x_1$ and $x_2$ in each connected component of $\Ceps$, 
and then that it satisfies $\partial_{33}\phi_1=0$ (see Proposition~\ref{pro:phi1}). In other words, introducing $\tilde W_{\eps,1} = W_{\eps,1}-\phi_1 $, we have
\begin{equation}\label{eq:defweps1}
-\mbox{div}\left(\aeps \nabla \tilde{W}_{\eps,1}  \right) =  \Delta \phi_1 \in \Omega.
\end{equation}
Following the strategy used for the interior source problem, we thus introduce $\tilde{V}_{1}\in H^1_0(\HOmeg)$, the solution of 
$$
-\Delta \tilde{V}_{1} = \Delta \phi_1 \text{ in } \Omega,
$$
In Section~\ref{sec:ABCD}, we will show in Proposition~\ref{pro:tilde-W1} that, for $\tau > \kappa\,\eps^{\frac{1-\eta}{2(1+\eta)}}$, $\frac{1}{2} < \eta < 1$, 
$0 < \nu < 2(\eta - \frac{1}{2})$ that
\[
\left\| \tilde{W}_{\eps,1} - \tilde{V}_{1}\right\|_{C^{1,\nu}\left(\Omegtaul\right)}
	\leq C \left\Vert \Phi_b \right\Vert_{C^1(\bar\HOmeg)}.
\]
On the other hand, as $\tilde V_1 + \phi_1$ is harmonic in $\HOmeg$, we have
\[
\left\| \tilde V_1 + \phi_1 \right\|_{C^{1,\nu}\left(\Omegtaul\right)}
		\leq C \left\Vert \Phi_b \right\Vert_{L^{\infty}(\Omega)}.
\]
As $W_{\eps,1} = (\tilde W_{\eps,1} - \tilde V_1) + (\tilde V_1 + \phi_1)$, we have thus shown that
\begin{equation}\label{eq:estw1}
\left\| W_{\eps,1}\right\|_{C^{1,\nu}\left(\Omegtaul\right)}
	\leq C \left\Vert \Phi_b \right\Vert_{L^{\infty}(\Omega)}.
\end{equation}

Finally, for $W_{\eps,2}$, we shall consider instead the function $\tilde{W}_{\eps,2}=\zeta\,W_{\eps,2}$ where  $\zeta(x)=p(x_3)\zeta_0$ 
and $\zeta_0$ is the cut-off function given by \eqref{eq:defzeta0}, and $p\in C_{c}^\infty(-l',l')$ with $l < l' < L$ is such that 
\begin{equation}\label{eq:defp}
0\leq p\leq 1 \mbox{ and } p\equiv 1 \mbox{ in } (-l,l).
\end{equation}
We compute 
\begin{equation}\label{eq:BDY-4DC}
-\mbox{div}\left(\aeps \nabla \tilde{W}_{\eps,2} \right) =\aeps \left(2\partial_{3}\zeta\partial_{3}\,W_{\eps,2}
+ \partial_{33}\zeta\,W_{\eps,2}\right)+ g_\eps,
\end{equation}
where $g_\eps$ has support in $(\omega_3\setminus\omega_2)\times(-l^\prime,l^\prime)$, and is given by
$$
g_\eps = (\partial_{11}\zeta+\partial_{22}\zeta)\,W_{\eps,2}+ 2\partial_{1}\zeta\,\partial_{1}W_{\eps,2} + 2\partial_{2}\zeta\,\partial_{2}\,W_{\eps,2}.
$$
By the standard maximum principle, 
\begin{equation}\label{eq:maxprinc-phi2}
\left\Vert W_{\eps,2}\right\Vert_{L^{\infty}(\Omega)}\leq\left\Vert  \phi_{2} \right\Vert _{L^{\infty}(\Omega)}.
\end{equation}

Just like $\tilde{W}_{\eps,0}$ enjoys additional regularity away from the support of the rods as shown by \eqref{eq:externalreg}, we have, for any $l' < l''<L$,
\begin{equation}\label{eq:externalreg2}
\| W_{\eps,2} \|_{C^{2}(\overline{(\omega_3\setminus\omega_2)\times(-l'',l'')})} \leq C(\omega_2,\omega_3,l'') \left\Vert W_{\eps,2}\right\Vert_{L^{\infty}(\Omega)}.
\end{equation}
Next, notice that $\partial_{x_3} W_{\eps,2}$ and $W_{\eps,2}$ are both $\aeps$-harmonic, a standard energy estimate 
shows that, for any $\Omega' \Subset \Omega'' \Subset \Omega$, we have
\begin{equation}\label{eq:reg-uepsd3}
\int_{\Omega'} \aeps\,|\nabla(\partial_3 W_{\eps,2})|^2\,dx
	\leq C\int_{\Omega''} \aeps\,|\partial_3 W_{\eps,2}|^2\,dx 
	\leq C\int_{\Omega} \aeps\,|W_{\eps,2}|^2\,dx 
	\leq C\,\|W_{\eps,2}\|_{L^\infty(\Omega)}.
\end{equation}
Note that the regularity estimate \eqref{eq:externalreg2} yields in particular that
\begin{equation}\label{eq:reg-g}
\| g_\eps \|_{C^1(\bar\HOmeg)} \leq C \left\Vert  \phi_{2} \right\Vert _{L^{\infty}(\Omega)}.
\end{equation}
The companion problem to \eqref{eq:BDY-4DC} is
\begin{equation}
-\Delta \tilde{V}_{\eps,2} = 2\partial_{3}\zeta\partial_{3}\,W_{\eps,2}
+ \partial_{33}\zeta\,W_{\eps,2} + g_\eps \text{ in } \HOmeg	\;,\\
\label{eq:BDY-TV01}
\end{equation}
Thanks to \eqref{eq:externalreg2}, the right-hand side of \eqref{eq:BDY-TV01} is in $C^1(\bar\Omega)$ uniformly in $\eps$, and 
therefore $\tilde{V}_{\eps,2}$ enjoys interior regularity. Proposition~\ref{pro:tilde-W2} proved in Section~\ref{sec:ABCD}
on the regularity of the difference $\tilde{W}_{\eps,2} -\tilde{V}_{\eps,2}$ implies a regularity result of $W_{\eps,2}$. For $\kappa > 0$, 
 $\tau > \kappa\,\eps^{\frac{1-\eta}{2(1+\eta)}}$ with $\eta
\in\left(\frac{2}{3},1\right)$, and $0<\nu <3\left(\eta - \frac{2}{3}\right)$, the solution $W_{\eps,2}$ of \eqref{eq:BDY-3DC} satisfies
\begin{equation}\label{eq:estw2}
\left\| W_{\eps,2} \right\|_{C^{1,\nu}\left(\Omegtaul\right)}
	\leq C\, \eps^\sigma\, \|\fib\|_{C^1(\bar\Omega)}.
\end{equation}
The claim of the Theorem is a consequence of the three bounds \eqref{eq:estw0}, \eqref{eq:estw1} and \eqref{eq:estw2}.
\qed
\section{Application to structures with defects of small volume\label{sec:defect}}

We now consider the case when a defect of small volume is present in the medium. 
In that case, the conductivity of the defective medium is given by \eqref{eq:aepsgamma1}. 
We assume that the defect of support $\Geps$ stays away from the high conductivity fibres. 
To fix ideas, given  $0<l<L$ and $\omega_1$ such that $\omega_0\Subset\omega_1\Subset\omega$, suppose
\begin{equation}\label{Omeps}
G_\eps\subset\Omega_\eps = \big\{x \in \omega_1: \dist(x,\Deps) \geq \eps^{17/16}\big\}\times(-l,l).
\end{equation}
Note that the set $\Omega_\eps$ grows as $\eps$ tends to zero. Furthermore, in the language of Section \ref{sec:main}, 
$\Omega_\eps = \Omegtau$ for $\tau = \eps^{1/16}$, $\kappa = 1$, $\eta = 7/9$. This guarantees a $C^{1,\nu}$ estimate with $0 < \nu < 1/18$.

Throughout the section, we will assume that \eqref{alpeps} holds, namely 
$$
\lim_{\eps\to 0}\;\ceps\,\pi \reps^2=\kappa\in[\kappa_-,\kappa_+],\quad
\lim_{\eps\to 0}\;\frac{2\pi}{\eps^2\left|\ln \reps\right|}=\gamma\in[\gamma_-,\gamma_+].
$$

When $\eps\to0$, Theorem~\ref{thm:hom} shows that the sequence $\uepsnodef$ 
converges weakly in $H^1_0(\Omega)$ to the solution $\uhom$ of the coupled system~\eqref{eq:hom-pro} when $\fib\in C^1(\bar\Omega)$.

We proceed to derive an asymptotic formula for the difference response operator (the difference
of Dirichlet-to-Neumann operators on $\partial \Omega$)
\begin{eqnarray*}
H^{1/2}(\partial \Omega) & \to & H^{-1/2}(\partial \Omega)\\
\fib & \mapsto &
\aeps\left.\frac{\partial}{\partial
n}\left(\uepsdef-\uepsnodef\right)\right|_{\partial
\Omega}.
\end{eqnarray*} As this operator can be recovered by
polarisation, we limit ourselves to the study of the quadratic
form
\begin{eqnarray}
\DN:H^{1/2}(\partial \Omega) & \to & \mathbb{R} \notag \\
\fib & \mapsto & \int_{\partial
\Omega}\aeps \frac{\partial}{\partial
n}\left(\uepsdef-\uepsnodef\right)(s)\fib(s)ds.\label{eq:Def-Rbil}
\end{eqnarray}

We can now state the main result of this section.

\begin{theorem} \label{thm:defect}
Assume that $|\Geps|\to 0$, and $|\Geps|^{-1} \mathbf{1}_{\Geps}$ converges weakly-$*$ to a Radon measure $\mu$ on $\bar{\Omega}$. There exist a subsequence, still denoted by $\eps$,
and a matrix-valued function $M\in L^2(\Omega,\mu)^{3\times 3}$ such that 
the bilinear response form $\DN$ given by~\eqref{eq:Def-Rbil} has the following asymptotic form
\begin{equation}\label{Reps(Phib)}
\DN(\fib)=\left|\Geps\right|\int_{\Omega}M\nabla \uhom \cdot\nabla\uhom\,d\mu +o\left(\left|\Geps \right|\right) \text{ for any } \fib\in C^1(\bar\HOmeg),
\end{equation}
where $\uhom$ denotes the solution to the homogenized problem \eqref{eq:hom-pro} with boundary condition $\fib$, and
\[
\lim_{\eps \rightarrow 0} \sup_{\|\fib\|_{C^{1}(\bar{\HOmeg})} \leq 1} \frac{o(|G_\eps|)}{|G_\eps|} = 0.
\]
In addition, the matrix-valued function $M$ is symmetric, independent of $\fib$, and satisfies
$$
\left(\gamma_1-1\right)\min\left(1,\frac{1}{\gamma_1} \right) \leq M(y)\xi\cdot\xi \leq \left(\gamma_1-1\right)\max\left(1,\frac{1}{\gamma_1} \right) \text{ $\mu$-a.e.} 
$$
\end{theorem}
\begin{remark}\label{rem.suppmu}
Due to the definition (\ref{Omeps}) of $\Omega_\eps$, the measure $\mu$ has actually support in $\bar{\omega}_1\times[-l,l]\Subset\Omega$. This gives a sense to the integral term of (\ref{Reps(Phib)}) since $W_*\in C^1_{\rm loc}(\Omega)$.
\end{remark}
\begin{remark}
We have characterised the signature of a defect provided that the boundary condition $\fib$ is in $C^1(\bar\HOmeg)$. 
The natural space for $\fib$ is $H^{1/2}(\partial\Omega)$, but we do not know if our result holds in this space as well.
\end{remark}
Our strategy is inspired by  \cite{CAPDEBOSCQ-VOGELIUS-03A,AMMARI-BONNETIER-CAPDEBOSCQ-09}. The first key element is  the following Lemma.

\begin{lemma}\label{lem:Aubin-Nitsche}
There exist
a positive constant $C$ independent of $\fib$ and $\eps$ such that 
\begin{align*}
&\|\nabla(\uepsnodef - \uepsdef)\|_{L^2(\Omega)} + |G_\eps |^{-1/3} \|\uepsnodef - \uepsdef\|_{L^2(\Omega)} 
\leq C\,|G_\eps |^{1/2}\,\|\fib\|_{C^{1}(\bar{\HOmeg})}.
\end{align*}
\end{lemma}

\begin{proof} Write $\chi_{\eps ,d} = \uepsnodef - \uepsdef \in H^1_0(\Omega)$, we have 
\[
\divop(\aepsdef\nabla \chi_{\eps ,d}) = \divop((\aepsdef - \aeps)\nabla\uepsnodef) = \divop((\cepsdef - 1)\,\mathbf{1}_{\Geps}\,\nabla\uepsnodef)
\]
Thus, by testing against $\chi_{\eps ,d}$, we get
\[
\|\nabla \chi_{\eps ,d}\|_{L^2(\Omega)} 
	\leq C\|\nabla \uepsnodef\|_{L^2(\Geps)}	
	\leq C\,|\Geps|^{1/2}\,\|\nabla \uepsnodef\|_{L^\infty(\Geps)}
	\;.
\]
Next note that, thanks to Theorem~\ref{thm:Regularity-3d**}, we have
\begin{equation}\label{est.GWeps}
\|\nabla \uepsnodef\|_{L^\infty(\Geps)}\leq \|\uepsnodef\|_{C^{1,\nu}(\HOmeg_\eps)} 
\leq C \|\fib\|_{C^{1}(\bar{\HOmeg})},
\end{equation}
with $\nu=1/20$. Thus
\begin{equation}
\|\nabla \chi_{\eps ,d}\|_{L^2(\Omega)} 	
	\leq C\,|\Geps|^{1/2}\,\|\fib\|_{C^{1}(\bar{\HOmeg})}
	\;.\label{eq:wepsdefL2}
\end{equation}
This proves the first half of the desired estimate,
namely 
$$
\|\nabla(\uepsnodef - \uepsdef)\|_{L^2(\Omega)} \leq C\,|\Geps|^{1/2}\,\|\fib\|_{C^{1}(\bar{\HOmeg})}.
$$
To prove the second part of the estimate, introduce the solution $F_\eps\in H^1_0(\Omega)$ 
of
$$
-\divop \left( \aeps \nabla F_\eps\right) = \chi_{\eps ,d}\;.
$$
Thanks to Corollary~\ref{cor:Regularity-3d}, we have
$$
\left\| \nabla F_\eps \right\|_{L^6(\Omega_\eps)}\leq C \|\chi_{\eps ,d}\|_{L^2(\HOmeg)}.
$$
On the other hand, note that
\[
\divop(\aeps\nabla \chi_{\eps ,d}) = \divop((\aepsdef - \aeps)\nabla\uepsdef) = \divop((\cepsdef - 1)\,\mathbf{1}_{\Geps}\,\nabla\uepsdef).
\]
After an integration by part, we see that
\begin{eqnarray*}
\|\chi_{\eps ,d}\|_{L^2(\Omega)}^2 &=& \int_\Omega a_\eps \nabla F_\eps \cdot \nabla \chi_{\eps ,d}\,dx \\ 
	&=& \int_{\Geps} (\cepsdef - 1)\,\mathbf{1}_{\Geps}\,\nabla F_\eps \cdot \nabla \uepsdef \,dx\\
	&=& -\int_{\Geps} \left(\gamma_1 -1 \right)\nabla F_\eps \cdot \nabla \chi_{\eps ,d}\,dx 
         + \int_{\Geps} \left(\gamma_1 -1 \right)\nabla F_\eps \cdot \nabla \uepsnodef \,dx\\
	&\leq&  C \left\| \nabla F_\eps \right\|_{L^6(\Omega_\eps)}  \left( \left\| \nabla \chi_{\eps ,d} \right\|_{L^{6/5}(\Geps)}
         + \|\nabla \uepsnodef\|_{L^\infty(\Geps)}  |\Geps|^{5/6}\right) \\
	&\leq&  C \|\chi_{\eps ,d}\|_{L^2(\HOmeg)}  \left( \|\nabla \chi_{\eps ,d} \|_{L^{2}(\HOmeg)} |\Geps|^{5/6 - 1/2}
         + \|\nabla \uepsnodef\|_{L^\infty(\Geps)}  |\Geps|^{5/6}\right) \\
	&\leq&  C \|\chi_{\eps ,d}\|_{L^2(\HOmeg)}  |\Geps|^{5/6}\|\fib\|_{C^{1}(\bar{\HOmeg})},
\end{eqnarray*}
where we have used \eqref{eq:wepsdefL2} in the penultimate inequality. The proof is complete.
\qed\end{proof}

The second ingredient is a pointwise uniform estimate on $\uepsnodef - \uhom$. We use the following extension of 
Theorem~\ref{thm:hom} with varying boundary data, the proof of which consists of a straightforward adaptation of the 
first step of the proof of Theorem~\ref{thm:hom}.
\begin{theorem}\label{thm:hom2}
Assume that the conditions (\ref{eq:boundkappa}), (\ref{eq:Cond1}) and (\ref{alpeps}) hold.
Consider a sequence $\varphi_\eps\in C^1(\bar\HOmeg)$ which converges strongly in $C^0(\bar\HOmeg)$ to some $\varphi \in C^{0,1}(\bar\HOmeg)$. 
Then, the solution $\Weps$ of  problem (\ref{eq:u-no-defect}) with the boundary condition $\fib=\varphi_\eps$ converges weakly in $H^1(\HOmeg)$ 
to the solution $\uhom$ of \eqref{eq:hom-pro} with the boundary condition $\fib=\varphi$. 
Furthermore, $\uhom\in W^{2,p}_{\rm loc}(\HOmeg)$ and $\vhom\in W^{2,p}_{\rm loc}(\HomO)$ for any $p>2$.
\end{theorem}

\begin{proposition}
\label{pro:def-regularite} The solutions $\uepsnodef$ and $\uhom$ to respectively \eqref{eq:u-with-defect} 
and \eqref{eq:hom-pro} satisfy the following convergence in the set $\Omega_\eps$ of \eqref{Omeps},
\begin{equation}\label{conv.C1Weps}
\lim_{\eps\to 0}\sup_{\begin{array}{c}
 \fib\in {C^{1}(\bar{\HOmeg})}\\
\left\|\fib\right\|_{C^{1}(\bar{\HOmeg})}\leq 1                                    
                                   \end{array}
}\|\nabla\uepsnodef-\nabla\uhom\|_{C^{0,1/20}\left(\Omega_\eps\right)}=0.
\end{equation}
\end{proposition}
\begin{proof}
By contradiction, suppose that there exist
 a constant $c_{0}>0$, a sequence $(\eps_n)_{n\in\mathbb{N}}$, and a sequence
 $\varphi_{n}\in {C^{1}(\bar{\HOmeg})}$ such that  
\[
\left\Vert \varphi_{n}\right\Vert_{C^{1}(\bar{\HOmeg})}\leq1,
\]
and 
\[
\left\Vert \nabla W_{n}-\nabla W_{*,n}\right\Vert _{C^{0,1/20}\left(\Omega_{\eps_n}\right)}\geq c_{0},
\]
where $W_{n}$ is the solution of
\[
\mbox{div}(a_{\eps_{n}}\nabla W_{n})=0\mbox{ in }\HOmeg,\mbox{ and }W_{n}=\varphi_{n}\mbox{ on }\partial\Omega,
\]
and $W_{*,n}$ is the solution of 
\[
\left\{\begin{array}{rl}
\displaystyle -\,\Delta W_{*,n}+ \gamma\left(W_{*,n}-V_{*,n}\right)\mathbf{1}_{\HOmeg_0}=0 & \mbox{in }\Omega
\\*[.6em]
-\,\kappa\,\partial^2_{33}V_{*,n}+ \gamma\left(V_{*,n}-W_{*,n}\right)=0 & \mbox{in }\HOmeg_0=\omega_0\times(-L,L)
\\*[.4em]
W_{*,n}=\varphi_{n} & \mbox{on }\partial\HOmeg
\\*[.4em]
V_{*,n}(\cdot,\pm L)=\varphi_{n}(\cdot,\pm L) & \mbox{in }\omega_0.
\end{array}\right.
\]
The regularity Theorem~\ref{thm:Regularity-3d**} shows that there exists a constant $K$,
independent of $n$, such that 
\[
\left\Vert \nabla W_{n}\right\Vert _{C^{0,1/19}(\Omega_{\eps})}\leq K\left\Vert \varphi_{n}\right\Vert _{C^{1}(\bar{\Omega})}\leq K.
\]
Then, by virtue of the extension theorem for H{\"{o}}lder functions (see, e.g., \cite{MINTY-70}) there exists an extension 
$\xi_n$ of $\nabla W_n$ to the set $\Omega_1=\omega_1\times(-l,l)$, such that 
\[
\|\xi_n\|_{C^{0,{1/19}}(\Omega_1)^3}\leq K,
\]
with the same constant $K$. Note that the embedding 
$
C^{0,1/19}(\Omega_{1})\hookrightarrow C^{0,1/20}(\Omega_{1})
$
is compact. Therefore we can extract a convergent subsequence $\xi_{p}$ (and the associated sequence $\varphi_{p}$)
such that 
\[
\left\Vert \xi_{p}-\xi\right\Vert _{C^{0,1/20}(\Omega_{1})^3} \to 0.
\]
Thanks to the Ascoli-Arzel\`a Theorem, we can extract yet another subsequence  $\xi_{q},\varphi_{q}$ such that 
\[
\left\Vert \varphi_{q}-\varphi\right\Vert_{C^{0}(\bar{\HOmeg})}\to0  \text{ for some } \varphi \in C^{0,1}(\bar\Omega).
\]
Now, the homogenization Theorem~\ref{thm:hom2} shows that 
\[
\nabla W_{q}\rightharpoonup\nabla \uhom\mbox{ in }L^{2}(\Omega)^{3},
\]
where $\uhom$ is uniquely defined by 
\[
\left\{\begin{array}{rl}
\displaystyle -\,\Delta \uhom+ \gamma\left(\uhom-\vhom\right)\mathbf{1}_{\HOmeg_0}=0 & \mbox{in }\Omega
\\*[.6em]
-\,\kappa\,\partial^2_{33}\vhom+ \gamma\left(\vhom-\uhom\right)=0 & \mbox{in }\HOmeg_0=\omega_0\times(-L,L)
\\*[.4em]
\uhom=\varphi & \mbox{on }\partial\HOmeg
\\*[.4em]
\vhom(\cdot,\pm L)=\varphi(\cdot,\pm L) & \mbox{in }\omega_0.
\end{array}\right.
\]
A uniqueness argument shows that  $\nabla \uhom= \xi$ on $\Omega_{1}$. In particular,
\[
\left\Vert \nabla W_{q}-\nabla \uhom\right\Vert _{C^{0,1/20}(\Omega_{\eps})}=o(1).
\]
On the other hand, thanks to Theorem~\ref{thm:hom}, the linearity of the homogenized system $W_{*,q}-\uhom\in W^{2,4}(\Omega_{1})$ 
and the Sobolev embedding Theorem we have,
\[
\|\nabla W_{*,q} - \nabla \uhom \|_{C^{0,1/20}(\Omega_\eps)}\leq C \left\Vert \varphi_q -\varphi \right\Vert _{C^{0}\left(\HOmeg\right)}=o(1)
\]
for a constant $C$ independent of $q$, $\varphi$ and $\eps$. We therefore have obtained that
$$
c_0 \leq \left\Vert \nabla W_{q}-\nabla W_{*,q}\right\Vert _{C^{0,1/20}\left(\Omega_\eps\right)}
    = o(1),
$$
which is a contradiction. 
\qed\end{proof}

The third ingredient is an asymptotic representation formula for $\nabla\left(\uepsdef-\uepsnodef\right)$ in the set $\Geps$,
\begin{proposition}
\label{pro:def-repfor}
(a) Let $\vieps$ and $\viepsdef$ be
solutions to~\eqref{eq:u-no-defect} and~\eqref{eq:u-with-defect},
respectively for $\Phi_b^i=x_i$.  Then, there exist
a subsequence, still denoted by $\eps$, and a matrix-valued function $M^{*}\in L^2(\Omega,\mu)^{3\times 3}$, such that
\begin{equation}\label{eq:conM}
\lim_{\eps \to 0}\int_\Omega \frac{\mathbf{1}_{\Geps}}{|\Geps|}\frac{\partial \left(\viepsdef -\vieps\right)}{\partial x_j}\,\psi\,dx = \int_\Omega M^{*}_{ij}\,\psi\,d\mu, 
\end{equation}
for all  $\psi\in C^0(\Omega)$. Furthermore, the matrix $M^*$ is symmetric, and satisfies
\begin{equation}\label{eq:bound-mstar}
0\leq M^{*}(y) \leq \frac{(\gamma_1 -1)^2}{\gamma_1} \qquad \mbox{$\mu$-a.e.}
\end{equation}

\medskip
\noindent(b) We have, for all $w\in W^{2,4}(\Omega)$,
\[ 
\int_{\Omega
}\left(\aeps-\aepsdef\right) \nabla\left(\uepsdef-\uepsnodef\right)\cdot\nabla w\,dx
=\left|\Geps \right|\int_{\Omega} M^{*}\nabla \uhom\cdot\nabla w\,d\mu +o(|\Geps|) \left\Vert w\right\Vert _{W^{2,4}(\Omega)},
\]
where
$o\left(\left|\Geps \right|\right)/\left|\Geps\right|$
converges to zero uniformly for $\left\Vert \fib \right\Vert
_{C^1(\bar\Omega)}\leq1$.

\end{proposition}
\begin{remark}
It is natural to ask if the extraction of a subsequence is necessary, even in the case of bounded conductivities. 
The answer is positive. In the case of domains $\Geps$  shrinking to a point, there is an entire set of possible 
polarisation tensors for a given constant $\gamma_1$, delimited by the so-called Hashin-Strikman bounds, see 
e.g. \cite{LIPTON-93,MILTON-02,CAPDEBOSCQ-VOGELIUS-04,AMMARI-KANG-04}.
\end{remark}
\begin{remark}\label{rem:drfUniform}
We can use a compactness argument to derive from \eqref{eq:conM} that
\[
\lim_{\eps \rightarrow 0} \sup_{\|\psi\|_{C^{0,\nu}(\Omega)} \leq 1} \left|\int_\Omega \frac{\mathbf{1}_{\Geps}}{|\Geps|}\frac{\partial \left(\viepsdef -\vieps\right)}{\partial x_j}\,\psi\,dx - \int_\Omega M^{*}_{ij}\,\psi\,d\mu\right| = 0
\]
for any $0 < \nu \leq 1$. We remind the reader of Remark~\ref{rem.suppmu}: $\mu$ has compact support in $\Omega$. Indeed, if this fails for some $\nu$, then there exist some $c_0 > 0$ and a sequence $\psi_n \in C^{0,\nu}(\Omega)$ such that (along some subsequence $\eps \rightarrow 0$)
\[
\left|\int_\Omega \frac{\mathbf{1}_{\Geps}}{|\Geps|}\frac{\partial \left(\viepsdef -\vieps\right)}{\partial x_j}\,\psi_n\,dx - \int_\Omega M^{*}_{ij}\,\psi_n\,d\mu\right| \geq c_0. 
\]
On the other hand, as the embedding $C^{0,\nu}(\Omega) \hookrightarrow C^0(\Omega)$ is compact, we can assume that $\psi_n \rightarrow \psi_*$ in $C^0(\Omega)$, which implies a violation to \eqref{eq:conM}.

\end{remark}

\begin{proof}
The proof of this proposition is a variant of that of the main theorem of \cite{CAPDEBOSCQ-VOGELIUS-03A}.  
We provide the proof here for sake of completeness, adapted to the case at hand.

\medskip
\noindent(a) Introducing $\vieps$ and $\viepsdef$ solutions of solutions to~\eqref{eq:u-no-defect} and~\eqref{eq:u-with-defect},
respectively for a given $\Phi_b^i=x_i$, we know thanks to Theorem \ref{thm:hom}
and Proposition~\ref{pro:def-regularite} that
$$
\lim_{\eps\to0}\Big[\|\vieps - v^{i}_{*} \|_{L^2(\Omega)} + \|\nabla \vieps - \nabla v^{i}_{*} \|_{C^{0,1/20}(\Omega_\eps)}\Big]=0,
$$
where $v^{i}_{*}$ is the solution of $\uhom$ of \eqref{eq:hom-pro} corresponding to $\fib=x_i$.
It is immediate to verify that  $v^{i}_{*}=x_i$. In particular, we have
\begin{equation}
\left\Vert \vieps -x_i \right\Vert_{W^{1,\infty}(\Omega_\eps)}   \leq 
C \left(\left\Vert \nabla \vieps -e_{i}\right\Vert_{C^{0,1/20}(\Omega_\eps)} + \left\Vert  \vieps -x_{i}\right\Vert_{L^2(\Omega)}\right) = o(1),
\label{eq:bornesvdelta} 
\end{equation}
Here and in the sequel, $o(1)$ denotes any quantity going to zero with $\eps$, independently of $\fib$. On the other hand, 
using Lemma \ref{lem:Aubin-Nitsche} we have
\begin{equation}
\left\Vert \nabla \viepsdef -\nabla\vieps \right\Vert _{L^2(\Omega)}   \leq   C\,
\left|\Geps \right|^{1/2}\mbox{ and }
\left\Vert \viepsdef -\vieps \right\Vert_{L^2(\Omega)} \leq C\left|\Geps
\right|^{5/6}. \label{eq:bornesVdelta}
\end{equation}
Applying the Cauchy-Schwarz inequality yields
\begin{equation}
\frac{1}{|\Geps|}\,\int_{\Geps}\left| \nabla \viepsdef -\nabla\vieps \right|
\leq{\displaystyle
 \frac{1}{|\Geps|}\,}\left|\Geps \right|^{1/2}
\left\Vert \nabla \viepsdef -\nabla \vieps \right\Vert_{L^2(\Geps)}\leq C,
\label{eq:bornL1V}\end{equation}
thanks to \eqref{eq:bornesVdelta}. We may therefore extract a subsequence, still denoted by $\eps$,
such that
\begin{equation}
|\Geps |^{-1}\,\mathbf{1}_{\Geps}\left(1-\cepsdef\right)\,\frac{\partial}{\partial x_j} \left( \viepsdef - \vieps \right)\, dx 
 \longrightarrow 
d\mathcal{M}_{ij},\label{convdM}
\end{equation}
where $\mathcal{M}_{i,j},\;1\leq i,j\leq2$ is a Borel measure with support in $\bar\omega_0\times[-l,l]$. The above convergence
results hold in the  weak-$*$ topology of ${\mathcal M}(\bar\omega_0\times[-l,l])$ (the set of Radon measures on $\bar\omega_0\times[-l,l]$).
Note that  $1-\gamma_1 \in C^{0}(\Omega)$. We see that, for any $f\in C^{0}(\bar{\Omega})$,
\begin{eqnarray*}
\left|{\displaystyle \int_{\Omega}f\, d\mathcal{M}_{ij}}\right| &
= & \left|\lim_{\eps\rightarrow0}|\Geps|^{-1}\,
\int_{\Geps}\left(1-\gamma_1\right)\left(
 \frac{\partial}{\partial x_j} \viepsdef - \frac{\partial}{\partial x_j} \vieps \right)\, fdx\right|\\
 & \leq &\limsup_{\eps\to 0 }\frac{||1-\gamma_1||_{L^\infty(\Omega)}}{\sqrt{|\Geps|}} 
 \, ||\nabla \viepsdef - \nabla \vieps ||_{L^2(\Geps)}\,\left(\frac{1}{|\Geps|}
 {\displaystyle \int_{\Geps}f^{2}}dx\right)^{1/2}.\end{eqnarray*}
Since $\left|\Geps \right|^{-1}{\displaystyle
\int_{\Geps }f^2dx\rightarrow\int_{\Omega_0}f^2 d\mu}$, we
conclude, using~\eqref{eq:bornesVdelta} that
\begin{eqnarray}
\left|{\displaystyle \int_{\Omega} f \, d\mathcal{M}_{ij}}\right|
& \leq & C\,\left(\int_{\Omega_0}\left|f\right|^2d\mu\right)^{1/2}.\label{Riesz}
\end{eqnarray}
It follows that $d\mathcal{M}_{ij}$ is absolutely continuous with
respect to $\mu$ and thus, there exists a $3\times3$ matrix valued function
$M^* \in L^2(\Omega,\mu)^{3 \times 3}$ such that 
$$
\int_{\Omega_0} f d\mathcal{M}_{ij} =
\int_{\Omega_0} {M}^{*}_{ij}fd\mu.
$$
Let us now turn to the properties of the matrix valued function $M^*$, following  \cite{CAPDEBOSCQ-VOGELIUS-03A}. The matrix $M^*$  is characterised by 
$$
\int_\Omega M^{*}_{ij}\, \psi \,d\mu =\frac{1}{|\Geps|} \int_{\Geps} \left(1-\gamma_1\right)\nabla \left(\viepsdef-\vieps \right)\cdot \nabla x_j \,\psi\, dx +  o(1)
$$
for any $\psi\in C^{0}\left(\bar{\HOmeg}\right)$. Note that $\viepsdef-\vieps$ is the solution in $H^1_0(\HOmeg)$ of 
\begin{equation}\label{eq:correc-x}
\divop \left(\aepsdef \nabla \left(\viepsdef-\vieps \right)\right) = \divop \left((1-\gamma_1)\mathbf{1}_{\Geps}\nabla \vieps \right).
\end{equation}
Using this identity together with Proposition~\ref{pro:def-regularite} and Lemma~\ref{lem:Aubin-Nitsche},  we obtain
\begin{eqnarray*}
\int_\Omega M^{*}_{ij} \psi d\mu  &=& \frac{1}{|\Geps|} \int_{\Geps} \left(1-\gamma_1\right)\nabla\left( \left(\viepsdef-\vieps \right)\psi\right)\cdot \nabla \vjeps  dx +  o(1) \\
&=& \frac{1}{|\Geps|} \int_{\HOmeg} \aepsdef \nabla \left( \left(\viepsdef-\vieps \right)\psi\right)\cdot \nabla \left(\vjepsdef-\vjeps \right) dx +  o(1) \\
&=& \frac{1}{|\Geps|} \int_{\HOmeg} \aepsdef \nabla \left(\viepsdef-\vieps \right) 
\cdot \nabla \left(\vjepsdef-\vjeps \right) \psi\, dx +  o(1).
\end{eqnarray*}
Under this last form, it is apparent that $M^{*}_{ij}=M^{*}_{ji}$. 
Furthermore, given $\xi\in\mathbb{R}^3$, introducing $\Veps = \sum_{i=1}^3  (\viepsdef-\vieps) \xi_i$, this last identity yields
\begin{equation}\label{eq:Mstar-1}
\int_\Omega M^{*}\xi\cdot\xi \psi d\mu = \int_\Omega \sum_{1 \leq i,j \leq 3}M^{*}_{ij}\xi_i\,\xi_j \psi d\mu 
=\frac{1}{|\Geps|} \int_{\HOmeg} \aepsdef \left|\nabla \Veps\right|^2 \psi \,dx +o(1).
\end{equation}
This shows that $M^{*}\geq0$, $\mu$-almost everywhere. Alternatively, from \eqref{eq:correc-x} we derive, using 
Proposition~\ref{pro:def-regularite} and Lemma~\ref{lem:Aubin-Nitsche}, that for $\psi\geq0$,
\begin{eqnarray}
\int_\Omega \aepsdef \left| \nabla \Veps\right|^2  \psi dx &=& 
\int_\Omega \aepsdef \nabla \Veps\cdot \nabla \left(\Veps  \psi\right)dx + o(|\Geps|) \nonumber\\
&=& \int_\Omega \left((1-\gamma_1)\mathbf{1}_{\Geps} \right) \left(\sum_{i=1}^3  \nabla \vieps \xi_i\right)\cdot \nabla \left(\Veps  \psi\right)dx + o(|\Geps|) \nonumber \\
&=& \int_\Omega \left((1-\gamma_1)\mathbf{1}_{\Geps} \right) \xi \psi  \cdot \nabla \Veps  dx + o(|\Geps|) \label{eq:VW1}\\ 
&\leq& \left(\int_\Omega \frac{(1-\gamma_1)^2}{\gamma_1} \mathbf{1}_{\Geps} |\xi|^2 \psi\, dx \right)^{1/2}
\left(\int_\Omega \aepsdef \left|\nabla \Veps\right|^2  \psi \, dx\right)^{1/2} +   o(|\Geps|),\nonumber
\end{eqnarray}
where we have used we have used \eqref{eq:bornesvdelta} to derive \eqref{eq:VW1}. This shows that
$M^{*}\leq \frac{(1-\gamma_1)^2}{\gamma_1}$, $\mu$-almost everywhere.

\medskip
\noindent(b) Let us now show that
\begin{equation}\label{eq:conv-w}
\left|\Geps \right|^{-1}\int_{\Omega} \left(\aeps-\aepsdef \right)
\nabla \left( \uepsdef-\uepsnodef\right) \cdot\nabla w dx=\left|\Geps
\right|^{-1}\int_{\Omega} M^{*}\nabla \uhom \cdot \nabla w d\mu 
+\err,
\end{equation}
where $\err$ satisfies 
$$
|\err| = o(1)\left| \Geps \right| \,\|\fib\|_{C^{1}(\bar{\HOmeg})}\| w\|_{W^{2,4}(\Omega)}.
$$
Identity \eqref{eq:conv-w} can be interpreted as a `separation of scales' result.  On the left-hand side, the perturbation induced by the defect is present both in the `microscopic' term 
$\aeps-\aepsdef$ and in the `macroscopic' term $\nabla \uepsdef - \nabla \uepsnodef$. On the right-hand side, 
the `macroscopic' term $\nabla \uhom$ is independent of the defect. To prove this result, we follow the 
adaptation of the oscillating test function method in homogenization of Murat \& Tartar \cite{MURAT-TARTAR-97} 
proposed in \cite{CAPDEBOSCQ-VOGELIUS-03A}. In the following computation,  $\viepsdef -\vieps$ plays the role
of the corrector function in homogenization. By a chain of integration by parts, 
we will transfer the derivatives from $\nabla (\uepsdef -  \uepsnodef)$ to $\nabla(\viepsdef -\vieps)$, and then pass 
to the limit.
\par 
To express the left-hand side of \eqref{eq:conv-w} in terms of $M^*$, we introduce $\viepsdef -\vieps$ in the 
computation as follows. 
Using Einstein summation convention for the index $i$ and noting that $v^i_* = x_i$, we have
\begin{eqnarray}\label{eq:compen-1}
 \int_{\Omega} \left(\aeps-\aepsdef \right)
\nabla \left( \uepsdef-\uepsnodef\right) \cdot\nabla w dx  
&=& \int_{\Geps} \left(1-\gamma_1 \right) 
\nabla \left( \uepsdef-\uepsnodef\right) \cdot \nabla v^{i}_{*}\,\frac{\partial w}{\partial x_i}  dx \\
&=& \int_{\Geps} \left(1-\gamma_1 \right) 
\nabla \left( \uepsdef-\uepsnodef\right) \cdot \nabla \vieps \, \frac{\partial w }{\partial x_i}   dx + \err_1, \nonumber
\end{eqnarray}
with
\begin{eqnarray*}
|\err_1|& \leq & C \|\nabla \vieps -\nabla v^{i}_{*}\|_{C^0(\Omega_\eps)}  \|\nabla w \|_{L^\infty(\Geps)} 
\left| \Geps \right|^{1/2}  \left\|\nabla \left( \uepsdef-\uepsnodef\right)\right\|_{L^2(\Geps)} \\
      & = & o(1) \left| \Geps \right|\,\|\fib\|_{C^{1}(\bar{\HOmeg})}\| w\|_{W^{2,4}(\Omega)},
\end{eqnarray*}
thanks to \eqref{eq:bornesVdelta} and Lemma~\ref{lem:Aubin-Nitsche}.
Continuing the transformation, we write
\begin{eqnarray}\label{eq:compen-2}
 \int_{\Geps} \left(1-\gamma_1 \right) 
\nabla \left( \uepsdef-\uepsnodef\right) \cdot \nabla \vieps \, \frac{\partial w }{\partial x_i}   dx 
 &=&
\int_{\Geps} \left(1-\gamma_1 \right) 
\nabla\left(\frac{\partial w}{\partial x_i} \left( \uepsdef-\uepsnodef\right)\right) \cdot \nabla \vieps    dx + \err_2 \\
&=& \int_{\Omega} \aepsdef 
\nabla\left(\frac{\partial w}{\partial x_i} \left( \uepsdef-\uepsnodef\right)\right) \cdot \nabla \left(\viepsdef- \vieps\right)    dx + \err_2, \nonumber 
\end{eqnarray}
with 
\begin{eqnarray*}
|\err_2| \leq  C \| \nabla \vieps \|_{C^0(\Omega_\eps)}  \|\nabla^2 w \|_{L^2(\Geps)} 
\left\| \left( \uepsdef-\uepsnodef\right)\right\|_{L^2(\Geps)} 
      & \leq & C \left| \Geps \right|^{\frac{13}{12}} \,\|\fib\|_{C^{1}(\bar{\HOmeg})} \| w\|_{W^{2,4}(\Omega)}\\
      & = &  o(1) \left| \Geps \right| \,\|\fib\|_{C^{1}(\bar{\HOmeg})}\| w\|_{W^{2,4}(\Omega)},
\end{eqnarray*}
thanks to Lemma~\ref{lem:Aubin-Nitsche} and H\"{o}lder's inequality.
To remove the $\uepsdef$ term, we write
\begin{eqnarray}\label{eq:compen-3}
&& \int_{\Omega} \aepsdef 
\nabla\left(\frac{\partial w}{\partial x_i} \left( \uepsdef-\uepsnodef\right)\right) \cdot \nabla \left(\viepsdef- \vieps\right)    dx\\
 &=&
\int_{\Omega} \aepsdef  
\nabla\left( \uepsdef-\uepsnodef\right) \cdot \nabla \left( \frac{\partial w}{\partial x_i}  \left(\viepsdef- \vieps\right) \right) dx + \err_3 \nonumber\\
&=& \int_{\Omega} \left(\aeps - \aepsdef\right)  
\nabla \uepsnodef \cdot \nabla \left( \frac{\partial w}{\partial x_i}  \left(\viepsdef- \vieps\right) \right) dx + \err_3, \nonumber
\end{eqnarray}
with 
\begin{eqnarray*}
|\err_3|& \leq & C \int_{\Omega} \left|\nabla^2 w\right| \Big(\left|\uepsdef-\uepsnodef\right| \left|\nabla \left(\viepsdef- \vieps\right)\right| + \left|\nabla \left(\uepsdef-\uepsnodef\right)\right| \left| \viepsdef- \vieps\right| \Big)\,dx  \\
        & \leq & C \| w\|_{W^{2,4}(\Omega)} \left| \Geps \right|^{\frac{13}{12}}  
\,\|\fib\|_{C^{1}(\bar{\HOmeg})}\\
        & = &  o(1) \left| \Geps \right| \,\|\fib\|_{C^{1}(\bar{\HOmeg})}\| w\|_{W^{2,4}(\Omega)},
\end{eqnarray*}
using Lemma~\ref{lem:Aubin-Nitsche}, together with the interpolation inequality
$$
\|f\|_{L^4(\Omega)}\leq C\|f\|^{1/4}_{L^2(\Omega)} \|f\|^{3/4}_{L^6(\Omega)}.
$$ 
Expanding this last expression, we have
\begin{eqnarray}\label{eq:compen-4}
&&\int_{\Omega} \left(\aeps - \aepsdef\right)  
\nabla \uepsnodef \cdot \nabla \left( \frac{\partial w}{\partial x_i}  \left(\viepsdef- \vieps\right) \right) dx\\
 &= & \int_{\Geps} \left(1 - \cepsdef\right)  
\nabla \uepsnodef \cdot \nabla \left( \viepsdef- \vieps \right) \frac{\partial w}{\partial x_i} dx + \err_4, \nonumber \end{eqnarray}
with
\begin{eqnarray*}
|\err_4|& \leq & C \| \viepsdef - \vieps\|_{L^2(\Geps)}  \|\nabla \uepsnodef \|_{L^\infty(\Geps)} 
\left| \Geps \right|^{1/2} \| w\|_{W^{1,\infty}(\Omega)}\\
      & = & o(1) \left| \Geps \right|\,\|\fib\|_{C^{1}(\bar{\HOmeg})}\| w\|_{W^{2,4}(\Omega)}.
\end{eqnarray*}
Using the convergence of $\uepsnodef$ to $\uhom$, we get
\begin{eqnarray}\label{eq:compen-5}
&&\int_{\Geps} \left(1 - \cepsdef\right)  
\nabla \uepsnodef \cdot \nabla \left(\viepsdef- \vieps \right) \frac{\partial w}{\partial x_i} dx\\
 &=&
\int_{\HOmeg} 
\nabla \uhom \cdot \left(\mathbf{1}_{\Geps}\left(1 - \cepsdef\right)  \nabla  \left(\viepsdef- \vieps\right)\right) \frac{\partial w}{\partial x_i} dx + \err_5, \nonumber
\end{eqnarray}
with
\begin{eqnarray*}
|\err_5|& \leq & C \| \left(\mathbf{1}_{\Geps}  \nabla  \left(\viepsdef- \vieps\right)\right) \frac{\partial w}{\partial x_i} \|_{L^1(\Geps)} \| \nabla \uepsnodef -\nabla \uhom \|_{L^\infty(\Geps)} \\
      & \leq & C \left| \Geps \right| \| w\|_{W^{2,4}(\Omega)} \| \nabla \uepsnodef -\nabla \uhom \|_{L^\infty(\Geps)}.
\end{eqnarray*}
Note that thanks to Proposition~\ref{pro:def-regularite}, we know that 
$$
\| \nabla \uepsnodef -\nabla \uhom \|_{L^\infty(\Geps)}=o(1) \|\fib\|_{C^{1}(\bar{\HOmeg})},
$$
therefore $|\err_5|=  o(1) \left| \Geps \right|\,\|\fib\|_{C^{1}(\bar{\HOmeg})}\| w\|_{W^{2,4}(\Omega)}$. Finally, by Remark \ref{rem:drfUniform}, 
\begin{eqnarray}
 \frac{1}{\left| \Geps \right|}\int_{\HOmeg} 
\nabla \uhom \cdot \left(\mathbf{1}_{\Geps}\left(1 - \cepsdef\right)  \nabla  \left(\viepsdef- \vieps\right)\right) \frac{\partial w}{\partial x_i} dx 
&=& 
\int_{\Omega} \nabla \uhom \cdot \left(M^*\right)^{T} \nabla w \,d\mu  + \err_6 
\label{eq:compen-6}
\end{eqnarray}
where 
$$
|\err_6| =  o(1) \| \nabla \uhom \|_{C^{0,1/20}(\Omega)} \| \nabla w \|_{C^{0,1/20}(\Omega_\eps)}
        = o(1)  \,\|\fib\|_{C^{1}(\bar{\HOmeg})}\| w\|_{W^{2,4}(\Omega)}.
$$

Combining \eqref{eq:compen-1},\eqref{eq:compen-2},\eqref{eq:compen-3},\eqref{eq:compen-4},\eqref{eq:compen-5} and \eqref{eq:compen-6}, we obtain \eqref{eq:conv-w}.
\qed\end{proof}
\noindent{\it Proof of Theorem~\ref{thm:defect}. }
A straight-forward integration by parts shows that
\begin{equation}
\DN(\fib)=\int_{\Geps }
\left(\cepsdef-1\right)\nabla\uepsdef\cdot\nabla\uepsnodef\,dx,
\label{eq:formuleR1}
\end{equation}
which we rewrite in the form
\begin{eqnarray*}
\DN(\fib) & = & \int_{\Geps }\left(\cepsdef-1\right)\nabla \uhom \cdot\nabla \uhom\,dx\\
&& +  \int_{\Geps}\left(\aepsdef-\aeps\right)\nabla\left(\uepsdef-\uepsnodef\right)\cdot\nabla \uhom\,dx+r_{\eps}
	,\\
\mbox{with}\quad r_{\eps} & = &
\int_{\Geps}\left(\cepsdef -1 \right)
\left(\nabla \uepsnodef\cdot\nabla \uepsnodef-\nabla \uhom\cdot \nabla \uhom\right)dx\\
	&& + \int_{\Geps}\left(\aepsdef-\aeps\right)\nabla\left(\uepsdef-\uepsnodef\right)\cdot\nabla (W_\eps - \uhom)\,dx 
	.
\end{eqnarray*}
By Proposition~\ref{pro:def-regularite}, Lemma \ref{lem:Aubin-Nitsche}, Theorem \ref{thm:hom} and \eqref{est.GWeps}, we have
$$
|r_\eps|\leq C\,|\Geps|\,\|\nabla\uepsnodef-\nabla\uhom\|_{C^0(\Omega_\eps)^3}\,\|\fib\|_{C^1(\bar\Omega)}=o(|\Geps|).
$$
On the other hand, Proposition~\ref{pro:def-repfor} shows that
$$
\int_{\Geps}\left(\aepsdef-\aeps\right)\nabla\left(\uepsdef-\uepsnodef\right)\cdot\nabla \uhom\,dx = -
\left| \Geps\right|\int_{\Omega} M^{*}\nabla \uhom \cdot\nabla \uhom\,d\mu +o(|\Geps|) \left\Vert \uhom\right\Vert _{W^{2,4}(\Omega)},
$$
and this establishes the representation formula given by Theorem~\ref{thm:defect}, thanks to Theorem~\ref{thm:hom} with 
$$
M_{ij}=(\gamma_1 -1 )\delta_{ij} - M^{*}_{ij}.
$$ 
The bounds \eqref{eq:bound-mstar} on $M^*$ imply the announced bounds on $M$.
\qed

To conclude this section, we now provide an alternative characterisation of $M^{*}$, following \cite{CAPDEBOSCQ-VOGELIUS-06}.
\begin{proposition}\label{pro:energy-mstar}
Let $M^{*}$ be the polarisation tensor introduced by Proposition~\ref{pro:def-repfor}. 
Let $\psi$ be a uniformly positive, smooth function on $\Omega$, and $\xi\in\mathbb{R}^3$. 
Then $M^*$ satisfies
\begin{align*}
\int_\Omega M^*\xi\cdot\xi\,\psi\, d\mu 
	&=   \int_{\HOmeg} \frac{\left(\gamma_1 -1\right)^2}{\gamma_1}\,|\xi|^2\,\psi\, d\mu\\
		&\qquad\qquad - \frac{1}{\left|G_\eps\right|}\min_{w\in H^1_0(\HOmeg)} \int_\Omega \aepsdef \left| \nabla w 
+ \frac{\gamma_1 -1}{\gamma_1} \mathbf{1}_{\Geps} \xi \right|^2 \psi \,dx +o(1),
\end{align*}
where $o(1)$ may depend on $\psi$ but goes to zero with $\eps$.
\end{proposition}

\begin{proof}
Let $\zeta_\eps$ be the solution in $H^1_0\left(\Omega\right)$ of
\begin{equation}\label{eq:defweps}
 \divop\left(\aepsdef\psi\nabla \zeta_\eps\right) = \divop\left(\psi\left(1-\gamma_1\right) \mathbf{1}_{\Geps} \xi\right).
\end{equation}
Note that $\zeta_\eps$ is the unique minimizer of 
$$
\min_{w\in H^1_0(\HOmeg)} \int_\Omega \aepsdef \left| \nabla w + \frac{\gamma_1 -1}{\gamma_1} \mathbf{1}_{\Geps} \xi \right|^2 \psi \,dx.
$$
In particular,
$$
\int_\Omega \aepsdef \left( \nabla \zeta_\eps + \frac{\gamma_1 -1}{\gamma_1} \mathbf{1}_{\Geps} \xi \right)\cdot \nabla \zeta_\eps \psi \,dx=0,
$$
and therefore, as $\gamma_1\in C^{0}(\Omega)$,
\begin{eqnarray}
&&\int_\Omega \aepsdef \left| \nabla \zeta_\eps + \frac{\gamma_1 -1}{\gamma_1} \mathbf{1}_{\Geps} \xi \right|^2 \psi \,dx\nonumber\\
	&&\qquad\qquad = - \int_\Omega \aepsdef \left| \nabla \zeta_\eps \right|^2 \psi \,dx + \int_{\Geps} \frac{\left(\gamma_1 -1\right)^2}{\gamma_1}\,|\xi|^2\,\psi\, dx\nonumber\\
	&&\qquad\qquad - \int_\Omega \aepsdef \left| \nabla \zeta_\eps \right|^2 \psi \,dx + \int_{\Omega} \frac{\left(\gamma_1 -1\right)^2}{\gamma_1}\,|\xi|^2\,\psi\, d\mu +o(1).
\label{eq:minident}
\end{eqnarray}

Let us prove that $\zeta_\eps$ satisfies an estimate similar to that of Lemma~\ref{lem:Aubin-Nitsche}, namely
\begin{equation}\label{eq:A-N-weps}
\left\|\nabla \zeta_\eps\right\|_{L^2(\Omega)^3} + \left|\Geps\right|^{-1/3} \|\zeta_\eps\|_{L^2(G_\eps )}\leq C \left|\Geps\right|^{1/2},
\end{equation}
where $C$ may depend on $\psi$, but is independent of $\eps$. Testing \eqref{eq:defweps} against $\zeta_\eps$, and integrating by parts, we obtain
\begin{eqnarray*}
\left\|\nabla \zeta_\eps\right\|^2_{L^2(\Omega)^3} &\leq& \frac{1}{\min \psi} \int_\Omega  \aepsdef\psi \left|\nabla \zeta_\eps\right|^2 dx \\
                                              &  = & \frac{1}{\min \psi} \int_{\Geps}  \left(1-\gamma_1\right) \xi \cdot \nabla \zeta_\eps\\
                                              & \leq& C \left|\Geps\right|^{1/2} \left\|\nabla \zeta_\eps\right\|_{L^2(\Omega)^3},  
\end{eqnarray*}
using the Cauchy-Schwarz inequality. On the other hand, using H\"{o}lder's inequality followed by Poincar\'{e}-Sobolev's inequality,
$$
\|\zeta_\eps\|_{L^2(G_\eps )} 
	\leq \left|\Geps\right|^{1/3}\,\|\zeta_\eps\|_{L^6(\Omega)}   
	\leq C\left|\Geps\right|^{1/3}\|\nabla \zeta_\eps\|_{L^2(\Omega)} ,
$$
and \eqref{eq:A-N-weps} is established. Introducing $\Veps = \sum_{i=1}^3  (\viepsdef-\vieps) \xi_i$, the identity \eqref{eq:VW1} shows that
\begin{eqnarray*}
\int_\Omega \aepsdef \left| \nabla \Veps\right|^2  \psi dx 
&=& \int_\Omega \left((1-\gamma_1)\mathbf{1}_{\Geps} \right) \xi \psi  \cdot \nabla \Veps  dx + o(|\Geps|)\\
& =& \int_\Omega  \aepsdef\psi \nabla \zeta_\eps \cdot \nabla \Veps dx + o(|\Geps|).
\end{eqnarray*}
On the other hand, using Proposition~\ref{pro:def-regularite}, followed by \eqref{eq:A-N-weps}, used twice,
\begin{eqnarray*}
\int_\Omega \aepsdef \left| \nabla \zeta_\eps\right|^2  \psi dx &=& \int_\Omega \left((1-\gamma_1)\mathbf{1}_{\Geps} \right) \xi \psi  \cdot \nabla \zeta_\eps  dx \\
&=& \int_\Omega \left((1-\gamma_1)\mathbf{1}_{\Geps} \right) \left(\sum_{i=1}^3  \nabla \vieps \xi_i\right)  \cdot \nabla \zeta_\eps  \psi dx + o(|\Geps|) \\
&=& \int_\Omega \left((1-\gamma_1)\mathbf{1}_{\Geps} \right) \left(\sum_{i=1}^3  \nabla \vieps \xi_i\right)  \cdot \nabla \left( \zeta_\eps  \psi \right)  dx + o(|\Geps|) \\
& =& \int_\Omega  \aepsdef  \psi \nabla \Veps \cdot \nabla \zeta_\eps  dx + o(|\Geps|).
\end{eqnarray*}
We have obtained that
$$
\int_\Omega \aepsdef \left| \nabla \Veps \right|^2  \psi dx = \int_\Omega \aepsdef \left| \nabla \zeta_\eps\right|^2  \psi dx +o(|\Geps|).
$$
The conclusion then follows directly from the above identity, \eqref{eq:Mstar-1} and \eqref{eq:minident}.
\qed\end{proof}
\section{\label{sec:hom} Proof of the homogenization result}
This section provides a proof of Theorem~\ref{thm:hom}.
The proof is divided in three steps.
\par\medskip\noindent
{\it First step:} Derivation of the homogenization problem \eqref{eq:hom-pro}.
\par\noindent
This step uses the same ingredients of \cite{BRIANE-TCHOU-01} and \cite{BELLIEUD-BOUCHITTE-98}, 
but the boundary conditions are different. Choosing $\Weps-\fib\in H^1_0(\HOmeg)$ as a test function 
in equation (\ref{eq:u-no-defect}), we obtain
\[
\int_\HOmeg \aeps |\nabla \Weps|^2\,dx=\int_\HOmeg \aeps\nabla \Weps\cdot\nabla\fib\,dx
=\int_{\Ceps}\ceps\nabla \Weps\cdot\nabla\fib\,dx+\int_{\HOmeg\setminus \Ceps}\nabla \Weps\cdot\nabla\fib\,dx.
\]
This combined with (\ref{eq:Cond1}), (\ref{alpeps}) and the Cauchy-Schwarz inequality yields
\[
\int_\HOmeg \aeps |\nabla \Weps|^2\,dx\leq c\,\|\nabla\fib\|_{L^\infty(\HOmeg)}\left(\int_\HOmeg \aeps |\nabla \Weps|^2\,dx\right)^{1/2},
\]
hence the energy estimate
\begin{equation}\label{ener.est}
\int_\HOmeg \aeps |\nabla \Weps|^2\,dx\leq c\,\|\nabla\fib\|^2_{L^\infty(\HOmeg)}.
\end{equation}
Estimate (\ref{ener.est}) implies that $\nabla \Weps$ is bounded in $L^2(\HOmeg)$. 
Due to the Dirichlet boundary condition and the regularity of $\HOmeg$, the sequence $\Weps$ is bounded in $H^1(\HOmeg)$, 
and thus converges weakly to some function $\uhom$ in $H^1(\HOmeg)$ up to a subsequence.
\par
Integrating along vertical lines, we can prove (see \cite{BELLIEUD-BOUCHITTE-98}, \cite{BRIANE-TCHOU-01} 
for further details) that the rescaled function $\Veps=\,(\pi \reps^2)^{-1}\,1_{\Ceps}\,\Weps$ converges
 weakly-$*$ in $\mathcal{M}(\Omega_0)$ (i.e., in the sense of measures on $\HomO$) to some function $\vhom\in H^1\big((-L,L);L^2(\omega_0)\big)$. 
Moreover, the uniform repartition of the highly conductive cylinders $Q_{m,n,\eps}$ and 
the continuity of $\fib$ imply that $\vhom$ inherits of the Dirichlet boundary condition on $\omega_0\times\{\pm L\}$.
\par
On the one hand, by (\ref{alpeps}) and (\ref{ener.est}) we have
\[
\aeps\,\partial_3 \Weps\,1_{\Ceps}=\partial_3\left( \pi \reps^2\,\ceps \,\Veps\right)\;\rightharpoonup\;\kappa\,\partial_3 \vhom
\quad\mbox{weakly-$*$ in }\mathcal{M}(\HomO),
\]
and there is no transverse diffusion induced by the cylinders $Q_{m,n,\eps}$ (see \cite{BRIANE-TCHOU-01}).
Moreover, due to the energy estimate (\ref{ener.est}) and the periodicity of $a_\eps$, there is no 
concentration effect of $a_\eps\nabla \Weps$ on $\partial\HomO$. Therefore, we obtain the convergences of the flux
\begin{equation}\label{flux.conv}
\left\{\begin{array}{llll}
\aeps\nabla \Weps\, \csigeps & \rightharpoonup & \nabla \uhom & \mbox{weakly-$*$ in }\mathcal{M}(\bar{\HOmeg})
\\*[.4em]
\aeps\nabla \Weps\,(1- \csigeps) & \rightharpoonup
& \displaystyle \kappa\,\partial_3 \vhom\,e_3 & \mbox{weakly-$*$ in }\mathcal{M}(\bar{\HOmeg}_0).
\end{array}\right.
\end{equation}
On the other hand, the sequence $ \csigeps$ defined by (\ref{eq:defceps}) satisfies the convergences (see~\cite{BRIANE-TCHOU-01})
\begin{equation}\label{conv.hVeps}
 \csigeps\rightharpoonup\, 1\;\;\mbox{weakly in }H^1(\HOmeg)\quad\mbox{and}\quad
|\nabla \csigeps|^2\,\rightharpoonup\,\gamma\;\;\mbox{weakly-$*$ in }\mathcal{M}(\bar{\HOmeg}_0).
\end{equation}
Then, thanks to \cite[Lemma~2]{BRIANE-TCHOU-01} combined with (\ref{eq:Cond1}), and again using the fact that there is no 
concentration effect on $\partial\HomO$, we get the convergence
\begin{equation}\label{nloc.eff}
\nabla \Weps\cdot\nabla \csigeps\;\rightharpoonup\;\gamma\left(\uhom-\vhom\right)\quad\mbox{weakly-$*$ in }\mathcal{M}(\bar{\HOmeg}_0),
\end{equation}
which induces the non-local effect in the homogenization process.
\par
Let $\psi_0,\psi_1\in C^1_c(\HOmeg)$. Choosing $\psi_0\, \csigeps+\psi_1\,(1- \csigeps)$ as a test function 
in equation \eqref{eq:hom-pro}, and passing to the limit as $\eps\to 0$ thanks to the 
convergences (\ref{flux.conv}), (\ref{conv.hVeps}), (\ref{nloc.eff}), we obtain the equality 
\[
\int_\HOmeg\nabla \uhom\cdot\nabla\psi_0\,dx+\int_{\HomO}\kappa\,\partial_3 \vhom\,\partial_3\psi_1\,dx
+\int_{\HomO}\gamma\left(\uhom-\vhom\right)\left(\psi_0-\psi_1\right)dx=0,
\]
which corresponds to the weak formulation of problem \eqref{eq:hom-pro}.
\par\medskip{}\noindent
{\it Second step:} Proof of the local regularity of $\uhom$ and $\vhom$.

Rewrite the first and the second equations of problem \eqref{eq:hom-pro} as
\begin{equation}\label{equW*}
-\,\Delta \uhom = -\gamma\,1_{\HomO}\,(\uhom  - \vhom)\quad\mbox{in }\HOmeg,
\end{equation}
and
\begin{equation}\label{equV*}
-\,\kappa\,\partial^2_{33}\vhom+\gamma\,\vhom=\gamma\,\uhom\quad\mbox{in }\HomO=\omega_0\times(-L,L).
\end{equation}
As $\uhom \in H^1(\Omega)$ and $\fib \in C^1(\bar\Omega)$, $\vhom \in H^1(\Omega_0)$. 
By the Sobolev embedding theorem, the right-hand side of \eqref{equW*} is in $L^6(\Omega)$. 
Thus, as $\partial\Omega$ is Lipschitz thus satisfies the exterior cone condition, the de Giorgi - Nash - 
Moser estimate (see e.g. \cite[Theorems 8.22, 8.27]{GILBARG-TRUDINGER-83}), $\uhom \in C^{0,\alpha}(\bar\Omega)$ for some $\alpha \in (0,1)$. 
Consequently, $\vhom, \partial_{x_3} \vhom, \partial_{x_3}^2 \vhom \in C^{0,\alpha}(\bar\Omega_0)$.

Going back to \eqref{equW*}, the right-hand side belongs to $L^\infty(\Omega)$. Thus, $\uhom \in W^{2,p}_{\rm loc}(\Omega)$ for any $p > 2$. 
Bootstrapping between \eqref{equW*} and \eqref{equV*} we obtain $\uhom \in C^\infty_{\rm loc}(\Omega_0) \cap C^\infty_{\rm loc}(\Omega \setminus \Omega_0)$ 
and $\vhom \in C^\infty_{\rm loc}(\Omega_0)$.

\par\medskip\noindent
{\it Third step:} Proof of the corrector result (\ref{cor.res}).

\par\smallskip{}\noindent
Denote by $E_\eps$ the left-hand side of (\ref{cor.res}). By equation (\ref{eq:u-no-defect}) we have
\[
\int_\HOmeg \aeps |\nabla \Weps|^2\,dx=\int_\HOmeg \aeps\nabla \Weps\cdot\nabla\fib\,dx.
\]
Hence, by the definition (\ref{eq:defceps}) of $ \csigeps$, it follows that
\[
\begin{array}{ll}
E_\eps =  & \displaystyle \int_\HOmeg \aeps\nabla \Weps\cdot\nabla\fib\,dx
+\int_{\HomO}|\nabla \csigeps|^2\,(\uhom-\vhom)^2\,dx
\\*[.4em]
& \displaystyle +\int_\HOmeg( \csigeps)^2\,|\nabla \uhom|^2\,dx
+\int_{\HomO}\aeps(1- \csigeps)^2\,(\partial_3 \vhom)^2\,dx
\\*[.4em]
& \displaystyle -\,2\int_{\HomO}\nabla \Weps\cdot\nabla \csigeps\,(\uhom-\vhom)\,dx
-\,2\int_{\HOmeg} \csigeps\,\nabla \Weps\cdot\nabla \uhom\,dx
\\*[.4em]
& \displaystyle -\,2\int_{\HomO} \aeps(1- \csigeps)\,\partial_3 \Weps\,\partial_3 \vhom\,dx+o(1).
\end{array}
\]
Then, passing to the limit as $\eps\to 0$ thanks to the convergences (\ref{flux.conv}), (\ref{conv.hVeps}), (\ref{nloc.eff}) combined 
with the continuity of the functions $\uhom-\vhom$ and $\partial_3 \vhom$, we obtain
\[
\begin{array}{ll}
\displaystyle \limsup_{\eps\to 0}E_\eps \leq 
& E_0 {}= \displaystyle\int_\HOmeg\nabla \uhom\cdot\nabla\fib\,dx+\int_{\HomO}\kappa\,\partial_3 \vhom\,\partial_3\fib\,dx
\\*[.4em]
& \displaystyle +\int_{\HomO}\gamma\,(\uhom-\vhom)^2\,dx+\,\int_\HOmeg|\nabla \uhom|^2\,dx+\int_{\HomO}\kappa\,(\partial_3 \vhom)^2\,dx
\\*[.4em]
& \displaystyle -\,2\int_{\HomO}\gamma\,(\uhom-\vhom)^2\,dx-2\int_\HOmeg|\nabla \uhom|^2\,dx
-2\int_{\HomO}\kappa\,(\partial_3 \vhom)^2\,dx
\\*[.4em]
& \displaystyle =\int_\HOmeg\nabla \uhom\cdot\nabla\fib\,dx+\int_{\HomO}\kappa\,\partial_3 \vhom\,\partial_3\fib\,dx
-\int_{\HomO}\gamma\,(\uhom-\vhom)^2\,dx
\\*[.4em]
& \displaystyle -\int_\HOmeg|\nabla \uhom|^2\,dx-\int_{\HomO}\kappa\,(\partial_3 \vhom)^2\,dx.
\end{array}
\]
Here we have used
\[
\aeps(1 - \csigeps)^2 \rightharpoonup \lim_{\eps\to 0}\left(\frac{1}{\eps^2}\int_{\eps Y}\aeps(1-\csigeps)^2\,dx\right)
=\lim_{\eps\to 0}\left(\alpha_\eps\,\pi r_\eps^2\right)
=\kappa \quad\text{weakly-$*$} \text{ in } \mathcal{M}(\bar\Omega_0).
\]
On the other hand, choosing $\uhom-\fib$ as test function in the first equation of \eqref{eq:hom-pro} and $\vhom-\fib$ in the 
second equation of \eqref{eq:hom-pro}, it follows that
\[
\begin{array}{rl}
\displaystyle \int_\HOmeg\nabla \uhom\cdot\nabla(\uhom-\fib)\,dx+\int_{\HOmeg}\gamma\,(\uhom-\vhom)\,(\uhom-\fib)\,dx & =0
\\*[.8em]
\displaystyle \int_{\HomO}\kappa\,\partial_3 \vhom\,\partial_3(\vhom-\fib)\,dx-\int_{\HomO}\gamma\,(\uhom-\vhom)\,(\vhom-\fib)\,dx & =0.
\end{array}
\]
Therefore, adding the two previous equalities we obtain that $E_0 = 0$, which gives the thesis.

\section{\label{sec:dGNM}On supremum estimates}
A less {\it ad hoc} version of Lemma~\ref{lem:dGNM} is given below. It is probably known to the experts.
\begin{lemma}\label{lem:dGNM2}
For $\Omega \subset \RR^n$, assume that $w \in H^1_0(\Omega)$ satisfies
\[
-\partial_i(A^{ij}(x)\partial_j w) + c(x)\,w \leq f + \divop(\bh) \text{ in } \Omega
	\;,
\]
for some $f \in L^p(\Omega)$, $\bh \in L^{2\tilde p}(\Omega)^n$, $p, \tilde p > \frac{n}{2}$. 
If the coefficients $A^{ij}$ and $c$ are locally bounded and satisfy
\begin{align*}
&A^{ij}(x)\,\xi_i\,\xi_j \geq \nu\,|\xi|^2 \text{ for any } x \in \Omega \text{ and } \xi \in \RR^n
	\;,\\
&c(x) \geq \lambda \text{ for any } x \in \Omega
	\;,
\end{align*}
then $w^+ \in L^\infty(\Omega)$ and, for any $q > \frac{2p}{2p-n}$ and $\tilde q > \frac{2\tilde p }{2\tilde p - n}$,
\[
\|w^+\|_{L^\infty(\Omega)} \leq C(n,q,\nu)\,\Big[\frac{1}{\lambda^{\frac{1}{q}}}\,\|f\|_{L^{p}(\Omega)} 
+ \frac{1}{\lambda^{\frac{1}{2\tilde q}}}\,\|\bh\|_{L^{2p}(\Omega)}\Big]
	\;.
\]
\end{lemma}

\noindent {\it Proof of Lemma~\ref{lem:dGNM}. }
We follow de Giorgi's method for the proof. For the sake of Lemma~\ref{lem:dGNM2} we will assume that
$$
\int_{\omega}\aeps f^{p}dx <\infty \text{ and } \int_{\omega}\aeps \bh^{2\tilde{p}}dx <\infty
$$
for some $p, \tilde p > 1$; note that $n = 2$ in the present proof. 

For $k > 0$, let $A(k){}=\{x \in \omega: \vpe(x)\geq k\}$. Let $B$ represent the quantity
$$
B{}= \int_{A(k)}\aeps \left|\nabla\left(\vpe-k\right)_{+}\right|^{2} +\lambda\int_{A(k)}\aeps \left(\vpe-k\right)_{+}^{2}
+\lambda\int_{A(k)}\aeps k\left(\vpe-k\right)_{+}.
$$
Integrating \eqref{eq:dGNM} by parts against $(u-k)_{+}$, we obtain
\begin{equation}\label{eq:bound0}
B=\int_{A(k)}\aeps f(u-k)_{+} - \int_{A(k)}\aeps \bh \cdot \nabla (u-k)_{+}.
\end{equation}
 We shall use the notation conventions that \[
\left\Vert f\right\Vert _{\mathcal{L}^{p}\left(A_{k}\right)}{}=\left(\int_{A_{k}}\aeps f^{p}dx\right)^{1/p}\mbox{ and }N\left(A_{k}\right){}=\int_{A_{k}}\aeps dx.\]
Using the weighted Sobolev embedding given by Lemma~\ref{lem:Sob-Marc} we obtain 
\begin{eqnarray*}
\mu_{\eps}\left\Vert \left(\vpe-k\right)_{+}\right\Vert _{\mathcal{L}^{s}
\left(A_{k}\right)}^{2}
+\lambda\left\Vert \left(\vpe-k\right)_{+}\right\Vert _{\mathcal{L}^{2}\left(A_{k}\right)}^{2}
& \leq & B,
\end{eqnarray*}
where 
\begin{equation}
\mu_{\eps}: =C(s)\,\eps^{2\left(1-2/s\right)}.
	\label{eq:defmueps}
\end{equation}
Using Young's inequality, we have, for any $1>\theta_1>0$ and $\theta_2=1-\theta_1$,
\begin{eqnarray*}
&& \mu_{\eps}\left\Vert \left(\vpe-k\right)_{+}\right\Vert _{\mathcal{L}^{s}
\left(A_{k}\right)}^{2}
+\lambda\left\Vert \left(\vpe-k\right)_{+}\right\Vert _{\mathcal{L}^{2}\left(A_{k}\right)}^{2}\\
& \geq & \left(\frac{\mu_{\eps}}{\theta_{1}}\right)^{\theta_{1}}
	\left(\frac{\lambda}{\theta_{2}}\right)^{\theta_{2}}
	\left\Vert \left(\vpe-k\right)_{+}\right\Vert _{\mathcal{L}^{s}\left(A_{k}\right)}^{2\theta_{1}}\left\Vert 
	\left(\vpe-k\right)_{+}\right\Vert _{\mathcal{L}^{2}\left(A_{k}\right)}^{2\theta_{2}}
	\\
& = & \left(\frac{\mu_{\eps}}{\theta_{1}}\right)^{\theta_{1}}\left(\frac{\lambda}{\theta_{2}}\right)^{\theta_{2}}\left\Vert \left(\vpe-k\right)_{+}^{r\theta_{1}}\right\Vert _{\mathcal{L}^{s/r\theta_{1}}\left(A_{k}\right)}^{2/r}\left\Vert \left(\vpe-k\right)_{+}^{r\theta_{2}}\right\Vert _{\mathcal{L}^{2/r\theta_{2}}\left(A_{k}\right)}^{2/r}.
\end{eqnarray*}
On the other hand, Holder's inequality shows that
\begin{eqnarray*}
\left\Vert \left(\vpe-k\right)_{+}\right\Vert _{\mathcal{L}^{r}\left(A_{k}\right)}^{2} & = & \left\Vert \left(\vpe-k\right)_{+}^{\theta_{1}r}\left(\vpe-k\right)_{+}^{\theta_{2}r}\right\Vert _{\mathcal{L}^{1}\left(A_{k}\right)}^{2/r}\\
& \leq & \left\Vert \left(\vpe-k\right)_{+}^{\theta_{1}r}\right\Vert _{\mathcal{L}^{\alpha}\left(A_{k}\right)}^{2/r}\left\Vert \left(\vpe-k\right)_{+}^{\theta_{2}r}\right\Vert _{\mathcal{L}^{\beta}\left(A_{k}\right)}^{2/r},
\end{eqnarray*}
for any $1 <\alpha,\beta < \infty$ such that $\frac{1}{\alpha}+\frac{1}{\beta}=1.$ Choosing $\alpha=\frac{s}{r\theta_{1}}$ yields $\beta=\frac{2}{r\theta_{2}}$ provided
\[
\frac{1}{r}=\frac{\theta_{1}}{s}+\frac{\theta_{2}}{2}.\]
We have obtained 
\begin{eqnarray}\label{eq:bound1}
&& \left(\frac{\mu_{\eps}}{\theta_{1}}\right)^{\theta_{1}}\left(\frac{\lambda}{\theta_{2}}\right)^{\theta_{2}}\left\Vert \left(\vpe-k\right)_{+}\right\Vert _{\mathcal{L}^{r}\left(A_{k}\right)}^{2} \\
&\leq &
\left\|\nabla\left(\vpe-k\right)_{+}\right\|^{2}_{\mathcal{L}^2\left(A_{k}\right)} +  \lambda\left\Vert \left(\vpe-k\right)_{+}\right\Vert _{\mathcal{L}^{2}\left(A_{k}\right)}^{2}  \leq B\nonumber.
\end{eqnarray}
Let us now turn to the right-hand side.

\medskip
\noindent\underline{Case 1:} $\bh = 0$. We have 
\begin{eqnarray*}
B & \leq & \left\Vert f\left(\vpe-k\right)_{+}\right\Vert _{\mathcal{L}^{1}\left(A_{k}\right)}\\
& \leq & \left\Vert f\right\Vert _{\mathcal{L}^{p}\left(A_{k}\right)}\left\Vert \left(\vpe-k\right)_{+}\right\Vert _{\mathcal{L}^{r}\left(A_{k}\right)}\left\Vert 1\right\Vert _{\mathcal{L}^{\kappa}\left(A_{k}\right)}\end{eqnarray*}
with $\frac{1}{\kappa}+\frac{1}{p}+\frac{1}{r}=1$. We require that
$\theta_{1},\theta_{2}$ be chosen so that $\kappa>1$. For $p = 2$ and $s > 2$, 
this requirement is fulfilled by any $\theta_1 > 0$ since $r \in (2,s)$. 
For the general case when $p > \frac{n}{2}$, we can satisfy the above requirement by selecting $s$ sufficiently large 
(but smaller than the Sobolev exponent) so that $\frac{1}{p} + \frac{1}{s} < 1$. We have obtained 
\begin{equation}\label{eq:iterzero}
\left\Vert \left(\vpe-k\right)_{+}\right\Vert _{\mathcal{L}^{r}\left(A_{k}\right)}\leq\left(\frac{\theta_{1}}{\mu_{\eps}}\right)^{\theta_{1}}\left(\frac{\theta_{2}}{\lambda}\right)^{\theta_{2}}\left\Vert f\right\Vert _{\mathcal{L}^{p}\left(\omega\right)}N\left(A_{k}\right)^{\chi/r},
\end{equation}
where $ \chi=r\left(1-p^{-1}-r^{-1}\right) = r/\kappa$. Now, for all $h<k$, 
$$
N\left(A_k\right) = \int_{A_{k}}\aeps dx  
	\leq  \int_{A_{k}}\aeps\frac{\left(w-h\right)_{+}^r}{(k-h)^r} dx\\
	=  \left\Vert \left(\vpe-h\right)_{+}\right\Vert _{\mathcal{L}^{r}\left(A_{h}\right)}^{r}\frac{1}{\left|k-h\right|^{r}}.
$$
Therefore
\[
\left\Vert \left(\vpe-k\right)_{+}\right\Vert _{\mathcal{L}^{r}\left(A_{k}\right)}\leq\left(\frac{\theta_{1}}{\mu_{\eps}}\right)^{\theta_{1}}\left(\frac{\theta_{2}}{\lambda}\right)^{\theta_{2}}\frac{1}{|k-h|^{\chi}}\left\Vert f\right\Vert _{\mathcal{L}^{p}\left(\omega\right)}\left\Vert \left(\vpe-h\right)_{+}\right\Vert _{\mathcal{L}^{r}\left(A_{h}\right)}^{\chi}.\]

To proceed, we select $\theta_1$ such that $\theta_1  > \frac{s}{p(s-2)}$. 
Then  $\chi>1$. 

We now set $k_{j}=2d\left(1-2^{-j-1}\right)$. Introducing
$$
\beta  =  \left\Vert f\right\Vert _{\mathcal{L}^{p}\left(\omega\right)}\left(\frac{\theta_{1}}{\mu_{\eps}}\right)^{\theta_{1}}\left(\frac{\theta_{2}}{\lambda}\right)^{\theta_{2}} 2^\chi\,d^{-\chi},
	\mbox{ and }
x_{j}  =  \beta^{\frac{1}{\chi-1}}\left\Vert \left(\vpe-k_{j}\right)_{+}\right\Vert _{\mathcal{L}^{r}\left(A_{k_{j}}\right)},
$$
Then \eqref{eq:iterzero} takes the form 
$$
x_{j+1}\leq 2^{j\chi}x_{j}^{\chi}.
$$
An induction shows that 
$$
\mbox{ if }
x_{0}\leq2^{-\chi/(\chi-1)^{2}}, \mbox{ then }
x_{j}\leq2^{-\frac{j\chi}{\chi-1}}x_{0}.
$$
Therefore, $\lim_{j\to\infty}x_{j}=0,$  and $\left\Vert \left(\vpe\right)_{+}\right\Vert _{L^{\infty}(\omega)}\leq2d$. Thus, $d$ is given by the constraint
\[
x_{0}\leq2^{-\frac{\chi}{(\chi-1)^{2}}}.\]
Using the $L^{\infty}$ bound just derived, for $j=0$, we obtain \[
x_{0}=\beta^{\frac{1}{\chi-1}}\left\Vert \left(\vpe-d\right)_{+}\right\Vert _{\mathcal{L}^{r}\left(A_{0}\right)}\leq N(A_0)^{\frac{1}{r}}\,\beta^{\frac{1}{\chi-1}}\,d.
\]
Therefore we choose $d$ to be 
$$
d  = \left\Vert f\right\Vert _{\mathcal{L}^{p}\left(\omega\right)}\left(\frac{\theta_{1}}{\mu_{\eps}}\right)^{\theta_{1}}\left(\frac{\theta_{2}}{\lambda}\right)^{\theta_{2}}N(A_0)^{\frac{\chi-1}{r}}2^{\frac{\chi^2}{\chi-1}}.
$$
Altogether, we have obtained 
$$
\left\Vert \vpe\right\Vert _{L^{\infty}(\omega)} \leq C(|\omega|,\kappa_+,\chi,r)\left\Vert f\right\Vert _{\mathcal{L}^{p}\left(\omega\right)}\left(\frac{\theta_{1}}{\mu_{\eps}}\right)^{\theta_{1}}\left(\frac{\theta_{2}}{\lambda}\right)^{\theta_{2}}
$$
For $p=2$, we write $\theta_1 = \mu \frac{s}{2(s-2)}$, with $1<\mu<\frac{2(s-2)}{s} < 2$ and use \eqref{eq:defmueps} to obtain
$$
\left\Vert \vpe\right\Vert _{L^{\infty}(\omega)} \leq C(\mu,s) \left\Vert f\right\Vert _{\mathcal{L}^{2}\left(\omega\right)} 
\frac{1}{\displaystyle \eps^{\mu}\lambda^{1-\frac{\mu s}{2(s-2)}}}. 
$$

Now, for a fixed $1 < \alpha < 2$, we can choose $s > 2$ sufficient large so that we can choose $\mu = \alpha$ in the above estimate.
 By further enlarging $s$, we see that we have obtained the assertion for $\beta$ as close to $1 - \frac{\alpha}{2}$ as we wish. 
The conclusion for smaller $\beta$ also follows.

\noindent\underline{Case 2:} $f = 0$. Using H\"older's inequality we get
\begin{eqnarray*}
B &\leq &  \|\bh\|_{\mathcal{L}^{2\tilde p}(\omega)}\,\|\nabla (w - k)_+\|_{\mathcal{L}^2(A(k))}\,N\left(A(k)\right)^{\frac{1}{2} - \frac{1}{2\tilde p}}.
\end{eqnarray*}
Together with \eqref{eq:bound1}, this shows that
\begin{equation}\label{eq:zerobis}
\left\Vert \left(\vpe-k\right)_{+}\right\Vert _{\mathcal{L}^{r}\left(A_{k}\right)}\leq\left(\frac{\theta_{1}}{\mu_{\eps}}\right)^{\theta_{1}/2}\left(\frac{\theta_{2}}{\lambda}\right)^{\theta_{2}/2}
\|\bh\|_{\mathcal{L}^{2\tilde p}(\omega)}\,N\left(A(k)\right)^{\chi/r},
\end{equation}
where $\frac{\chi}{r} = \frac{1}{2} - \frac{1}{2\tilde p}$. Note that \eqref{eq:zerobis} is
 a reverse H\"older inequality of the same type as \eqref{eq:iterzero}. Therefore, arguing as above, 
and provided $\frac{1}{2} - \frac{1}{2\tilde p} >\frac{1}{r}$,  to ensure that  $\chi>1$, we obtain
$$
\left\Vert \vpe\right\Vert _{L^{\infty}(\omega)} \leq  C(|\omega|,\kappa_+,\chi,r)\,\left\Vert \bh \right\Vert _{\mathcal{L}^{2\tilde p}\left(\omega\right)}\left(\frac{\theta_{1}}{\mu_{\eps}}\right)^{\theta_{1}/2}
\left(\frac{\theta_{2}}{\lambda}\right)^{\theta_{2}/2}
$$
Considering now the case $\tilde{p}=\infty$, introducing $\theta_1= (\alpha-1) \frac{s}{2(s-2)}$, with $1<\alpha<2\frac{s-2}{s}$, and $\beta=\frac{1 - \theta_1}{2}$, 
we obtain
$$
\left\Vert \vpe\right\Vert _{L^{\infty}(\omega)} \leq  
    C(\alpha,\beta)\left\Vert \bh \right\Vert _{\mathcal{L}^{2\tilde p}\left(\omega\right)}
                                             \frac{1}{\eps^{(\alpha-1)/2}\lambda^{\beta}},
$$
as announced.
\qed

\par\bigskip{}\noindent{\it Proof of Lemma \ref{lem:dGNM2}. }
The proof is similar to the one above, with $\aeps$ replaced by $\nu$. 
The constant $\mu(s)$ is now independent of $\eps$. Some of the details are provided in the proof of Lemma~\ref{lem:dGNM} for the reader's convenience.
\qed

\begin{lemma}\label{lem:Sob-Marc}
Assume that conditions \eqref{eq:defaeps}, \eqref{eq:boundkappa} and \eqref{eq:Cond1} hold. Then, for any $s\in[2,\infty)$, there exists a constant $C(s)>0$ such that
\begin{equation}\label{weightS}
\int_\omega a_\eps\left|v\right|^s dx\leq C(s)\,\eps^{2-s}\left(\int_\omega a_\eps\left|\nabla v\right|^2 dx\right)^{\frac{s}{2}} \qquad \forall\,v\in H^1_0(\omega).
\end{equation}
\end{lemma}

\begin{proof}
As a first step of the proof, we establish a rescaled version of \eqref{weightS} in the cell $Y=(-1/2,1/2)^2$. 
Define a rescaled conductivity $A_\eps(y)=a_\eps(\eps y)$, for $y\in Y$. Using a $r_\eps$-rescaling 
and the Sobolev embedding of $H^1(D_2)$ into $L^s(D_2)$ applied to the unit disk $D_2$, we get that
\[
\forall\,V\in H^1(Y),\quad\moy_{D(r_\eps)}\left|\,V-\moy_{D(r_\eps)}V\right|^s dy\leq C\,r_\eps^s\left(\moy_{D(r_\eps)}|\nabla V|^2\,dy\right)^{\frac{s}{2}}.
\]
Moreover, by estimate (3.13) in \cite{BRIANE-MOKOBODZKI-MURAT-08} we have
\[
\forall\,V\in H^1(Y),\quad\,\int_Y \left|V-\moy_{D(r_\eps)}V\,\right|^s dy\leq C\,|\ln r_\eps|^{\frac{s}{2}}\left(\int_Y |\nabla V|^2\,dy\right)^{\frac{s}{2}}.
\]
Then, combining the two previous estimates it follows that for any $V\in H^1(Y)$,
\[
\begin{array}{ll}
\displaystyle \moy_{D(r_\eps)}\left|\,V-\int_Y V\right|^s dy & \displaystyle \leq C\left|\,\int_Y Vdy-\moy_{D(r_\eps)}V\,dy\right|^s
+C\moy_{D(r_\eps)}\left|\,V-\moy_{D(r_\eps)}V\right|^s dy
\\
& \displaystyle \leq C\,|\ln r_\eps|^{\frac{s}{2}}\left(\int_Y |\nabla V|^2\,dy\right)^{\frac{s}{2}}
+C\,r_\eps^s\left(\moy_{D(r_\eps)}|\nabla V|^2\,dy\right)^{\frac{s}{2}}.
\end{array}
\]
Hence, by the definition of $A_\eps$ and \eqref{eq:defaeps},
\[
\int_{D(r_\eps)}A_\eps\left|\,V-\int_Y V\right|^s dy\leq C\,|\ln r_\eps|^{\frac{s}{2}}\left(\int_Y A_\eps\,|\nabla V|^2\,dy\right)^{\frac{s}{2}}.
\]
Since the Sobolev embedding inequality in $Y$ gives
\[
\int_Y\left|\,V-\int_Y V\right|^s dy
	\leq C	\left(\int_Y |\nabla V|^2\,dy\right)^{\frac{s}{2}}
	\leq C\left(\int_Y A_\eps\,|\nabla V|^2\,dy\right)^{\frac{s}{2}},
\]
we thus deduce the following estimate in $Y$,
\begin{equation}\label{weightSY}
\forall\,V\in H^1(Y),\quad\int_Y A_\eps\left|\,V-\int_Y V\right|^s dy
	\leq C\,|\ln r_\eps|^{\frac{s}{2}}\left(\int_Y A_\eps\,|\nabla V|^2\,dy\right)^{\frac{s}{2}}.
\end{equation}
\par

We now turn to the proof of \eqref{weightS}. Let $v\in H_0^1(\omega)$ and extend it by zero outside $\omega$. 
Rescaling estimate \eqref{weightSY} in each square $\eps\left(n+Y\right)$, for $n\in\ZZ^2$, we obtain that
\[
\begin{array}{ll}
\displaystyle \int_{\RR^2} a_\eps\left|v-\bar{v}_\eps\right|^s dx
& \displaystyle \leq C\,\eps^2\,|\ln r_\eps|^{\frac{s}{2}}\sum_{n\in\ZZ^2}\left(\int_{\eps\left(n+Y\right)}a_\eps\,|\nabla v|^2\,dx\right)^{\frac{s}{2}}
\\
& \displaystyle \leq C\,\eps^2\,|\ln r_\eps|^{\frac{s}{2}}\left(\int_{\RR^2}a_\eps\,|\nabla v|^2\,dx\right)^{\frac{s}{2}}\quad (\mbox{since }s\geq 2),
\end{array}
\]
where $\bar{v}_\eps$ is the piecewise function which takes the average value of $v$ in each square $\eps\left(n+Y\right)$.
Moreover, we have
\[
\begin{array}{ll}
\displaystyle \int_{\RR^2} a_\eps\left|\bar{v}_\eps\right|^s dx
& \displaystyle =\sum_{n\in\ZZ^2}\left(\int_{\eps\left(n+Y\right)}a_\eps\,dx\right)\left|\moy_{\eps\left(n+Y\right)}v\,dx\right|^s
\\
& \displaystyle \leq\sum_{n\in\ZZ^2}C\,\eps^2\moy_{\eps\left(n+Y\right)}|v|^s\,dx\quad\mbox{(by the Jensen inequality)}
\\
& \displaystyle \leq C\int_{\omega}|v|^s\,dx
\\
& \displaystyle \leq C\left(\int_{\omega}|\nabla v|^2\,dx\right)^{\frac{s}{2}}\quad\mbox{(by the Sobolev embedding of $H^1_0(\omega)$ into $L^s(\omega)$).}
\end{array}
\]
Finally, combining the two previous inequalities we find
\[
\int_{\RR^2} a_\eps\left|v\right|^s dx\leq C\left(\eps^2\,|\ln r_\eps|^{\frac{s}{2}}+1\right)\left(\int_{\omega}a_\eps\,|\nabla v|^2\,dx\right)^{\frac{s}{2}},
\]
which yields the desired estimate \eqref{weightS} taking into account \eqref{eq:defaeps}.
\qed\end{proof}
\begin{remark}
The constant in the weighted inequality of \cite{SAWYER-WHEEDEN-92} provides an estimate from above of the constant appearing in \eqref{weightS}. 
It is not clear that this constant is optimal. The dependence in $\eps$
of the constant in \eqref{weightS} is optimal: this can be verified
with the choice $v=g_\eps$, where $g_\eps$  is given by \eqref{eq:def-geps-test}. 
\end{remark}

\section{Proof of Proposition~\ref{prop:Linfybd2d}\label{sec:proof-2d}} 
All of this section is in the two-dimensional setting. We will therefore drop the subscript $2$ to denote two-dimensional gradients or divergences.
We consider the solution $\ueps$ of \eqref{eq:Eq2.5d}, for $\lambda \geq \lambda_0 > 0$,
\[
\left\{\begin{array}{ll}
-\divop(\aeps\,\nabla \ueps) + \lambda\,\aeps\,\ueps = f +\aeps\, g+\divop(\bh)&\text{ in } \homeg	\;,\\
\ueps = 0 &\text{ on } \partial\homeg	\;.
\end{array}\right.
\]
Let us start with a simple energy bound.
\begin{lemma}\label{lem:bounds} We have
\begin{equation}
\int_{\homeg} \aeps\,\Big[|\nabla  \ueps|^2 + \lambda\,|\ueps|^2\Big]\,dx
	\leq \frac{C}{\lambda}\int_{\homeg} \big[|f|^2 + \aeps |g|^2\big]\,dx + C\,\int_{\homeg} |\bh|^2\,dx
	\label{eq:ape-enerbd}.
\end{equation}
\end{lemma}

\begin{proof}
Integrating \eqref{eq:Eq2.5d} by parts against $u_\eps$ yields
$$
\int_{\homeg} \aeps\Big[|\nabla  \ueps|^2 + \lambda | \ueps|^2\Big]\,dx =  \int_{\homeg} \Big[f \ueps + \aeps g \ueps - \bh\cdot\nabla \ueps\Big]\,dx 
$$ 
Using the Cauchy-Schwarz inequality on the right-hand side we obtain \eqref{eq:ape-enerbd}. 
\qed\end{proof}

We now turn to the main part of our estimate. Our strategy, inspired by the limit case when the conductivity 
tends to infinity independently of the periodic structure \cite{CIORANESCU-PAULIN-79,BRIANE-97,CIORANESCU-DAMLAMIAN-DONATO-MASCARENHAS-99}, 
is to consider three contributions to $u_\eps$. 
The first one is the contribution of the right-hand side when no highly conducting fibres are present. We introduce $v \in H^1_0(\homeg)$, the solution to 
\begin{equation}\label{eq:defv}
-\Delta v + \lambda v = f+ g + \divop(\bh).
\end{equation}
The second is the contribution coming from the \emph{average} of $\ueps - v $ on the cross section of the fibres. Set
\[
u_{m,n} = \frac{1}{|\partial D_{m,n,\eps }|}\,\int_{\partial D_{m,n,\eps }} [\ueps(x) - v(x)]\,d\sigma(x),
\]
and define $\tueps \in H^1(\homeg \setminus \Deps)$ as the solution to
\begin{equation}\label{eq:defutilde}
\left\{\begin{array}{ll}
-\Delta \tueps + \lambda\,\tueps  = 0 & \text{ in } \homeg \setminus \Deps	\;,\\
\tueps = 0 & \text{ on } \partial\homeg	\;,\\
\tueps = u_{m,n} & \text{ on } \partial D_{m,n,\eps }	\;.
\end{array}\right.
\end{equation}
The third is simply the remainder, given by $\hueps = \ueps - v - \tueps$ on $\homeg\setminus\Deps$. We will consider each contribution 
separately. Intuitively, our estimates show that, outside the fibres, $\ueps$ consists ``mainly'' of two parts: $v$, the contribution of 
the background medium, and $\tueps$, the contribution from the high contrast.
\par\bigskip\noindent
{\it Proof of Proposition \ref{prop:Linfybd2d}. }
Fix $\tau> \kappa \eps^{\frac{1-\eta}{2(1+\eta)}}$.

Since $\tueps+\hueps=\ueps-v$, Lemma~\ref{lem:utilde} (with $\beta = \eta/2$) 
and Corollary~\ref{cor:estuhat} show that
\begin{eqnarray*}
 && \|\ueps-v \|_{L^\infty(\homegtau)} +\|\nabla \ueps- \nabla v \|_{L^\infty(\homegtau)} \\
  &\leq& \frac{C}{\lambda^{\frac{\eta}{2}}}\left( \frac{\eps^2}{\lambda}\left(\|f\|_{L^2(\homeg)} + \|\sqrt{\aeps}g\|_{L^2(\homeg)}
			+ \|\bh\|_{L^\infty(\homeg)}\right) + \|\ueps - v\|_{L^\infty(\homeg^\prime)}\right) \\
\end{eqnarray*}
where $\omega_0\Subset\omega^\prime\Subset\omega$. Using the  $L^\infty$ estimates on $v$ given by Lemma~\ref{lem:EstGradVInfty} (with $q = 2/\eta$), we obtain
\begin{eqnarray*}
&& \|\ueps-v \|_{L^\infty(\homegtau)} +  \|\nabla \ueps- \nabla v \|_{L^\infty(\homegtau)} \\
 &\leq& \frac{C}{\lambda^{\frac{\eta}{2}}}
\left(\frac{1}{\lambda^{\frac{\eta}{2}}} \left(\|f\|_{L^2(\homeg)} + \|\sqrt{\aeps}g\|_{L^2(\homeg)}
			+ \|\bh\|_{L^\infty(\homeg)}\right) + \|\ueps\|_{L^\infty(\homeg^\prime)}\right).
\end{eqnarray*}
Lemma~\ref{lem:HolderX} together with Corollary~\ref{cor:estuhat} show that provided $\nu<\eta$,
$$
[\nabla_2 \ueps - \nabla_2 v]_{C^\nu(\homegtau)} \leq \frac{C}{\lambda^{\frac{\eta}{2} - \frac{\nu}{2}}}  \|\ueps - v\|_{L^\infty(\homeg^\prime)} 
+ C \frac{\eps^2}{\lambda^{\frac{3}{2}}}\left(\|f\|_{L^2(\homeg)} + \|\sqrt{\aeps}g\|_{L^2(\homeg)}
			+ \|\bh\|_{L^\infty(\homeg)}\right).
$$
Using again the $L^\infty$ estimates on $v$ given by Lemma~\ref{lem:EstGradVInfty}, we obtain
$$
[\nabla_2 \ueps - \nabla_2 v]_{C^\nu(\homegtau)} \leq \frac{C}{\lambda^{\frac{\eta}{2} - \frac{\nu}{2}}}\left(\frac{1}{\lambda^{\frac{\eta}{2}}} \left(\|f\|_{L^2(\homeg)} 
+ \|\sqrt{\aeps}g\|_{L^2(\homeg)}
			+ \|\bh\|_{L^\infty(\homeg)}\right) + \|\ueps\|_{L^\infty(\homeg^\prime)}\right),
$$
which concludes the proof.
\qed

\subsection{Estimates for $v$}
It is quite straightforward to obtain estimates on $v$, as the following Lemma shows. 
\begin{lemma}\label{lem:EstGradVInfty}
There holds
\begin{align*}
&\sqrt{\lambda}\,\|\nabla v\|_{L^2(\homeg)} + \lambda\,\|v\|_{L^2(\homeg)} \leq C\,\left[\|f\|_{L^2(\homeg)}+ \|g\|_{L^2(\homeg)}
 + \sqrt{\lambda}\,\|\bh\|_{L^2(\homeg)}\right]
	,\\
&\|\nabla v\|_{L^q(\omega)}
	\leq C(q)\,\left[\|f\|_{L^2(\homeg)}+ \|g\|_{L^2(\homeg)} + \|\bh\|_{L^\infty(\homeg)}\right]
	, \qquad 2 \leq q < \infty.
\end{align*}
Furthermore, for any $q > 2$, and any $\homeg^\prime\Subset\omega$, 
\[
\|v\|_{L^\infty(\omega^\prime)}
	\leq \frac{C(q,\omega^\prime)}{\lambda^{\frac{1}{q}}}\,\left[\|f\|_{L^2(\homeg)}+ \|g\|_{L^2(\homeg)} + \|\bh\|_{L^\infty(\homeg)}\right].
\]
\end{lemma}

\begin{proof}
The $L^2$ estimate for $\nabla v$ and $v$ follows directly by testing \eqref{eq:defv} against $v$. 
To obtain the $L^q$ gradient estimate, we write $v = v_1 + v_2$ where $v_1, v_2 \in H^1_0(\omega)$ are the solutions of
\[
-\Delta v_1 = \lambda\,v + f + g = \tilde f \text{ and } - \Delta v_2  = \divop(\bh).
\]
Next, by standard elliptic estimates, we have
\begin{align*}
\|v_1\|_{H^2(\omega)} 
	&\leq C\,\|\tilde f\|_{L^2(\homeg)} \leq C\Big[\lambda\,\|v\|_{L^2(\homeg)} + \|f+g\|_{L^2(\homeg)}\Big] \leq C\,\left(\|f\|_{L^2(\homeg)}+ \|g\|_{L^2(\homeg)}\right)
	,\\
\|v_2\|_{W^{1,q}(\omega)} 
	&\leq C\,\|\bh\|_{L^q(\homeg)}
	.
\end{align*}
Since $H^2(\omega) \hookrightarrow W^{1,q}(\omega)$, the second estimate in the lemma follows. The last assertion follows from Lemma~\ref{lem:dGNM2}.
\qed\end{proof}

A direct consequence of the above result is a local $L^2$ bound on $\nabla v$.
\begin{corollary}\label{lem:GaNiV}
For any $\omega' \Subset \homeg$, $\beta < \frac{1}{2}$ and $\varrho < \frac{1 - 2\beta}{4}$, there holds
$$
\| \nabla v \|_{L^2(\omega')}\leq C\,|\omega'|^{\varrho}\,\left[\frac{1}{\lambda^\beta}\left(\|f\|_{L^2(\homeg)} + \|\sqrt{\aeps}\,g\|_{L^2(\homeg)}\right) + \|\bh\|_{L^\infty(\homeg)}\right].
$$
\end{corollary}

\begin{proof}
We split $v = v_a + v_b$ where $v_a, v_b \in H^1_0(\omega)$ are the solutions to
\[
-\Delta v_a + \lambda v_a = f + g \text{ and } - \Delta v_b + \lambda v_b = \divop(\bh).
\]
By the $L^q$ gradient bound in Lemma \ref{lem:EstGradVInfty} and the Sobolev embedding theorem,
\[
\|\nabla v_a\|_{L^p(\homeg)} \leq C\,\left(\|f\|_{L^2(\homeg)}+ \|g\|_{L^2(\homeg)}\right) \text{ and } \|\nabla v_b\|_{L^p(\homeg)} \leq C\,\|\bh\|_{L^\infty(\homeg)} \text{ for any } p > 2.
\]
Thus, by H\"older's inequality, we have, for any $\omega' \subset \homeg$ and $\delta \in (0,1)$, 
\begin{align*}
\|\nabla v_a\|_{L^2(\omega')} \leq C\,\left(\|f\|_{L^2(\homeg)}+ \|g\|_{L^2(\homeg)}\right)\,|\omega'|^{1 - \delta}
	\text{ and }
\|\nabla v_b\|_{L^2(\omega')} \leq C\,\|\bh\|_{L^\infty(\homeg)}\,|\omega'|^{1 - \delta}.
\end{align*}
On the other hand, by Lemma \ref{lem:EstGradVInfty}, we also have
\[
\|\nabla v_a\|_{L^2(\homeg)} \leq \frac{C}{\sqrt{\lambda}}\,\left(\|f\|_{L^2(\homeg)}+ \|g\|_{L^2(\homeg)}\right) \text{ and } \|\nabla v_b\|_{L^2(\homeg)} \leq C\,\|\bh\|_{L^2(\homeg)}.
\]
The conclusion follows by an interpolation.
\qed\end{proof}

\subsection{Estimates for $\tueps$}

The first part of our estimate for $\tueps$ is given by the following Lemma. 
The second part, concerning the H\"{o}lder regularity of its gradient, is given by Lemma~\ref{lem:HolderX}.

\begin{lemma}\label{lem:utilde}
For any $0 < \beta < \frac{1}{2}$ and $\tau > \kappa \eps^{\frac{1-2\beta}{1+2\beta}}$,  the solution $\tueps$ of \eqref{eq:defutilde} satisfies
\[
\|\tueps\|_{L^\infty(\homegtau)} + \lambda^{\beta - 1/2}\|\nabla \tueps\|_{L^\infty(\homegtau)} \leq \frac{C}{\lambda^{1/2}}\sup_{m,n} |u_{m,n}|
	\;,
\]
where the constant $C$ depends on $\beta$, $\kappa$ and $\homeg$ only, and $\homegtau$ is defined in \eqref{eq:defhomegtau}.
\end{lemma}
\par\noindent
Here and below $\sup_{m,n}$ refers to the supremum taken as $(m,n)$ varies in $I_\eps$.
\par
We shall use two local estimations. We first estimate the gradient $\nabla \tueps$ on the outer boundary~$\partial\homeg$.

\begin{lemma}\label{TU}
There holds
\[
\|\nabla \tueps\|_{L^\infty(\partial\homeg)} \leq C \sup_{m,n} |u_{m,n}|.
\]
\end{lemma}

\begin{proof} By the maximum principle,
\begin{equation}
\sup_{\homeg \setminus \Deps} |\tueps| \leq \sup_{m,n} |u_{m,n}|
	\;.\label{TU-E1x}
\end{equation}

Fix $\omega_0 \Subset \omega'  \Subset \omega'' \Subset \homeg$. Let $\eta$ be a cut-off function which 
is one in $\homeg \setminus \omega'$ and vanishes in $\omega_0$. Using $\eta^2\,\tueps$ as a test function in \eqref{eq:defutilde}, we obtain
\begin{eqnarray*}
\int_{\homeg\setminus \omega_0 } \eta^2|\nabla \tueps|^2 + \int_{\homeg\setminus \omega_0} \lambda\eta^2\tueps^2  
	&= & - 2 \int_{\homeg\setminus \omega_0} \eta \tueps \nabla\eta \cdot\nabla \tueps, \\
	&\leq & \frac{1}{2}\int_{\homeg\setminus \omega_0 } \eta^2|\nabla \tueps|^2 + 2 \int_{\homeg\setminus \omega_0} \ueps^2\,|\nabla\eta|^2.
\end{eqnarray*}
Thanks to \eqref{TU-E1x}, we conclude that
\begin{equation}
\|\nabla\tueps\|_{L^2(\homeg \setminus \omega')} \leq C\, \sup_{m,n} |u_{m,n}|
	\;.\label{TU-E2}
\end{equation}

Let $u^\pm \in H^1(\homeg\setminus \omega')$ be the solution to
\[
\left\{\begin{array}{ll}
-\Delta u^\pm = 0  & \text{ in } \homeg \setminus \omega'	\;,\\
u^\pm = 0          & \text{ on } \partial\homeg	\;,\\
u^\pm = \tueps^\pm & \text{ on } \partial\omega^\prime	\;,
\end{array}\right.
\]
where $\tueps^\pm = \max(\pm \tueps, 0)$.
By $H^2$-estimates, valid up the boundary $\partial\homeg$ thanks to the vanishing boundary condition and the regularity of $\partial\omega$, 
$u^\pm \in H^2(\homeg \setminus \omega'')$, with
\[
\|u^\pm\|_{H^2(\homeg\setminus\omega'')} \leq C\,\|\tueps\|_{H^1(\homeg\setminus\omega^\prime)}.
\]
This, in turn, because $u^\pm$ satisfies a homogeneous equation in $\homeg\setminus\omega^\prime$, implies a better estimate, 
namely a $W^{2,p}$-estimates for an arbitrary $p$ \cite[Lemma 9.16]{GILBARG-TRUDINGER-83}. For $p$ large enough, 
thanks to the Sobolev embedding Theorem, $u^\pm \in W^{2,p}(\homeg \setminus \omega'') \hookrightarrow W^{1,\infty}(\homeg \setminus \omega'')$ and we deduce that
\[
\|u^\pm\|_{W^{1,\infty}(\homeg \setminus \omega'')} \leq C\,\|\tueps\|_{H^1(\homeg\setminus\omega')}
	\;.
\]
Inserting \eqref{TU-E1x} and \eqref{TU-E2} in this last estimate, we obtain
\begin{equation}
\|u^\pm\|_{W^{1,\infty}(\homeg \setminus \omega'')}  \leq C \sup_{m,n} |u_{m,n}|
	\;.\label{TU-E3}
\end{equation}
As an application of the maximum principle we note that $u^\pm \geq 0$ and so
\[
-\Delta u^+ + \lambda\,u^+ \geq 0 = - \Delta \tueps + \lambda \tueps \geq -\Delta (-u^-) + \lambda\,(-u^-) \text{ in } \homeg \setminus \omega'
	\;.
\]
Applying the maximum principle again, we thus get
\[
u^+ \geq \tueps \geq - u^- \text{ in } \homeg\setminus \omega'.
\]
Since the three functions agree on $\partial\homeg$, the required estimate follows from \eqref{TU-E3}.
\qed\end{proof}
Next, we estimate the trace of the gradient  $\nabla \tueps$ on the boundary of each rod $D_{m,n,\eps}$.
\begin{lemma}
\label{SCBdryGrad}
There holds
\[
\|\nabla \tueps\|_{L^\infty(\partial D_{m,n,\eps })} \leq C \theta(\lambda,\eps)\,\sup_{m,n}|u_{m,n}|
	\;,
\]
where
\begin{equation}\label{eq:mulambda}
\theta(\lambda,\eps ) = \left\{\begin{array}{ll}
\displaystyle \frac{\eps}{\reps}  &\text{ if } \displaystyle  \sqrt{\lambda} \leq \frac{1}{\eps},\\
\\
\displaystyle \frac{\eps}{\reps\left(1-\alpha +\eps^2 \right)} &\text{ if } \displaystyle   \sqrt{\lambda} = \frac{1}{\eps\reps^\alpha} \mbox{ with } 0\leq\alpha\leq 1, \\
\\
\sqrt{\lambda}                    &\text{ if } \displaystyle \frac{1}{\eps\reps} \leq \sqrt{\lambda}.
\end{array}\right.
\end{equation}

\end{lemma}
This result is proved in Appendix~\ref{sec:lem46}.

We are now in position to prove Lemma~\ref{lem:utilde}. 
Note that the fundamental solution $\Phi$ of $-\Delta + \lambda$ on~$\mathbb{R}^2$ is given by
\begin{equation}
\Phi(x) = \Phi(|x|) = \frac{1}{2\pi}\,K_0(\sqrt{\lambda}\,|x|)
	\;,\label{FundSol}
\end{equation}
where $K_0$ is the second modified Bessel function of 
order zero. In what follows we will make extensive use of well known facts concerning $K_0$ and its derivatives, easily recovered from the recurrence
relations and asymptotic properties, see e.g. \cite{NIST-10}.
Let $K_0^{(n)}$ denote de $n$-th derivative of $K_0$, $n\geq0$. Then  $|K_0^{(n)}|$ is decreasing, and, for all $x>0$ and $\alpha\geq0$,
\begin{eqnarray}
|K_0(x)| \leq C(\alpha)\frac{|\ln(x)|}{1+x^{\alpha}} \text{ and } |K_0^{(n)}(x)| \leq \frac{C(n,\alpha)}{x^n(1+x^{\alpha})}, n \geq 1. \label{eq:bdkn}
\end{eqnarray}

\noindent {\it Proof of Lemma~\ref{lem:utilde}. }
Our argument uses layer potentials. We set $M {}= \sup_{m,n} |u_{m,n}|$. Let $\Phi$ be given by \eqref{FundSol}.  We have
\begin{align}
\tueps(x) 
	&= \int_{\partial(\homeg \setminus \Deps)} \Big[\frac{\partial\tueps}{\partial n(y)} (y)\,\Phi(x - y) - \tueps(y)\,\frac{\partial\Phi}{\partial n(y)}(x - y)\Big]\,d\sigma(y)
	\;.\label{RepForm17Oct10}
\end{align}
Here $n(y)$ denotes the outward unit normal to $\partial(\omega \setminus \Deps)$ at $y$. 
Fix any multi-index $J \in \ZZ^2$. Differentiating \eqref{RepForm17Oct10} we obtain
\begin{align}
\nabla^J \tueps(x) 
	&= \int_{\partial(\homeg \setminus \Deps)} \Big[\frac{\partial\tueps}{\partial n(y)} (y)\,\nabla^{J}_x\Phi(x - y) - \tueps(y)\,\frac{\partial \nabla^J_x\Phi}{\partial n(y)}(x - y)\Big]\,d\sigma(y)\nonumber\\
	&= I_1^J + I_2^J + I_3^J
	\;,\label{RepForm17Oct10-Der}
\end{align}
where
\begin{align*}
I_1^J
	& {}= \int_{\partial \homeg} \frac{\partial\tueps}{\partial n(y)} (y)\,\nabla^J_x\Phi(x - y)\,d\sigma(y),\\
I_2^J
	& {}= -\sum_{(m,n) \in I_\eps}\int_{\partial D_{m,n,\eps }} \tueps(y)\,\frac{\partial \nabla^J_x\Phi}{\partial n(y)}(x - y)\,d\sigma(y),\\
I_3^J
	& {}= \sum_{(m,n) \in I_\eps}\int_{\partial D_{m,n,\eps }} \frac{\partial\tueps}{\partial n(y)} (y)\,\nabla^J_x\Phi(x - y) \,d\sigma(y)
	\;.
\end{align*}

Fix $x \in \homegtau$. We now estimate the right-hand side of \eqref{RepForm17Oct10-Der}. We begin with the easy estimates for $I_1^J$ and $I_2^J$. 
Thanks to \eqref{eq:bdkn} and  Lemma \ref{TU}, 
\begin{equation}
|I_1^J| \leq \frac{C(|J|,\tilde\beta)} {\lambda^{\tilde\beta}} M \mbox{ for all } \tilde\beta \geq 0.
	\label{I1Est}
\end{equation}
To bound $I_2^J$, we estimate the contributions of inclusions located ``radially'' between $\eps\tau+ j\eps$ 
and $\eps\tau+ (j+1)\eps$ for $j=0,1,\ldots$ We obtain using \eqref{eq:bdkn},
\begin{align}
|I_2^J|
	&\leq C\sum_{j=0}^{C\,\eps ^{-1}} M\,(j+1)\,\sqrt{\lambda}^{|J| + 1}\,
       \left|K_0^{(|J| + 1)}(\sqrt{\lambda}\,(j  + \tau)\eps )\right|\,\eps \,\reps\nonumber\\
	&\leq CM\,\frac{\reps}{\lambda^{\tbeta}\,\tau^{|J| + 1 + 2\tbeta}\,\eps^{|J| + 2\tbeta}}\,
          \sum_{j=0}^{C\,\eps ^{-1}} \big(\sqrt{\lambda}\,(j + \tau)\,\eps)^{|J| + 1 + 2\tbeta}\,
          \left|K_0^{(|J| + 1)}(\sqrt{\lambda}\,(j  + \tau)\eps )\right|\nonumber\\
	&\leq CM\,\frac{\reps}{\lambda^{\tbeta}\,\tau^{|J| + 1 + 2\tbeta}\,\eps^{|J| + 1+ 2\tbeta}} 
          \sup_{0 < t< \infty} t^{|J| + 1 + 2\tbeta} K_0^{(|J| + 1)}(t)\nonumber\\
       &= CM\,\frac{\reps}{\lambda^{\tbeta}\,\tau^{|J| + 1 + 2\tbeta}\,\eps^{|J| + 1+ 2\tbeta}}, \qquad \tbeta \geq 0.
	\label{I2Est}
\end{align}
Let us now turn to the third integral, $I_3^J$. Using Lemma \ref{SCBdryGrad} and counting contributions from ``rings'' of inclusions as above, we have
\[
|I_3^{J}|
	\leq C\sum_{j=0}^{C\,\eps ^{-1}} \theta(\lambda,\eps ) \,M\,(j+1)\,\sqrt{\lambda}^{|J|}
          \left|K_0^{(|J|)}(\sqrt{\lambda}\,(j + \tau)\eps )\right|\,\eps \,\reps.
\]
Since $\theta(\lambda,\eps)$ can be ``large'', we proceed to decompose $I_3^J$ into two parts: one counts the contribution 
from the ring of inclusions closest to $x$, where the dependence on $\tau$ is dominant, and the other counts the contribution of the 
further away inclusions, where $\tau$ does not play a role. 
Proceeding in this way, we compute, using the decay properties of $K_0^{(|J|)}$ for the second term,
\begin{align*}
|I_3^{J}|
	&\leq C\sum_{j=0}^{C\,\eps ^{-1}} \theta(\lambda,\eps ) \,M\,(j+1)\,\sqrt{\lambda}^{|J|}
          \left|K_0^{(|J|)}(\sqrt{\lambda}\,(j + \tau)\eps )\right|\,\eps \,\reps\\
	&\leq \underbrace{C\,M\,\theta(\lambda,\eps )\,\eps\,\reps\,\sqrt{\lambda}^{|J|}
               \left|K_0^{(|J|)}(\sqrt{\lambda}\,\tau\,\eps )\right|}_{X^J = \text{near neighbour terms, i.e. $j = 0$}}\\
        &\qquad + \underbrace{C\,M\,\theta(\lambda,\eps )\,\eps\,\reps\,\sum_{j=1}^{C\,\eps ^{-1}} \,j\,\sqrt{\lambda}^{|J|}
               \left|K_0^{(|J|)}(\sqrt{\lambda}\,j\,\eps )\right|}_{Y^J = \text{distant neighbour terms}}.
\end{align*}
The near term $X^J$ can be bounded as follows,
\begin{align}
X^J
	&\leq C\,M\,\frac{\theta(\lambda,\eps )\,\reps}{\lambda^\tbeta\,\tau^{|J| + 2\tbeta}\,\eps^{|J| - 1+ 2\tbeta}}\,
 (\sqrt{\lambda}\,\tau\,\eps)^{|J| + 2\tbeta}\left|K_0^{(|J|)}(\sqrt{\lambda}\,\tau\,\eps )\right|\nonumber\\
	&\leq C\,M\,\frac{\theta(\lambda,\eps )\,\reps}{\lambda^\tbeta\,\tau^{|J| + 2\tbeta}\,\eps^{|J| - 1 + 2\tbeta}}
 \,\sup_{t > 0} t^{|J| + 2\tbeta}\left|K_0^{(|J|)}(t)\right|\nonumber\\
	&\leq C\,M\,\frac{\theta(\lambda,\eps )\,\reps}{\lambda^\tbeta\,\tau^{|J| + 2\tbeta}\,\eps^{|J| - 1 + 2\tbeta}}, \qquad \tbeta > 0.
	\label{I3X}
\end{align}
Similarly,  we have
\begin{align}
Y^J
	&\leq C\,M\,\frac{\theta(\lambda,\eps )\,\reps}{\lambda^\tbeta\,\eps^{|J| -1 + 2\tbeta}}\,
\sum_{j=1}^{C\,\eps ^{-1}} (\sqrt{\lambda}\,j\,\eps)^{|J| + 2\tbeta}\left|K_0^{(|J|)}(\sqrt{\lambda}\,j\,\eps )\right|\nonumber\\
	&\leq C\,M\,\frac{\theta(\lambda,\eps )\,\reps}{\lambda^\tbeta\,\eps^{|J| + 2\tbeta}}
\,\sup_{t > 0} t^{|J| + 2\tbeta}\left|K_0^{(|J|)}(t)\right|\nonumber\\
	&\leq C\,M\,\frac{\theta(\lambda,\eps )\,\reps}{\lambda^\tbeta\,\eps^{|J| + 2\tbeta}}, \qquad \tbeta > 0.
	\label{I3YCrude}
\end{align}
We then insert estimate \eqref{eq:mulambda} in \eqref{I3X} and \eqref{I3YCrude}. For $\lambda \geq \frac{1}{\reps}$, we can choose $\tbeta = 1$ and obtain
\begin{align}
|I_3^{J}|
	&\leq C\,M\,\max\Big(\frac{1}{\eps\,\reps},\sqrt{\lambda}\Big)\,\frac{\reps}{\lambda\,\eps^{|J| + 2}}\Big(\frac{\eps}{\tau^{|J| + 2}} + 1\Big)\nonumber\\
	&\leq C\,M\,\max\Big(\frac{1}{\eps\,\reps^{1/2}},1\Big)\,\frac{\reps}{\sqrt{\lambda}\,\eps^{|J| + 2}}\Big(\frac{\eps}{\tau^{|J| + 2}} + 1\Big)\nonumber\\
	&\leq C\,M\,\frac{\reps^{1/2}}{\sqrt{\lambda}\,\eps^{|J| + 3}}\Big(\frac{\eps}{\tau^{|J| + 2}} + 1\Big).
	\label{I3-CrudeLarge}
\end{align}
For $\lambda \leq \frac{1}{\reps}$ we can take $\tbeta = \frac{1}{2}$ to get
\begin{equation}
|I_3^{J}|
	\leq C\,M\,\frac{1}{\sqrt{\lambda}\,\eps^{|J|}}\,\Big(\frac{\eps}{\tau^{|J| + 1}} + 1\Big).
	\label{I3-CrudeSmall}
\end{equation}

These estimates are enough to conclude for the case $J = (0,0)$. Adding \eqref{I1Est}, \eqref{I2Est} (with $\tbeta = \frac{1}{2}$), \eqref{I3-CrudeLarge} 
and \eqref{I3-CrudeSmall}, we obtain, for all $\lambda \geq \lambda_0$,
\[
|\tueps(x)| \leq \frac{CM}{\sqrt{\lambda}}\Big(1 + \frac{\reps}{\tau^2\,\eps^2} + \frac{\eps}{\tau} + \frac{\reps^{1/2}}{\tau^2\,\eps^3}\Big(\frac{\eps}{\tau^2} + 1\Big)\Big),
\]
which proves the first desired estimate as $\tau > \kappa\,\eps^{\frac{1-2\beta}{1+2\beta}} \gg \eps$.

For $|J| = 1$ and $\lambda \leq \frac{1}{\reps}$, estimate \eqref{I3-CrudeSmall} deteriorates. We keep the 
near estimate \eqref{I3X} for $X^J$ and improve the far estimate \eqref{I3YCrude} for $Y^J$. 
First, we write
\begin{align*}
Y^{(0,1)} + Y^{(1,0)}
	&= C\,M\,\theta(\lambda,\eps )\,\eps\,\reps\,\sum_{j=1}^{C\,\eps ^{-1}} \,j\,\sqrt{\lambda}\left|K_0^{(1)}(\sqrt{\lambda}\,j\,\eps )\right|\\
	&\leq C\,M\,\frac{\theta(\lambda,\eps )\,\reps}{\eps}\,\sum_{j=1}^{C\,\eps ^{-1}} (\sqrt{\lambda}\,j\,\eps)\left|K_0^{(1)}(\sqrt{\lambda}\,j\,\eps )\right|\,\eps.
\end{align*}
Note that since $t\to t \,|K_0^{(1)}(t)|$ is decreasing and summable on $(0,\infty)$, its lower Riemann sum is bounded from above by the continuous integral. Thus
\begin{eqnarray*}
Y^{(0,1)} + Y^{(1,0)}
&\leq& C\,M\,\frac{\theta(\lambda,\eps)\,\reps}{\eps}\,\int_0^\infty \sqrt{\lambda}t\,K_0^{(1)}(\sqrt{\lambda}t)\,dt\\
&=& C\,M\,\frac{\theta(\lambda,\eps)\,\reps}{\sqrt{\lambda}\,\eps}.
\end{eqnarray*}
This inequality together with \eqref{eq:mulambda} and \eqref{I3X} yields that for $\lambda \leq \frac{1}{\reps}$,
\begin{equation}
|I_3^{(1,0)}| + |I_3^{(0,1)}|  \leq \frac{CM}{\lambda^{\tbeta}}\Big(\frac{\eps^{1- 2\tbeta}}{\tau^{1 + 2\tbeta}} + 1\Big), \qquad 0 < \tbeta < \frac{1}{2}.
	\label{I3J1Sharp}
\end{equation}
Choosing $\tbeta = \beta$ in the above estimate and in \eqref{I1Est}, \eqref{I2Est} and 
recalling \eqref{I3-CrudeLarge}, we conclude for all $\lambda \geq \lambda_0$ that 
\[
|\nabla \tueps(x)|  \leq \frac{CM}{\lambda^\beta}\Big(1 + \frac{\reps}{\tau^{2 + 2\beta}\,\eps^{2 + 2\beta}} + \frac{\eps^{1 - 2\beta}}{\tau^{1 + 2\beta}} + \frac{\reps^{1/2}}{\eps^4}\Big(\frac{\eps}{\tau^3} + 1\Big)\Big),
\]
which provides the remaining estimate as $\tau > \kappa\,\eps^{\frac{1-2\beta}{1+2\beta}}$.
\qed

We now turn to the derivation of H\"older estimates for the gradient of $\tueps$.
\begin{lemma}\label{lem:HolderX}
For any $0 < \beta < \frac{1}{2}$, $0<\nu<2\beta$, $\kappa > 0$ and $\tau > \kappa\,\eps^{\frac{1 - 2\beta}{2(1 + 2\beta - \nu)}}$ there holds
\[
[\nabla\tueps]_{C^\nu(\homegtau)} \leq \frac{C}{\lambda^{\beta - \frac{\nu}{2}}}\,\sup_{m,n} |u_{m,n}|
\]
for some constant $C$ independent of $\eps$, $\tau$ and $\lambda$.
\end{lemma}
This Lemma is proved in Appendix~\ref{sec:lem47}.

\subsection{Estimates for $\hueps$}
The residual term $\hueps$ satisfies
\begin{equation}\label{eq:defuhat}
\left\{\begin{array}{ll}
-\Delta\hueps + \lambda\,\hueps = 0 &\text{ in } \homeg\setminus \Deps	\;,\\
\hueps = \ueps - v - u_{m,n} &\text{ on } \partial D_{m,n,\eps }\;, (m,n) \in I_\eps 	\;.
\end{array}\right.
\end{equation}
We derive the following estimate, which shows that $\hueps$ is negligible compared to other contributions to $\ueps$.

\begin{lemma}\label{lem:ufEst}
For any $J \in \ZZ^2$, $\eta > 0$, and $\tau >\kappa\,\eps^{-\frac{2 + |J| +2\eta}{|J|+2\eta}}\reps^{\frac{1}{16(|J|+2\eta)}}$, 
the solution $\hueps$ of \eqref{eq:defuhat} satisfies
\[
\|\nabla^J \hueps\|_{L^\infty(\homegtau)} \leq \frac{C\,\reps^{\frac{1}{16}}}{\lambda^{\eta}}
	\left[\|f\|_{L^2(\homeg)} + \|\sqrt{\aeps}g\|_{L^2(\homeg)} + \|\bh\|_{L^\infty(\homeg)}\right]
	,
\]
where the constant $C$ depends on $J$, $\eta$, $\kappa$ and $\homeg$ only, and $\homegtau$ is defined in \eqref{eq:defhomegtau}.%
\end{lemma}

For the proof of Proposition~\ref{prop:Linfybd2d}, we shall use the following simple consequence of Lemma~\ref{lem:ufEst}.
\begin{corollary}\label{cor:estuhat}
For any $\tau > \kappa \eps^2$, the solution $\hueps$ of \eqref{eq:defuhat} satisfies 
$$
\|\hueps\|_{L^\infty(\homegtau)} + \|\nabla \hueps\|_{L^\infty(\homegtau)} + [\nabla \hueps]_{C^{0,1}(\overline{\homegtau})}
	\leq \frac{C\,\eps^2}{\lambda^{\frac{3}{2}}}\,
		\left[\|f\|_{L^2(\homeg)} + \|\sqrt{\aeps}g\|_{L^2(\homeg)} + \|\bh\|_{L^\infty(\homeg)}\right],
$$ 
where $C>0$ depends on $\kappa$ and $\homeg$ only, and $\homegtau$ is defined in \eqref{eq:defhomegtau}.%
\end{corollary}

\begin{remark}
A similar statement is possible with $\frac{\eps^2}{\lambda^{\frac{3}{2}}}$ 
replaced by $\frac{\eps^\alpha}{\lambda^\beta}$ with arbitrary $\alpha > 0$ and $\beta > 0$. 
However, we chose those powers to fix ideas, since, for example, $\|v\|_{L^\infty(\homeg)}$ and $\|\tueps\|_{L^\infty(\homegtau)}$ have 
slower decay in $\lambda$, and the restriction on $\tau$ (with respect to $\eps$) is more stringent.
\end{remark}

The proof will use the following result, whose proof will be given later. 
\begin{lemma}\label{lem:extdecay}
Given $\alpha\in H^{1/2}\left(\Sphere^1\right)$, and $s>0$, there exists a unique solution in $H^1(\RR^2\setminus B(0,1))$ on the exterior of the unit disk of $\RR^2$ of
$$
-\Delta u + s u=0 \mbox{ in } \RR^2\setminus B(0,1) \quad u=\alpha \mbox{ on } \Sphere^1,
$$ 
still denoted by $\alpha$, which decays at infinity. It satisfies
\begin{equation}
|\alpha(x)| \leq C\,e^{-\sqrt{s}\,|x|/4}\,\|\alpha\|_{L^2(\Sphere^1)}
	\text{ for } |x| \geq 3.
	\label{eq:extdecay}
\end{equation}
\end{lemma}
\begin{remark}
We do not claim that this estimate is in any way optimal, but is it sufficient for our purpose.
\end{remark}
\par\noindent
{\it Proof of Lemma~\ref{lem:ufEst}.}
We first derive a bound for $\hueps$ away from $\Deps$. For a given $(m,n) \in I_\eps$, let $\alpha_{m,n}^+$ be the solution to
\[
\left\{\begin{array}{ll}
-\Delta \alpha_{m,n}^+ + \lambda\,\alpha_{m,n}^+ = 0 &\text{ in } \RR^2 \setminus D_{m,n,\eps }	\;,\\
\alpha_{m,n}^+ = \hueps^+ &\text{ on } \partial D_{m,n,\eps }\;,\\
\alpha_{m,n}^+(x') \rightarrow 0 & \text{ as } |x'| \rightarrow \infty	\;.
\end{array}\right.
\]
By the maximum principle, $\alpha_{m,n}^+$ is non-negative in $\homeg \setminus \Deps$. Thus, again by the maximum principle, 
\begin{equation}
\hueps \leq \sum_{(m,n) \in I_\eps } \alpha_{m,n}^+
	\text{ in } \homeg \setminus \Deps
	\;.\label{ufE::Eqn01}
\end{equation}

Now, thanks to Lemma~\ref{lem:extdecay} we have
\[
|\alpha_{m,n}^+| \leq C\,\exp\left(-\frac{\sqrt{\lambda}\,\eps\,\tau}{16}\right)\,\frac{1}{\sqrt{\eps \,\reps}}\|\ueps - v - u_{m,n}\|_{L^2(\partial D_{m,n,\eps })}
	\text{ in } \homeg_{\tau/4}^\eps
	\;.
\]
On the other hand, by Poincar\'e's inequality, \eqref{eq:ape-enerbd} and Corollary~\ref{lem:GaNiV},
\begin{eqnarray*}
\|\ueps - v - u_{m,n}\|_{L^2(\partial D_{m,n,\eps })} 
&\leq& C\,\sqrt{\eps \,\reps}\,\Big[\|\nabla\ueps\|_{L^2(D_{m,n,\eps })} + \|\nabla v\|_{L^2(D_{m,n,\eps })}\Big] \\
&\leq& C\,\sqrt{\eps \,\reps}\,\left(\reps + (\eps\,\reps)^{2\varrho}\right) \times\\
		&&\qquad\qquad \times \underbrace{\left[\frac{1}{\lambda^\beta}\,\left(\|f\|_{L^2(\homeg)} 
+ \|\sqrt{\aeps}\,g\|_{L^2(\homeg)}\right)+ \|\bh\|_{L^\infty(\homeg)}\right]}_{=: \zeta}
\end{eqnarray*}
for $\beta = \frac{1}{4} \in (0,\frac{1}{2})$ and $\varrho = \frac{1}{16} \in (0,\frac{1 - 2\beta}{4})$. 
Inserting this estimate in the upper bound of $\alpha_{m,n}^+$ gives
\[
\alpha_{m,n}^+ \leq C\,\exp\left(-\frac{\sqrt{\lambda}\,\eps\,\tau}{16}\right) \left(\reps + (\eps\,\reps)^{2\varrho}\right)\,\zeta
	\text{ in } \homeg_{\tau/4}^\eps
	\;.
\]
Substituting the above estimate into \eqref{ufE::Eqn01} results in
\[
\hueps \leq C\,\exp\left(-\frac{\sqrt{\lambda}\,\eps\,\tau}{16}\right) \frac{\left(\reps + (\eps\,\reps)^{2\varrho}\right)}{\eps^2}\,\zeta
	\text{ in } \homeg_{\tau/4}^\eps
	\;.
\]
A similar lower bound is derived by repeating the argument. We have obtained
\begin{equation}\label{ufE::Eqn02}
\|\hueps\|_{L^\infty(\homeg_{\tau/4}^\eps)}
\leq C\,\exp\left(-\frac{\sqrt{\lambda}\,\eps\,\tau}{16}\right) \frac{\left(\reps + (\eps\,\reps)^{2\varrho}\right)}{\eps^2}\,\zeta.
\end{equation}

We now proceed to prove the required derivative estimate for $\hueps$. 

We will only bound $\xi=\partial_{x_1} \hueps$. The other (higher) partial derivatives can be bounded similarly. 
Fix $x \in \homeg_{\tau}^\eps$. Note that the disk $B_{3\eps\tau/4}(x) \subset \homeg_{\tau/4}^\eps$. 
Using a suitable cut-off function which is one in $B_{\eps\tau/2}(x)$ and zero outside $B_{3\eps\tau/4}(x)$ 
and a standard energy estimate, it is easy to show that
\[
\|\nabla \hueps\|_{L^2(B_{\eps\tau/2}(x))} \leq \frac{C}{\eps\tau}\|\hueps\|_{L^2(B_{3\eps\tau/4}(x_0) \setminus B_{\eps\tau/2}(x_0))}
\]
which together with \eqref{ufE::Eqn02} implies
\begin{equation}
\|\nabla \hueps\|_{L^2(B_{\eps\tau/2}(x))} 
	\leq C\,\exp\left(-\frac{\sqrt{\lambda}\,\eps\,\tau}{16}\right) \frac{\left(\reps + (\eps\,\reps)^{2\varrho}\right)}{\eps^2}\,\zeta
	.\label{ufE::Eqn03}
\end{equation}
On the other hand, by differentiating \eqref{eq:defuhat}, 
\[
-\Delta \xi + \lambda\,\xi = 0 \text{ in } \homeg\setminus\Deps
	\;.
\]
Usual interior De Giorgi estimates then show that
$$
\|\xi\|_{L^\infty(B_{\eps\tau/4}(x))} 
	\leq \frac{C}{\eps\tau} \,\|\xi\|_{L^2(B_{\eps\tau/2}(x))}.
$$
Combining this bound with \eqref{ufE::Eqn03} and noting that $x$ is arbitrary in $\omega_\tau^\eps$, we get
$$
\|\xi\|_{L^\infty(\homeg_{\tau}^\eps)} 
	\leq C\,\exp\left(-\frac{\sqrt{\lambda}\,\eps\,\tau}{16}\right) \frac{\left(\reps + (\eps\,\reps)^{2\varrho}\right)}{\tau\eps^3}\,\zeta.
$$
The required estimate for $\xi$ then follows from the simple inequality $e^{-|x|} \leq \frac{C(\eta)}{|x|^\eta}$. 
Further higher derivative estimates are obtained by repeating the above process.
\qed
\par\bigskip
To conclude, we now provide the proof of Lemma \ref{lem:extdecay}.
\par\bigskip\noindent
{\it Proof of Lemma \ref{lem:extdecay}.}
It is straightforward to check that if $\alpha\in L^2\left(\Sphere^1\right)$ has the Fourier expansion
\[
\alpha(\theta) = \sum_{k = 0}^\infty \big[a_k\,\cos\,k\theta + b_k\,\sin\,k\theta\big]
	\;,
\]
then it extends to a solution of $(-\Delta + s)\alpha = 0$ on $\RR^2 \setminus D_1$ which vanishes at infinity by
\[
\alpha(r,\theta) = \sum_{k=0}^\infty \frac{1}{K_k(\sqrt{s})}\big[a_k\,\cos\,k\theta + b_k\,\sin\,k\theta\big] K_k(\sqrt{s}\,r)
	\;,
\]
where $K_k$ is the second modified Bessel Function of order $k$. The function $K_k$ has a representation as follows (see e.g. \cite{NIST-10})
\[
K_k(t) = \int_0^\infty e^{-t\,\cosh \tau}\,\cosh k\tau\,d\tau
	\;.
\]
This implies the obvious bound
\[
\frac{K_k(\sqrt{s}\,r)}{K_k(\sqrt{s})} \leq e^{-\sqrt{s}(r - 1)}
	, \qquad k \geq 0.
\]
On the other hand, for $k \geq 1$, we also have
\[
K_k(t) \leq \int_0^\infty e^{-\frac{1}{2}\,t\,e^\tau}\,e^{k\tau}\,d\tau
	= \frac{2^k}{t^k}\int_0^\infty e^{-\tau}\,\tau^{k-1}\,d\tau = \frac{2^k\,(k-1)!}{t^k}
	\;,
\]
and
\[
K_k(t) \geq \frac{1}{2}\int_0^\infty e^{-t\,e^\tau}\,e^{k\tau}\,d\tau
	= \frac{1}{2\,t^k}\int_0^\infty e^{-\tau}\,\tau^{k-1}\,d\tau = \frac{(k-1)!}{2\,t^k}
	\;.
\]
It follows that
\[
\frac{K_k(\sqrt{s}\,r)}{K_k(\sqrt{s})} \leq \frac{2^{k+1}}{r^k}
	\;.
\]
This is also valid for $k = 0$.

In particular, the above two bounds show that the series for $\alpha(r,\theta)$ converges absolutely for $r > 3$ and,
\begin{align}
|\alpha(r,\theta)| 
	&\leq C\,e^{-\sqrt{s}(r - 1)/2}\,\sum_{k=0}^\infty \frac{2^{k/2}}{r^{k/2}}\,[|a_k| + |b_k|] \nonumber\\
	&\leq C\,e^{-\sqrt{s}(r - 1)/2}\,\frac{r^2}{r^2 - 4}\,\Big\{\sum_{k=0}^\infty [a_k^2 + b_k^2]\Big\}^{1/2}\nonumber\\
	&\leq C\,e^{-\sqrt{s}\,r/4}\,\|\alpha\|_{L^2(\Sphere^1)}
	\text{ for } r > 3\;
\end{align}
as announced.
\qed
\section{\label{sec:ABCD} Interior estimates for the boundary value problem}

In this section, we detail regularity estimates for the solutions $\tilde{W}_{\eps,1}$ 
and $\tilde{W}_{\eps,2}$ of the boundary value problems \eqref{eq:BDY-3DC}
which appear in the proof of Theorem~\ref{thm:Regularity-3d**}.

\subsection{Regularity of $\tilde{W}_{\eps,1}$} 
Our result is the following
\begin{proposition}\label{pro:tilde-W1}
Let $\Omegtau = \homegtau \times(-L,L)$. For $\kappa > 0$,  $\tau > \kappa\,\eps^{\frac{1-\eta}{2(1+\eta)}}$
with $\eta
\in\left(\frac{1}{2},1\right)$, and $0<\nu < 2\left(\eta - \frac{1}{2}\right)$, the solution $W_{\eps,1}$ of \eqref{eq:defweps1} satisfies
\begin{align*}
&\left\| \tilde{W}_{\eps,1} - \tilde{V}_{1} \right\|_{L^\infty\left(\Omegtau\right)}
	+ \left\|\nabla \tilde{W}_{\eps,1}  -\nabla \tilde{V}_{1} \right\|_{L^\infty\left(\Omegtau\right)}
	\leq C(\kappa,\eta) \|\fib\|_{C^1(\bar{\Omega})},\\
&\left[\nabla \tilde{W}_{\eps,1} - \nabla \tilde{V}_{1}\right]_{C^\nu\left(\Omegtau\right)}
	\leq C(\kappa,\eta,\nu) \|\fib\|_{C^1(\bar\Omega)}.
\end{align*}
\end{proposition}
\begin{proof}
We apply a Fourier series decomposition in $x_3$. We write
\begin{eqnarray*}
\tilde{W}_{\eps,1}(x^\prime, x_3) 
	&\sim& \sum_{n=1}^\infty \twepsn(x')\,\sin\Big(\frac{n\pi}{2}\big(\frac{x_3}{L} + 1\big)\Big),\\
\tilde{V}_{1}(x^\prime, x_3) 
	&\sim& \sum_{n=1}^\infty \tvn(x')\,\sin\Big(\frac{n\pi}{2}\big(\frac{x_3}{L} + 1\big)\Big).
\end{eqnarray*}
As $\tilde{W}_{\eps,1} - \tilde{V}_{1}$ is harmonic therefore smooth in $\Omega \setminus \Ceps$, 
the Fourier expansions of $\tilde{W}_{\eps,1} - \tilde{V}_{1}$ 
and $\nabla(\tilde{W}_{\eps,1} - \tilde{V}_{1})$ converge pointwise to themselves in $\Omegtau$. 

Now $\twepsn\in H^1_0(\omega)$ satisfies
$$
\divop_2(\aeps\nabla_2 \twepsn) + \frac{n^2\pi^2}{4L^2}\,\aeps\,\twepsn  = \divop_2(\bh_n)
$$
where $\bh_n=(h_n^{1}, h_n^2)$ is given by 
$$
h_n^{i} = \int_{-L}^{L} \partial_{i} \phi_1 \sin\Big(\frac{n\pi}{2}\big(\frac{x_3}{L} + 1\big)\Big). 
$$
Since $\partial_{33} \phi_{1} =0$, we see from an integration by parts that
\begin{equation}\label{eq:bdhni}
\|h_n^{i}\|_{L^\infty(\homeg)} \leq \frac{C}{n} \|\nabla \phi_1\|_{L^\infty(\HOmeg)} \leq \frac{C}{n} \| \fib\|_{C^{1}(\bar\HOmeg)}.
\end{equation}
Lemma~\ref{lem:dGNM2}, applied with $\tilde p =\infty$, shows that for any $\eta<1$, 
\begin{equation}\label{eq:maxprinchni}
\|\twepsn\|_{L^\infty(\bar\Omega)} \leq \frac{C}{n^\eta} \|h_n^{i}\|_{L^\infty(\homeg)},
\end{equation}
Proposition~\ref{prop:Linfybd2d} together with \eqref{eq:bdhni} and \eqref{eq:maxprinchni} now shows that, for every $0<\nu<\eta<1$,
$$
 \|\twepsn - \tvn \|_{L^\infty(\homegtau)} + \|\nabla_2 \twepsn - \nabla_2 \tvn \|_{L^\infty(\homegtau)} 
\leq  \frac{C}{n^{2\eta+1}} \| \fib\|_{C^{1}(\bar\HOmeg)},
$$
and 
$$
[\nabla_2 \twepsn - \nabla_2 v]_{C^\nu(\homegtau)} \leq \frac{C}{n^{2\eta + 1 - \nu}} \| \fib\|_{C^{1}(\bar\HOmeg)},
$$
Following a variant of the proof of Theorem~\ref{thm:Regularity-3d}, we obtain consecutively the following estimates using the indicated choice of $\eta$,
\begin{align*}
&\left\|\tilde{W}_{\eps,1} - \tilde{V}_{1}\right\|_{L^\infty\left(\Omegtau\right)} 
+ \left\|\nabla_2 \tilde{W}_{\eps,1}- \nabla_2 \tilde{V}_{1} \right\|_{L^\infty\left(\Omegtau\right)}\leq  C \| \fib\|_{C^{1}(\bar\HOmeg)} 
	& \text{ with any } \eta > 0
	,\\
&[\nabla_2 \tilde{W}_{\eps,1} - \nabla_2 \tilde{V}_{1}]_{C^\nu(\Omegtau)} \leq C \| \fib\|_{C^{1}(\bar\HOmeg)}
	& \text{ with any } \eta > \frac{\nu}{2}
	,\\ 
&\left\|\partial_{x_3} \tilde{W}_{\eps,1}- \partial_{x_3} \tilde{V}_{1} \right\|_{L^\infty\left(\Omegtau\right)}\leq  C \| \fib\|_{C^{1}(\bar\HOmeg)} 
	& \text{ with any } \eta > \frac{1}{2}
	,\\
&[\partial_{x_3} \tilde{W}_{\eps,1} - \partial_{x_3} \tilde{V}_{1}]_{C^\nu(\Omegtau)} \leq C \| \fib\|_{C^{1}(\bar\HOmeg)}
	& \text{ with any } \eta > \frac{1}{2} + \frac{\nu}{2}
	.
\end{align*}
This completes the proof.
\qed\end{proof}

\subsection{Regularity of $\tilde{W}_{\eps,2}$} 
Our result is the following
\begin{proposition}\label{pro:tilde-W2}
Let $\Omegtau = \homegtau \times(-L,L)$. For $\kappa > 0$,  $\tau > \kappa\,\eps^{\frac{1-\eta}{2(1+\eta)}}$ with $\eta
\in\left(\frac{2}{3},1\right)$, and $0<\nu <3\left(\eta - \frac{2}{3}\right)$, 
the solutions $\tilde{W}_{\eps,2}$ and $\tilde V_{\eps,2}$ of \eqref{eq:BDY-4DC} and \eqref{eq:BDY-TV01}, respectively, satisfy
\begin{align*}
&\left\| \tilde{W}_{\eps,2} - \tilde{V}_{\eps,2} \right\|_{L^\infty\left(\Omegtau\right)}
	+ \left\|\nabla \tilde{W}_{\eps,2}  -\nabla \tilde{V}_{\eps,2} \right\|_{L^\infty\left(\Omegtau\right)}
	\leq C(\kappa,\eta) \eps^\sigma \|\fib\|_{C^1(\bar{\Omega})},\\
&\left[\nabla \tilde{W}_{\eps,2} - \nabla \tilde{V}_{\eps,2}\right]_{C^\nu\left(\Omegtau\right)}
	\leq C(\kappa,\eta,\nu) \eps^\sigma \|\fib\|_{C^1(\bar{\Omega})},
\end{align*}
\end{proposition}
\begin{proof}
As before, we decompose $\tilde{W}_{\eps,2}$ in a Fourier series along the third direction, i.e. 
\begin{align*}
\tilde{W}_{\eps,2} &\sim \sum_{n=1}^\infty \ttwepsn \sin \left(\frac{n\pi}{2L}\left(x_3 +L\right) \right),\\
\tilde{V}_{\eps,2} &\sim \sum_{n=1}^\infty \ttvn \sin \left(\frac{n\pi}{2L}\left(x_3 +L\right) \right).
\end{align*}
As usual, $\tilde W_{\eps,2} - \tilde V_{\eps,2}$ is harmonic and regular in $\Omega \setminus \Ceps$. 
Thus, we can sum the estimates on the Fourier coefficients to obtain estimates 
on $\tilde W_{\eps,2} - \tilde V_{\eps,2}$ and $\nabla(\tilde W_{\eps,2} - \tilde V_{\eps,2})$ in $\Omegtau$.

The problem satisfied by $\ttwepsn$ is now, in $H^1_0(\omega)$,
$$
-\mbox{div}_2\left(\aeps \nabla_2 \ttwepsn \right) + \lambda \aeps \ttwepsn = \aeps\, A_{\eps,n} + \aeps B_{\eps,n} + C_{\eps,n},
$$
where $\lambda=\frac{n^2\pi^2}{4L^2}$, and
\begin{eqnarray*}
A_{\eps,n} &=& \int_{-L}^L \sin \left(\sqrt{\lambda}\left(t +L\right)\right) \partial_{3} \zeta(\cdot,t)\partial_{3} W_{\eps,2}(\cdot,t) dt,\\
B_{\eps,n} &=& \int_{-L}^L \sin \left(\sqrt{\lambda}\left(t +L\right)\right) \partial_{33}\zeta(\cdot,t) W_{\eps,2}(\cdot,t)dt,\\
C_{\eps,n} &=& \int_{-L}^L \sin \left(\sqrt{\lambda}\left(t +L\right)\right)g_\eps(\cdot,t)dt.
\end{eqnarray*}
Proposition~\ref{prop:Linfybd2d} together with Proposition~\ref{pro:reg-ABC}, now shows that for all $0<\nu<\eta<1$,
\[
 \|\ttwepsn - \ttvn \|_{L^\infty(\homegtau)} + \|\nabla_2 \ttwepsn - \nabla_2 \ttvn \|_{L^\infty(\homegtau)} 
\leq  \eps^{\sigma}\frac{C}{n^{3\eta}} \| \fib\|_{C^{1}(\bar\HOmeg)},
\]
and
\[
[\nabla_2 \ttwepsn - \nabla_2 \ttvn ]_{C^\nu(\homegtau)} \leq  \eps^{\sigma}\frac{C}{n^{3\eta-\nu}} \| \fib\|_{C^{1}(\bar\HOmeg)}.
\]
Following a variant of the proof of Theorem~\ref{thm:Regularity-3d}, we obtain consecutively the following estimates using the indicated choice of $\eta$,
\begin{align*}
&\left\|\tilde{W}_{\eps,2} - \tilde{V}_{2}\right\|_{L^\infty\left(\Omegtau\right)} 
+ \left\|\nabla_2 \tilde{W}_{\eps,2}- \nabla_2 \tilde{V}_{2} \right\|_{L^\infty\left(\Omegtau\right)}\leq  C\,\eps^{\sigma} \| \fib\|_{C^{1}(\bar\HOmeg)} 
	& \text{ with any } \eta > \frac{1}{3}
	,\\
&[\nabla_2 \tilde{W}_{\eps,2} - \nabla_2 \tilde{V}_{2}]_{C^\nu(\Omegtau)} \leq C\,\eps^\sigma\, \| \fib\|_{C^{1}(\bar\HOmeg)}
	& \text{ with any } \eta > \frac{1}{3} + \frac{\nu}{3}
	,\\ 
&\left\|\partial_{x_3} \tilde{W}_{\eps,2}- \partial_{x_3} \tilde{V}_{2} \right\|_{L^\infty\left(\Omegtau\right)}\leq  C \eps^{\sigma}\| \fib\|_{C^{1}(\bar\HOmeg)} 
	& \text{ with any } \eta > \frac{2}{3}
	,\\
&[\partial_{x_3} \tilde{W}_{\eps,2} - \partial_{x_3} \tilde{V}_{2}]_{C^\nu(\Omegtau)} \leq C \eps^{\sigma}\| \fib\|_{C^{1}(\bar\HOmeg)}
	& \text{ with any } \eta > \frac{2}{3} + \frac{\nu}{3},
\end{align*}
which concludes our argument.
\qed\end{proof}

\begin{proposition}\label{pro:reg-ABC}
We have the following estimate
$$
\|\sqrt{\aeps}A_{\eps,n}\|_{L^2(\homeg)} + \|\sqrt{\aeps}B_{\eps,n}\|_{L^2(\homeg)}+\|\sqrt{\aeps}C_{\eps,n}\|_{L^2(\homeg)} 
\leq \frac{C}{\sqrt{\lambda}} \eps^{\sigma+2} \|\Phi\|_{C^1(\bar\Omega)}.
$$
As a consequence,
$$
\| \ttwepsn \|_{L^\infty(\homeg)} \leq  C(\eta) \frac{\eps^\sigma}{\lambda^{\eta}} \|\Phi\|_{C^1(\bar\Omega)}.
$$
for any $\eta<1$.
\end{proposition}

\begin{proof}
Note that \eqref{eq:reg-g} shows that $C_{\eps,n}$ satisfies
$$
\|C_{\eps,n}\|_{L^\infty(\omega)} \leq \frac{C}{\sqrt{\lambda}} \left\Vert  \phi_{2} \right\Vert _{L^{\infty}(\Omega)}.
$$
Turning to $B_{\eps,n}$ an integration by parts shows that
$$
\left|B_{\eps,n}\right|
	\leq \frac{C}{\sqrt{\lambda}}\left(\|W_{\eps,2}\|_{L^\infty(\Omega)} 
+ \int_{-L}^{L}\mathbf{1}_{\omega_3 \times (-l',l')} |\partial_{3}W_{\eps,2}(\cdot,t)|dt \right),
$$
which in turn shows that
$$
\int_\homeg \aeps \left|B_{\eps,n}\right|^2 dx^\prime 
\leq 
\frac{C}{\lambda}\left(\|W_{\eps,2}\|_{L^\infty(\Omega)} +  \int_{\omega_3 \times (-l',l')} \aeps |\partial_{3}W_{\eps,2}|^2 dx \right).
$$
Thanks to \eqref{eq:maxprinc-phi2} and the interior estimate \eqref{eq:reg-uepsd3}, we have obtained that
$$
\| \sqrt{\aeps} B_{\eps,n} \|_{L^2(\omega)} \leq \frac{C}{\sqrt{\lambda}} \left\Vert  \phi_{2} \right\Vert _{L^{\infty}(\HOmeg)}.
$$
We proceed with $A_{\eps,n}$ in a similar way. After an integration by parts, we derive that
$$
\left|A_{\eps,n}\right|^2\leq \frac{C}{\lambda} 
\left(  \int_{-L}^{L}\mathbf{1}_{\omega_3 \times (-l',l')} \left(|\partial_{3}W_{\eps,2}(\cdot,t)|^2 + |\partial_{33}W_{\eps,2}(\cdot,t)|^2\right) dt\right)
$$
Multiplying this quantity by $\aeps$, and integrating we obtain thanks to \eqref{eq:maxprinc-phi2} and \eqref{eq:reg-uepsd3}
$$
\| \sqrt{\aeps} A_{\eps,n} \|_{L^2(\omega)} \leq \frac{C}{\sqrt{\lambda}} \left\Vert  \phi_{2} \right\Vert _{L^{\infty}(\HOmeg)}.
$$ 
We have obtained that
$$
\|\sqrt{\aeps}A_{\eps,n}\|_{L^2(\homeg)} + \|\sqrt{\aeps}B_{\eps,n}\|_{L^2(\homeg)}+\|C_{\eps,n}\|_{L^2(\homeg)} 
\leq 
\frac{C}{\sqrt{\lambda}} \left\Vert  \phi_{2} \right\Vert _{L^{\infty}(\HOmeg)}.
$$
The first assertion now follows from Proposition~\ref{pro:phi1}.

Finally, using Lemma~\ref{lem:dGNM}, we know that for any $1 < \alpha < 2$, and any $\beta<1-\alpha/2$, 
\begin{eqnarray*}
\left\Vert \ttwepsn \right\Vert _{L^{\infty}(\homeg)} 
&\leq& \frac{C(\alpha,\beta)}{\displaystyle \eps^{\alpha}\lambda^{\beta}}\left(\|\sqrt{\aeps}A_{\eps,n}\|_{L^2(\homeg)} + 
\|\sqrt{\aeps}B_{\eps,n}\|_{L^2(\homeg)}+\|C_{\eps,n}\|_{L^2(\homeg)} \right) \\
                                          &\leq& C(\alpha,\beta) \frac{\eps^{\sigma+2-\alpha}}{\lambda^{1/2+\beta}} 
\left\Vert  \Phi \right\Vert _{C^{1}(\Omega)} \\
                                          &\leq& C(\beta) \eps^{\sigma} \lambda^{-1/2 -\beta}\left\Vert  \Phi \right\Vert _{C^{1}(\Omega)}. 
\end{eqnarray*}
\qed\end{proof}


\begin{acknowledgements}
Yves Capdeboscq is supported by the EPSRC Science and Innovation award to the Oxford Centre for Nonlinear PDE (EP/E035027/1). 
This paper was written in part during Yves Capdeboscq's stay in the Institute for Advanced Study, and he would like to gratefully 
acknowledge the fantastic time he had there.
This work was started when Luc Nguyen was in the Oxford Centre for Nonlinear PDE, where he spent a wonderful year. 
He would like to thank the centre for its financial support and fostering an encouraging environment. 
This material is based upon work supported by the National Science Foundation under agreement No. DMS-0635607. 
Any opinions, findings and conclusions or recommendations expressed in this material are those of the authors 
and do not necessarily reflect the views of the National Science Foundation.
\end{acknowledgements}

\appendix
{\normalsize
\section{\label{sec:AnnexA}Proofs of technical lemmas and propositions}

\subsection{\label{sec:pfax3}Proof of Proposition~\ref{pro:analx3}}
Note that we have
\[
\left\{\begin{array}{rl}
\displaystyle -\,\Delta \uhom+\gamma\left(\uhom-\vhom\right)=0 & \mbox{in }\Omega_0
	,\\
-\,\kappa\,\partial^2_{33}\vhom+\gamma\left(\vhom-\uhom\right)=0 & \mbox{in }\Omega_0.
\end{array}\right.
\]
And a simple induction shows that $\uhom$ and $\vhom$ are infinitely differentiable with respect to $x_3$ in $\Omega_0$. 
To gain analyticity, we resort to estimates. 
Fix a ball $B(p^0,R)$ in $\Omega_0$. We first derive an integral representation for $\uhom(p^0)$. 
For any $r > 0$, the function
\[
G_{p^0,r}(x) := \frac{\cosh(\sqrt{\gamma}|x - p^0|)}{4\pi|x - p^0|} - \coth(\sqrt{\gamma}r)\,\frac{\sinh(\sqrt{\gamma}|x - p^0|)}{4\pi|x - p^0|}
\]
satisfies
\[
\left\{\begin{array}{ll}
(-\Delta + \gamma) G_{p^0,r} = \delta_{p^0} &\text{ in } B(p^0,r),\\
 G_{p^0,r} = 0 &\text{ on } \partial B(p^0,r).
\end{array}\right.
\]
A direct application of Green's formula then yields that for $0 < r < R$, there holds
\begin{align*}
\uhom(p^0) 
	&= \gamma\int_{B(p^0,r)} G_{p^0,r}(x)\,\vhom(x)\,dx
		+ \frac{\sqrt{\gamma}}{4\pi\,r\,\sinh(\sqrt{\gamma}r)}\,\int_{\partial B(p^0,r)} \uhom(x)\,d\sigma(x).
\end{align*}
Here we have used
\[
\frac{\partial}{\partial n}G_{p^0,r}\Big|_{\partial B(p^0,r)} = - \frac{\sqrt{\gamma}}{4\pi\,r\,\sinh(\sqrt{\gamma}r)}.
\]
It follows that
\begin{align*}
&\uhom(p^0)\int_0^R r\,\sinh(\sqrt{\gamma}r)\,dr\\
	&\qquad =  \gamma \int_0^R r\,\sinh(\sqrt{\gamma}r)\,dr \int_{B(p^0,r)} G_{p^0,r}(x)\,\vhom(x)\,dx
		 + \frac{\sqrt{\gamma}}{4\pi}\,\int_{B(p^0,R)} \uhom(x)\,dx.
\end{align*}
Applying this identity to $\partial_{x_3}^{(n+1)} \uhom$ we obtain
\begin{align*}
\partial_{x_3}^{(n+1)}\uhom(p^0)\int_0^R r\,\sinh(\sqrt{\gamma}r)\,dr
	&\qquad =  \gamma \int_0^R r\,\sinh(\sqrt{\gamma}r)\,dr \int_{B(p^0,r)} G_{p^0,r}(x)\,\partial_{x_3}^{(n+1)}\vhom(x)\,dx\\
		&\qquad\qquad+ \frac{\sqrt{\gamma}}{4\pi}\,\int_{B(p^0,R)} \partial_{x_3}^{(n+1)}\uhom(x)\,dx\\
	&\qquad =  \gamma \int_0^R r\,\sinh(\sqrt{\gamma}r)\,dr \int_{B(p^0,r)}G_{p^0,r}(x)\,\partial_{x_3}^{(n+1)}\vhom(x)\,dx\\
		&\qquad\qquad+ \frac{\sqrt{\gamma}}{4\pi}\,\int_{\partial B(p^0,R)} \partial_{x_3}^{(n)}\uhom(x)\,\frac{x_3 - p^0_3}{|x - p_0|}\,d\sigma(x).
\end{align*}
Where we used the notation $p^0=(p^0_1,p^0_2,p^0_3)$. 
Noting that $0 \leq G_{p^0,r}(x) \leq \frac{\cosh(\sqrt{\gamma}|x - p^0|)}{|x - p^0|}$, we deduce that
\begin{align*}
&\left|\partial_{x_3}^{(n+1)}\uhom(p^0)\right|\,\int_0^R r\,\sinh(\sqrt{\gamma}r)\,dr\\
	&\qquad \leq C\Big[\int_0^R  r\,\sinh(\sqrt{\gamma}r)\,dr \int_{0}^r s\,\cosh(\sqrt{\gamma}s)\,ds \|\partial_{x_3}^{(n+1)}\vhom\|_{L^\infty(B(p^0,R))}\\
		&\qquad\qquad + R^2\,\|\partial_{x_3}^{(n)}\uhom\|_{L^\infty(B(p_0,R))}\Big],
\end{align*}
and we have bounded $\partial_{x_3}^{(n+1)}\uhom(p^0)$ in terms of lower derivatives by
\begin{equation}
\left|\partial_{x_3}^{(n+1)}\uhom(p^0)\right|
	\leq C\Big[R^2\|\partial_{x_3}^{(n+1)}\vhom\|_{L^\infty(B(p^0,R))}
		+ \frac{1}{R}\|\partial_{x_3}^{(n)}\uhom\|_{L^\infty(B(p_0,R))}\Big].
	\label{eq:VtoW}
\end{equation}

Let us now turn to $\vhom$. For $x_3 \in (p^0_3 - R, p^0_3 + R)$ we have
\begin{align*}
\partial_{x_3}^{(n)} \vhom(p^0_1,p^0_2,x_3)
	&= \alpha_+\,\exp\left(\sqrt{\frac{\gamma}{\kappa}}x_3\right)
		+ \alpha_-\,\exp\left(-\sqrt{\frac{\gamma}{\kappa}}x_3\right)\\
		&+ \sqrt{\frac{\gamma}{\kappa}}
\int_{p^0_3 - R}^{x_3}\sinh\left(\sqrt{\frac{\gamma}{\kappa}}(s-x_3)\right)\,\partial_{x_3}^{(n)} \uhom(p^0_1,p^0_2,s)\,ds
\end{align*}
where the constants $\alpha_\pm$ satisfy
\[
|\alpha_\pm| \leq C\Big[\frac{1}{R}\|\partial_{x_3}^{(n)}\vhom\|_{L^\infty(B(p^0,R))} + R\,\|\partial_{x_3}^{(n)}\uhom\|_{L^\infty(B(p^0,R))}\Big].
\]
In particular,
\begin{align*}
\partial_{x_3}^{(n+1)} \vhom(p^0)
	&= \alpha_+\,\sqrt{\frac{\gamma}{\kappa}}\,\exp\left(\sqrt{\frac{\gamma}{\kappa}}p^0_3\right)
		- \alpha_-\,\sqrt{\frac{\gamma}{\kappa}}\,\exp\left(-\sqrt{\frac{\gamma}{\kappa}}p^0_3\right)\\
		&\qquad - \frac{\gamma}{\kappa}
\int_{p^0_3 - R}^{p^0_3}\cosh\left(\sqrt{\frac{\gamma}{\kappa}}(s- p^0_3)\right)\,\partial_{x_3}^{(n)} \uhom(p^0_1,p^0_2,s)\,ds
\end{align*}
which provides a bound of  $\partial_{x_3}^{(n+1)}\vhom(p^0)$ in terms of lower derivatives given by
\begin{equation}
\left|\partial_{x_3}^{(n+1)}\vhom(p^0)\right|
	\leq C\Big[R\|\partial_{x_3}^{(n)}\uhom\|_{L^\infty(B(p^0,R))}
		+ \frac{1}{R}\|\partial_{x_3}^{(n)}\vhom\|_{L^\infty(B(p_0,R))}\Big].
	\label{eq:WtoV}
\end{equation}

Using \eqref{eq:WtoV} for the first term on the right hand side of \eqref{eq:VtoW}, we see for all $R$ sufficiently small that
\[
\left|\partial_{x_3}^{(n+1)}\uhom(p^0)\right|
	\leq C\Big[R\|\partial_{x_3}^{(n)}\vhom\|_{L^\infty(B(p^0,R))}
		+ \frac{1}{R}\|\partial_{x_3}^{(n)}\uhom\|_{L^\infty(B(p_0,R))}\Big].
\]
This inequality together with \eqref{eq:WtoV} shows that
\begin{equation}
\left|\partial_{x_3}^{(n+1)}\uhom(p^0)\right| + \left|\partial_{x_3}^{(n+1)}\vhom(p^0)\right|
	\leq \frac{C_0}{R}\Big[\|\partial_{x_3}^{(n)}\uhom\|_{L^\infty(B(p^0,R))}
		+ \|\partial_{x_3}^{(n)}\vhom\|_{L^\infty(B(p_0,R))}\Big].
	\label{eq:Key}
\end{equation}
To conclude, we will show by induction that \eqref{eq:Key} implies
\begin{equation}
\left|\partial_{x_3}^{(n)}\uhom(p^0)\right| + \left|\partial_{x_3}^{(n)}\vhom(p^0)\right|
	\leq \frac{C_0^n\,n^n}{R^n}\Big[\|\uhom\|_{L^\infty(B(p^0,R))}
		+ \|\vhom\|_{L^\infty(B(p_0,R))}\Big].
	\label{eq:Main}
\end{equation}
Since this holds for any value of $n$, analyticity of $\uhom$ and $\vhom$ in the $x_3$ direction then follows from Sterling's formula.

By \eqref{eq:Key}, the above estimate holds for $n = 1$. Assume that it holds for some $n$. Using estimate \eqref{eq:Key} we find,
\begin{align*}
\left|\partial_{x_3}^{(n+1)}\uhom(p^0)\right| + \left|\partial_{x_3}^{(n+1)}\vhom(p^0)\right|
	\leq \frac{C_0(n+1)}{R}\Big[\|\partial_{x_3}^{(n)}\uhom\|_{L^\infty(B(p^0,\frac{R}{n+1}))}
		+ \|\partial_{x_3}^{(n)}\vhom\|_{L^\infty(B(p_0,\frac{R}{n+1}))}\Big].
\end{align*}
Note that every point in $B(p_0,\frac{R}{n+1})$ is contained in a ball of radius $\frac{nR}{n+1}$ which is contained in $B(p^0,R)$. 
Thus, applying \eqref{eq:Main} to the right hand side we obtain
\begin{align*}
&\left|\partial_{x_3}^{(n+1)}\uhom(p^0)\right| + \left|\partial_{x_3}^{(n+1)}\vhom(p^0)\right|\\
	&\qquad \leq \frac{C_0(n+1)}{R}\,\frac{C_0^n\,n^n}{(\frac{nR}{n+1})^n}\Big[\|\uhom\|_{L^\infty(B(p^0,R)}
		+ \|\vhom\|_{L^\infty(B(p_0,R)}\Big]\\
	&\qquad = \frac{C_0^{n+1}(n+1)^{n+1}}{R^{n+1}}\Big[\|\uhom\|_{L^\infty(B(p^0,R)}
		+ \|\vhom\|_{L^\infty(B(p_0,R)}\Big].
\end{align*}
which is our induction thesis.

\subsection{\label{sec:phi1phi2}Proof of Proposition~\ref{pro:phi1}}
Most properties are easily verified by inspection. Regarding $\phi_1$ let us check its regularity in the support of  
$\nabla {c_\eps^{\sigma}}\times[-L,L]$. 
The definition of $\phi_1$ being local, it suffices to look at one cell, centered in $(0,0)$. We have 
$$
\partial_{3}\phi_{1}  =  \partial_{3}\phi_L  \, {c_\eps^{\sigma}} +\left(1-{c_\eps^{\sigma}} \right)
\partial_{3}\phi_L\left(0,0,x_{3}\right),
$$
and for $i=1,2$,
\begin{eqnarray*}
\left|\partial_{i}\phi_{1} \right|
 & = & \left|\partial_{i}
\phi_L \, {c_\eps^{\sigma}}
+\frac{1}{\eps^2 \ln \left( \eps^{\sigma} /(2\eps\, r_\eps)\right)} \frac{x_{i}}{x_{1}^{2}+x_{2}^{2}}
\left(x_{1}\partial_{1}\phi_L\left(\zeta,x_{3}\right)+x_{2}\partial_{2}\phi_L\left(\zeta,x_{3}\right)\right)\right|\\
 & \leq & C \left(\left\|\partial_{1}\phi_L\right\|_{\infty} + \left\|\partial_{2}\phi_L\right\|_{\infty}\right),
\end{eqnarray*}
since, thanks to \eqref{eq:Cond1},
$$
\eps^2 \ln \left(\eps^{\sigma-1} / r_\eps\right) \leq C.
$$
The conclusion for $\phi_1$ follows.

Regarding $\phi_2$, note that
$$
\phi_2(x)=\sum_{(m,n) \in I_\eps}\left(1-{c_\eps^{\sigma}}\right)\mathbf{1}_{m,n,\eps}\left(x\right)\left(\phi_L\left(m\eps,n\eps,x_{3}\right)-\phi_{L}(x)\right).
$$
Again, the definition of $\phi_2$ is local. In the cell $[-\eps/2, \eps/2]^3$, we have
\begin{eqnarray*}
\left|\phi_2\right| &\leq& C r \left\|\nabla \phi_L\right\|_{L^\infty(\Omega)}  \mbox{ for }  r\leq\eps r_\eps, \\
\left|\phi_2\right| &\leq& C r \frac{\ln(\eps^\sigma/2) -\ln r}{\ln(\eps^\sigma/2)-\ln(\eps r_\eps)} 
\left\|\nabla \phi_L\right\|_{L^\infty(\Omega)}\mbox{ for } \eps r_\eps\leq r\leq \eps^\sigma/2, \\
\phi_2 &=& 0 \mbox{ for } \eps^\sigma \leq 2r, 
\end{eqnarray*}
with $r=\sqrt{x_1^2 +x_2^2}$. Thanks to \eqref{eq:Cond1}, this yields the global bound
$$
\left|\phi_2\right| \leq C \eps^{\sigma+2}  \left\|\nabla \phi_L\right\|_{L^\infty(\Omega)}. \\
$$

\subsection{\label{sec:lem46} Proof of Lemma~\ref{SCBdryGrad}}

Without loss of generality, we can assume that the centre of $D_{m,n,\eps }$ is the origin. 

Note that from \eqref{TU-E1x} we know that
\[
\|\tueps\|_{L^\infty(\homeg \setminus \Deps)} \leq \sup_{m,n} |u_{m,n}| =: M
	\;.
\]
Remember that two linearly independent radial solutions of 
$$
-\Delta u +\lambda u = 0
$$
for $|x|>0$ are $I_0\left(\sqrt{\lambda} \cdot \right)$ and $K_0\left(\sqrt{\lambda} \cdot \right)$, the modified Bessel 
functions of the first and second kind.
They are defined by 
\begin{align}
K_0(r) 
	&= \int_0^\infty e^{-r\,\cosh t}\,dt	\;,\label{eq:K0def}\\
I_0(r)
	&= \sum_{n=0}^{\infty}\frac{\left(r^{2}/4\right)^{n}}{\left(n!\right)^{2}}\;,\label{eq:I0def}
\end{align}
see, e.g. \cite{NIST-10}. Introduce
\[
\psi(x) = \psi(|x|) = \alpha\,I_0(\sqrt{\lambda}|x|) + \beta\,K_0(\sqrt{\lambda}|x|),
\]
where $\alpha$ and $\beta$ are chosen so that
\begin{align*}
\psi(\eps \,\reps) = u_{m,n}, \qquad \psi(\eps /2) = M.
\end{align*}
Thanks to the maximum principle, $\tueps(x) \leq \psi(x)$. 
Since the two functions agree on $\partial D_{m,n,\eps }$, the normal derivative of $\psi$ at $|x|=\eps  \reps$ gives 
an upper bound for  the normal derivative of $\tueps$ at $\partial D_{m,n,\eps }$:
\begin{equation}
\sup_{\partial D_{m,n,\eps}} \frac{\partial \tueps}{\partial n} \leq |\psi'(\eps\,\reps)|.
	\label{eq:Normaltueps-Psi}
\end{equation}

We thus proceed to estimate $|\psi'(\eps\,\reps)|$. The constants $\alpha$ and $\beta$ are given by
\begin{eqnarray*}
\alpha & = & \frac{u_{m,n} K_{0}\left(\sqrt{\lambda}\eps /2\right)
                                  - M K_{0}\left(\sqrt{\lambda}\reps\eps \right)}{
I_{0}\left(\sqrt{\lambda}\reps\eps \right)K_{0}\left(\sqrt{\lambda}\eps /2\right)
-K_{0}\left(\sqrt{\lambda}\reps\eps \right)I_{0}\left(\sqrt{\lambda}\eps /2\right)
},\\
\beta & = & \frac{-u_{m,n} I_{0}\left(\sqrt{\lambda}\eps /2\right)
                               +M    I_{0}\left(\sqrt{\lambda}\reps\eps \right)}{
I_{0}\left(\sqrt{\lambda}\reps\eps \right)K_{0}\left(\sqrt{\lambda}\eps /2\right)
-K_{0}\left(\sqrt{\lambda}\reps\eps \right)I_{0}\left(\sqrt{\lambda}\eps /2\right)
}.
 \end{eqnarray*}
The proof relies on precise estimates on $\alpha$ and $\beta$ for various regime of $\lambda$.

It is convenient to introduce the notation  $z=\sqrt{\lambda}\eps $. Let us first consider the case when
$$
z\leq \reps^{-\frac{1}{3}},
$$
We shall use an {\it ad-hoc} bound, easily verifiable using the Frobenius decomposition of $K_0$, see e.g. \cite{NIST-10}. For all $x>0$, we have  
\begin{equation}\label{eq:K0low}
K_{0}(x)=I_{0}(x)\left(- \ln\left(\frac{x}{2}\right)-\gamma +\ln(x+e) \frac{t(x)}{1+t(x)}R_0(x)\right),
\end{equation}
where $t(x)=\frac{x^{2}}{4+x^{5/2}\sqrt{2\pi}}e^{x}$ and $R_0$ satisfies
$ \frac{2}{5}\leq R_0(x)\leq \frac{5}{4} $ for all $x>0$.
We now compute
\begin{eqnarray*}
&& \displaystyle \frac{I_{0}\left(\reps z\right)K_{0}\left(\frac{1}{2}z\right)-K_{0}\left(\reps z\right)I_{0}\left(\frac{1}{2}z\right)} {I_{0}\left(\frac{1}{2}z\right)I_{0}\left(\reps z\right)} \\ 
&=&  \ln \left(2 \reps \right)
 +   \ln\left(\frac{1}{2}z+e\right) \frac{t\left(\frac{1}{2}z\right)}{1+t\left(\frac{1}{2}z\right)}R_0\left(\frac{1}{2}z\right) 
 -   \ln\left(\reps z+ e \right)\frac{t\left(\reps z\right)}{1+t\left(\reps z\right)}R_0\left(\reps z\right).
\end{eqnarray*}
Note that 
\begin{eqnarray*}
\left|\ln(x/2+e) \frac{t(x/2)}{1+t(x/2)}R_0(x/2)\right|&\leq& \left|\max(1,\ln(x))\right| \\
&\leq& \frac{1}{3}|\ln(\reps)| \textrm{ when } x<\frac{1}{\reps^{1/3}}.
\end{eqnarray*}
Therefore we have 
\begin{equation}\label{eq:bddenom1}
I_{0}\left(\reps z\right)K_{0}\left(\frac{1}{2}z\right)-K_{0}\left(\reps z\right)I_{0}\left(\frac{1}{2}z\right)  = I_{0}\left(\frac{1}{2}z\right)I_{0}\left(\reps z\right)\ln(\reps)\left(1 + E_\eps \right),
\end{equation}
with
$$
|E_\eps | \leq \frac{2}{3} \mbox{ for all } z\leq \reps^{-1/3}.
$$
We can now estimate $\alpha$ and $\beta$ as follows
\begin{eqnarray*}
|\alpha| &\leq& \left| \frac{u_{m,n} K_{0}\left(z/2\right)}{I_{0}\left(\reps z\right)K_{0}\left(\frac{1}{2}z\right)-K_{0}\left(\reps z\right)I_{0}\left(\frac{1}{2}z\right)}\right| \\
&+& \left|M \frac{K_{0}\left(\reps z\right)}{I_{0}\left(\reps z\right)K_{0}\left(\frac{1}{2}z\right)-K_{0}\left(\reps z\right)I_{0}\left(\frac{1}{2}z\right)}\right|\\
&\leq& M \frac{\left|- \ln\left(\frac{z}{2}\right)-\gamma +\ln(\frac{z}{2}+e) \frac{t(z/2)}{1+t(z/2)}R_0(z/2)\right|}{I_{0}\left(\reps z\right)|\ln(\reps)|\left(1 + E_\eps \right)} \\
&+& M \frac{\left|- \ln\left(\reps z \right)-\gamma +\ln(\reps z+e) \frac{t(z\reps)}{1+t(z\reps)}R_0(\reps z)\right|}{I_{0}\left(\reps z\right)|\ln(\reps)|\left(1 + E_\eps \right)} \\
&\leq&3\frac{M}{I_{0}\left(\reps z\right)} + 9 \frac{M}{I_{0}\left(\frac{1}{2}z\right)} \mbox{ when } z\leq \reps^{-1/3}.
\end{eqnarray*}
Therefore
$$
|\alpha| \leq C \frac{M}{I_{0}\left(\reps z\right)}\leq C M\,.
$$
Similarly, we obtain
$$
|\beta|\leq \frac{M \left| I_{0}\left( z/2 \right)\right| + M \left| I_{0}\left( \reps z \right)\right| }{\left|I_{0}\left(\reps z\right) K_{0}\left(\frac{1}{2}z\right)
- K_{0}\left(\reps z\right) I_{0}\left(\frac{1}{2}z\right)
\right|} \leq C  \frac{M}{I_{0}\left(\reps z\right)|\ln(\reps)|}\leq C  \frac{M}{|\ln(\reps)|} .
$$
We are now in position to bound $\psi^\prime$, namely
$$
|\psi^\prime(t)| \leq \sqrt{\lambda}|\alpha| I_1\left(\sqrt{\lambda}t\right) +  \sqrt{\lambda}|\beta| K_1\left(\sqrt{\lambda}t\right).
$$
At $t=\eps  \reps$, since $\sqrt{\lambda}\eps  \reps \leq \reps^{2/3}<1$, we deduce that $I_1\left(\sqrt{\lambda}t\right)\leq \reps^{2/3}$. Using the fact that for all $x>0$, $x K_1(x)\leq1$, we obtain
$$
  \leq  \sqrt{\lambda}  C M \left( \reps^{2/3} + \frac{1}{\sqrt{\lambda}\eps  \reps |\ln(\reps)|}\right) \leq C M \frac{1}{\eps  \reps |\ln(\reps)|}.
$$
We have thus shown that
\begin{equation}
|\psi^\prime(\eps  \reps)| \leq  C \theta(\lambda,\eps)\,M \text{ for } z < \reps^{-\frac{1}{3}}.
	\label{eq:UpperNormal1}
\end{equation}

Let us now turn to the case when $z\geq \reps^{-1/3}$. In this case, we easily show that
$$
\alpha \approx M \frac{1}{I_0\left(\frac{1}{2} z \right)} \mbox{ and } \beta  \approx  u_{m,n} \frac{1}{K_0\left( \reps z \right)},
$$
leading to the bounds 
\begin{equation}
|\psi^\prime(\eps  \reps)|  \leq  \sqrt{\lambda}  C M \left(  \frac{I_1\left(\reps z\right)}{I_0\left(\frac{1}{2} z \right) } 
+ \frac{K_1\left(\reps z\right)}{K_0\left(\reps z \right)} \right) .
	\label{eq:Psiprime2}
\end{equation}
For $\reps^{-1/3} < z < \reps^{-1}$, we have
\[
\sqrt{\lambda}\frac{I_1\left(\reps z\right)}{I_0\left(\frac{1}{2} z \right) } \leq  \sqrt{\lambda}\frac{I_1(1)}{I_0\left(\frac{1}{2} \sqrt{\lambda}\eps \right) } \rightarrow 0 \text{ as } \eps \to 0.
\]
On the other hand, it is easy to check that for $0<x<1$,
$$
\frac{K_1\left(x \right)}{K_0\left(x \right)} \leq 2 \frac{1}{x\ln\left(\frac{e}{x}\right)}.
$$
We thus obtain for $z=\reps^{-\alpha}$, $\frac{1}{3} \leq \alpha \leq 1$, that 
\begin{equation}
|\psi^\prime(\eps  \reps)|\leq  C M \left(1 + \frac{1}{\reps \eps} \frac{1}{\ln\left(\frac{e}{\reps^{1-\alpha}}\right)}\right) \leq C M \frac{\eps}{\reps(1-\alpha+\eps^2)} = C\theta(\lambda,\eps)M.
	\label{eq:UpperNormal2}
\end{equation}

For the remaining case $ z \geq \reps^{-1}$, note that 
\begin{align*}
&I_1(\reps\,x) < I_0(\reps x) \leq I_0\left(\frac{1}{2}x\right) \text{ for }x > 0,\\
&K_1(x)\leq 2 K_0(x) \text{ for }x\geq1.
\end{align*}
Inserting these inequalities in \eqref{eq:Psiprime2}, we obtain
\begin{equation}
|\psi^\prime(\eps  \reps)|\leq  C M \sqrt{\lambda} \text{ for } z > \reps^{-1}.
	\label{eq:UpperNormal3}
\end{equation}

From \eqref{eq:Normaltueps-Psi}, \eqref{eq:UpperNormal1}, \eqref{eq:UpperNormal2} and \eqref{eq:UpperNormal3}, we conclude that
\[
\sup_{\partial D_{m,n,\eps}} \frac{\partial \tueps}{\partial n} \leq |\psi'(\eps\,\reps)| \leq  C\theta(\lambda,\eps)\,M.
\]
The lower bound for $\frac{\partial \tueps}{\partial n}$ is obtained by similar arguments. Since $\tueps$ is constant on $\partial D_{m,n,\eps}$, this concludes the proof.

\subsection{\label{sec:lem47} Proof of Lemma~\ref{lem:HolderX}}
We continue to use $M = \sup_{m,n} |u_{m,n}|$. Fix a multi-index $J \in  \{(0,1),(1,0)\}$ and
 split $\nabla^J \tueps(x) = I_1^J(x) + I_2^J(x) + I_3^J(x)$ as in \eqref{RepForm17Oct10-Der}.

We first show that
\begin{equation}
[\nabla\tueps]_{C^\nu(\homegtau)} \leq C\,M \text{ for } \tau > \kappa\,\eps^{(1-\nu)/2}. 
	\label{H-FirstRound}
\end{equation}

Recall that we have shown using \eqref{I1Est} and \eqref{I2Est} with $\tbeta = 0$ that
\[
|\nabla I_1^J(x)| \leq CM \text{ and }
	|\nabla I_2^J(x)| \leq \frac{CM\,\reps}{\tau^{3}\eps^{3}}.
\]

Fix $x$ and $z$ in $\homegtau$ and $\nu \in (0,1)$. Then the above estimates implies that
\[
\frac{|I_1^J(x) - I_1^J(z)| + |I_2^J(x) - I_2^J(z)|}{|x - z|^\nu} \leq C\,M\Big(\frac{\reps}{\tau^3\,\eps^3} + 1\Big).
\]
To establish \eqref{H-FirstRound}, it remains to bound
\begin{eqnarray*}
A(x,z)
	&:=& \frac{|I_3^J(x) - I_3^J(z)|}{|x - z|^\nu}\\
	&=& \sum_{(m,n) \in I_\eps}\int_{\partial D_{m,n,\eps }} \frac{\partial\tueps}{\partial n}(y)\,\frac{[\nabla^J_x\Phi(x - y) - \nabla^J_x\Phi(z - y)]}{|x - z|^\nu} \,d\sigma(y).
\end{eqnarray*}
For $\lambda \geq \frac{1}{\reps}$, thanks to \eqref{I3-CrudeLarge}, we have 
\[
A(x,z) \leq C |\nabla I_3^J| = C|I_3^{J+1}|
	\leq C\,M\,\frac{\reps^{1/2}}{\sqrt{\lambda}\,\eps^{5}}\Big(\frac{\eps}{\tau^{4}} + 1\Big).
\]
which is sufficient provided $\tau >\kappa \eps^2$ for example. We henceforth assume that $\lambda \leq \frac{1}{\reps}$. 
We first bound the derivatives of $I_3^J$. For $\xi \in \homegtau$, we use Lemma \ref{SCBdryGrad} and 
count contributions from ``rings'' of inclusions, distinguishing near and far contributions as before, to get
\begin{align*}
|\nabla I_3^{J}(\xi)|
	&\leq C\sum_{j=0}^{C\,\eps ^{-1}} \theta(\lambda,\eps ) \,M\,(j+1)\,\lambda\left|K_0^{(2)}(\sqrt{\lambda}(j + \tau)\eps )\right|\,\eps \,\reps\\
	&\leq C\,M\,\theta(\lambda,\eps )\,\eps \,\reps\,\lambda\left|K_0^{(2)}(\sqrt{\lambda}\,\tau\,\eps )\right|\\
		&\qquad\qquad + C\,M\,\theta(\lambda,\eps ) \,\eps \,\reps\,\sum_{j=1}^{C\,\eps ^{-1}} M\,j\,\sqrt{\lambda}^{2}\left|K_0^{(2)}(\sqrt{\lambda}\,j\,\eps )\right|\\
	&\leq C\,M\,\theta(\lambda,\eps )\,\eps \,\reps\,\left(\frac{1}{\tau^2\,\eps^2}
			+ \sum_{j=1}^{C\,\eps ^{-1}} \frac{1}{j\,\eps^2}\right)\\
	&\leq C\,M\,\theta(\lambda,\eps )\,\eps\,\reps\left(\frac{1}{\tau^2\,\eps^2} + \frac{|\log\eps|}{\eps^2}\right).
\end{align*}
Noting that $\theta(\lambda,\eps) \leq \frac{\eps}{\reps}$ by \eqref{eq:mulambda} and our restriction on $\lambda$, we thus have
\begin{equation}
|\nabla I_3^{J}(\xi)|
	\leq C\,M\,\Big(\frac{1}{\tau^2} + |\log \eps|\Big).
	\label{I3J2Log}
\end{equation}

From \eqref{I3J2Log} we deduce in particular that
\begin{equation}
A(x,z) \leq C\,M\,\eps^{1 - \nu}\,\Big(\frac{1}{\tau^2} + |\log \eps|\Big) \text{ for } |x - z| < 10\eps.
	\label{AClose}
\end{equation}

We turn to bounding $A$ when $|x - z| > 10\eps$. 
Pick $\tilde x$ and $\tilde z$ in $\homegtau$ such that $|x - \tilde x| < 2\eps$, 
$|z - \tilde z| < 2\eps$, $\dist(\tilde x, \Deps) > \eps/10$, $\dist(\tilde z, \Deps) > \eps/10$ and $|\tilde x - \tilde z|<|x-z|$. Then by \eqref{AClose}, 
\[
A(x,z) \leq A(x,\tilde x) + A(\tilde x, \tilde z) + A(\tilde z,z) \leq A(\tilde x, \tilde z) + C\,M\,\eps^{1 - \nu}\,\Big(\frac{1}{\tau^2} + |\log \eps|\Big).
\]
Thus, provided $\tau>\kappa \,\eps^{\frac{1-\nu}{2}}$, can focus on case when $\dist(x,\Deps) > \eps/10$ and $\dist(z,\Deps) > \eps/10$.

Split $A(x,z) = A_1(x,z) + A_2(x,z)$ where
\begin{align*}
A_1(x,z)
	&= \sum_{(m,n) \in I_{\eps}(x,z)}\int_{\partial D_{m,n,\eps }} \frac{\partial\tueps}{\partial n}(y) \,\frac{[\nabla^J_x\Phi(x - y) - \nabla^J_x\Phi(z - y)]}{|x - z|^\nu} \,d\sigma(y),\\
A_2(x,z)
	&= \sum_{(m,n) \in I_{\eps}(z,x)}\int_{\partial D_{m,n,\eps }} \frac{\partial\tueps}{\partial n}(y)\,\frac{[\nabla^J_x\Phi(x - y) - \nabla^J_x\Phi(z - y)]}{|x - z|^\nu} \,d\sigma(y).
\end{align*}
and
\begin{align*}
I_{\eps}(x,z)
	&= \left\{ (m,n) \in I_\eps: \dist(x,D_{m,n,\eps}) \leq \dist(z,D_{m,n,\eps})
\right\}.
\end{align*}
In the sequel we bound $A_1(x,z)$ independently of $x$ and $z$. Switching the role of $x$ and $z$, we can therefore use the same bound for $A_2(x,z)$.

Using Lemma \ref{SCBdryGrad} and the expression for $\theta(\lambda,\eps)$ for $\lambda \leq \frac{1}{\reps}$, we have
\begin{equation}
|A_1(x,z)| \leq \sum_{(m,n) \in I_{\eps}(x,z)} 
C\,M\,\eps^2\,\sup_{y \in \partial D_{m,n,\eps}}\left|\frac{[\nabla^J_x\Phi(x - y) - \nabla^J_x\Phi(z - y)]}{|x - z|^\nu}\right|.
	\label{eq:A1FirstBound}
\end{equation}

It is convenient to introduce
\[
K_{0}^{(1 + \nu)}(t) := \sup_{s,s' > t} \frac{|K_0^{(1)}(s) - K_0^{(1)}(s')|}{|s - s'|^\nu}.
\]
We will use the following inequality, which is easily proved using the monotonicity properties of $K_0^{(1)}$ and $K_0^{(2)}$, the derivative of $K_0^{(1)}$,
\begin{equation}
K_{0}^{(1 + \nu)}(t) \leq \frac{C}{t^{1 + \nu}} \text{ for } t > 0.
	\label{K1nu}
\end{equation}

To bound $A_1(x,z)$, we proceed as before by counting contribution from inclusions 
located in the rings centred at $x$ with radii $(j + 1/10)\eps$ and $(j + 1 + 1/10)\eps$. 
For example, consider an inclusion $D_{m,n,\eps}$ which is closer to $x$ than to $z$ and lying in the above $j$-th ring. For $y \in \partial D_{m,n,\eps}$, we estimate
\begin{align*}
&\frac{|\nabla^J_x\Phi(x - y) - \nabla^J_x\Phi(z - y)|}{|x - z|^\nu}\\
	&\qquad\qquad = \frac{1}{|x - z|^\nu}\,\left|\sqrt{\lambda}\,K_0^{(1)}(\sqrt{\lambda}\,|x - y|)\,\frac{(x - y) \cdot J}{|x - y|} - \sqrt{\lambda}\,K_0^{(1)}(\sqrt{\lambda}\,|z - y|)\,\frac{(z - y) \cdot J}{|z - y|}\right|\\
	&\qquad\qquad \leq \frac{1}{|x - z|^\nu}\,\sqrt{\lambda}\,|K_0^{(1)}(\sqrt{\lambda}\,|x - y|)|\,\left|\,\frac{(x - y) \cdot J}{|x - y|} - \frac{(z - y) \cdot J}{|z - y|}\right|\\
		&\qquad\qquad\qquad\qquad + \frac{1}{|x - z|^\nu}\left|\sqrt{\lambda}\,K_0^{(1)}(\sqrt{\lambda}\,|x - y|) - \sqrt{\lambda}\,K_0^{(1)}(\sqrt{\lambda}\,|z - y|)\right|\,\left|\frac{(z - y) \cdot J}{|z - y|}\right|\\
	&\qquad\qquad \leq \sqrt{\lambda}\,|K_0^{(1)}(\sqrt{\lambda}\,(j + 1/10)\,\eps)|\,\frac{C}{((j + 1/10)\eps)^\nu}
		+ \sqrt{\lambda}^{1+\nu}\,K_{0}^{(1 + \nu)}(\sqrt{\lambda}\,(j + 1/10)\,\eps)\\
	&\qquad\qquad \leq \frac{C}{((j + 1/10)\,\eps)^{1+\nu}}.
\end{align*}
Inserting this estimate in \eqref{eq:A1FirstBound} and summing over all inclusions we end up with
\[
A_1(x,z) \leq \sum_{j = 0}^{C\,\eps^{-1}} j\,\frac{C\,M\,\eps^{1 - \nu}}{(j + 1/10)^{1 + \nu}}
	\leq C\,M\,\eps^{1 - \nu}\,\sum_{j=1}^{10C\eps^{-1}} \frac{1}{j^{\nu}} \leq C\,M.
\]
Likewise
\[
A_2(x,z) \leq C\,M.
\]
We have thus established \eqref{H-FirstRound}.

To establish the result, we note that, by Lemma \ref{lem:utilde}, we have, for $0 < \tbeta < \frac{1}{2}$ and $\tau > \kappa\,\eps^{\frac{1-2\tbeta}{2(1+2\tbeta)}}$,
\[
\|\nabla\tueps\|_{L^\infty(\homegtau)} \leq \frac{C}{\lambda^\tbeta}\,M.
\]
This implies that
\[
\frac{|\nabla\tueps(x) - \nabla\tueps(z)|}{|x - z|^\nu} \leq \frac{C}{\lambda^\tbeta}\,\frac{M}{|x - z|^\nu}.
\]
On the other hand, for some small $\delta > 0$ to be determined, \eqref{H-FirstRound} implies that for $\tau > \kappa\,\eps^{\frac{\delta}{2}}$ that
\[
\frac{|\nabla\tueps(x) - \nabla\tueps(z)|}{|x - z|^\nu} 
	\leq [\nabla\tueps]_{C^{1-\delta}(\homegtau)}\,|x - z|^{1 - \delta - \nu}
	\leq C\,M\,|x - z|^{1 - \delta - \nu}.
\]
It follows that, provided $0<1-\delta -\nu$, 
\begin{align*}
\frac{|\nabla\tueps(x) - \nabla\tueps(z)|}{|x - z|^\nu} 
	&\leq \left\{\frac{C}{\lambda^\tbeta}\,\frac{M}{|x - z|^\nu}\right\}^{\frac{1-\delta - \nu}{1 - \delta}}\,\left\{C\,M\,|x - z|^{1 - \delta - \nu}\right\}^{\frac{\nu}{1 - \delta}}\\
	&= \frac{C}{\lambda^{\frac{\tbeta(1-\delta - \nu)}{1 - \delta}}}\,M \text{ provided } \tau > \kappa\,\eps^{\min\left(\frac{\delta}{2}, \frac{1 - 2\tbeta}{1 + 2\tbeta}\right)}.
\end{align*}
Picking $\tbeta = \frac{\beta - \frac{\nu}{4}}{1 - \frac{\nu}{2}}$ and $\delta = \frac{1 - 2\tbeta}{1 + 2\tbeta} = \frac{1 - 2\beta}{1 + 2\beta - \nu}$, 
we get the assertion, provided $\nu <2\beta$.

}

\bibliographystyle{spmpsci}
\bibliography{Mybib}

\end{document}